\documentclass[11pt,a4paper]{article}

\usepackage{amsmath}
\usepackage{amsfonts}
\usepackage{amsthm}
\usepackage{amssymb}
\usepackage{url}
\usepackage[usenames,dvips]{color}
\usepackage{stmaryrd}
\usepackage[dvips]{graphicx}
\usepackage{psfrag, float}  

\DeclareMathAlphabet{\mathpzc}{OT1}{pzc}{m}{it}

\newtheorem{theorem}{Theorem}[section]
\newtheorem{lemma}[theorem]{Lemma}
\newtheorem{corollary}[theorem]{Corollary}

\newtheorem{proposition}[theorem]{Proposition}

\newcommand{\be}{\begin{equation}}
\newcommand{\ee}{\end{equation}}

\newcommand{\ga}{\gamma}
\newcommand{\dps}{\displaystyle}
\newcommand{\RR}{\mathbb{R}}
\newcommand{\NN}{\mathbb{N}}
\newcommand{\CC}{\mathbb{C}}

\newcommand{\TT}{\mathbb{T}}
\newcommand{\ZZ}{\mathbb{Z}}

\newcommand{\WW}{\mathcal{W}}
\newcommand{\PP}{\mathcal{P}}
\newcommand{\GG}{\mathcal{G}}
\newcommand{\II}{\mathcal{I}}

\newcommand{\LL}{\mathcal{L}}

\newcommand{\YY}{\mathcal{Y}}
\newcommand{\KK}{\mathcal{K}}
\newcommand{\SSS}{\mathcal{S}}

\newcommand{\XX}{\mathcal{X}}
\newcommand{\OO}{\mathcal{O}}

\newcommand{\FF}{\mathcal{F}}
\newcommand{\HH}{\mathcal{H}}

\newcommand{\RRR}{\mathcal{R}}

\newcommand{\QQQ}{\mathcal{Q}}

\newcommand{\EE}{\mathcal{E}}
\newcommand{\JJ}{\mathcal{J}}
\newcommand{\AAA}{\mathcal{A}}
\newcommand{\ZZZ}{\mathcal{Z}}
\newcommand{\CCC}{\mathcal{C}}
\newcommand{\Id}{\mathrm{Id}}
\newcommand{\lln}{\llfloor}
\newcommand{\rrn}{\rrfloor}

\newcommand{\Oo}{\OO(1)}

\newcommand{\ii}{^{-1}}
\newcommand{\de}{\delta}
\newcommand{\pa}{\partial}
\newcommand{\la}{\lambda}
\newcommand{\al}{\alpha}

\newcommand{\out}{\mathrm{out}}
\newcommand{\kk}{\kappa}
\newcommand{\rr}{\rho}

\newcommand{\ol}{\overline}

\renewcommand{\Re}{\mathrm{Re\, }}
\renewcommand{\Im}{\mathrm{Im\,}}
\newcommand{\wt}{\widetilde}
\newcommand{\wh}{\widehat}

\newcommand{\h}{\mathrm{h}}
\newcommand{\fl}{\mathrm{flow}}

\begin{document}

\title{Oscillatory motions for the restricted planar circular three body problem}
\author{Marcel Guardia\thanks{\tt marcel.guardia@upc.edu, mguardia@umd.edu}, Pau Mart\'\i n\thanks{\tt martin@ma4.upc.edu} \ and  Tere M. Seara\thanks{\tt tere.m-seara@upc.edu}}
\maketitle

\medskip
\begin{center}$^{*}$
Department of Mathematics\\
Mathematics Building, University of Maryland\\
College Park, MD 20742-4015, USA
\end{center}
\smallskip
\begin{center}$^\dagger$
Departament de Matem\`atica Aplicada IV\\
Universitat Polit\`ecnica de Catalunya\\
Campus Nord, Edifici C3, C. Jordi Girona, 1-3. 08034 Barcelona, Spain
\end{center}
\smallskip
\begin{center}$^\ddagger$
Departament de Matem\`atica Aplicada I\\
Universitat Polit\`ecnica de Catalunya\\
Diagonal 647, 08028 Barcelona, Spain
\end{center}

\begin{abstract}
In this paper we consider the circular restricted three body problem
which models the motion of a massless body under the influence of
the Newtonian gravitational force caused by two other bodies, the
primaries, which move along circular planar Keplerian orbits. In a
suitable system of coordinates, this system has two degrees of
freedom and the conserved energy is  usually called  the Jacobi
constant. In 1980, J. Llibre and C. Sim\'o~\cite{LlibreS80} proved
the existence of oscillatory motions for the restricted planar
circular three body problem, that is, of orbits  which leave every
bounded region but which return infinitely often to some fixed
bounded region. To prove their existence they had to assume   the
ratio  between the masses of the two primaries to be exponentially
small with respect to the Jacobi constant. In the present work, we
generalize their result proving the existence of oscillatory motions
for any value of the mass ratio.

To obtain such motions, we  show that, for any value of the mass ratio and for big values of the Jacobi constant, there exist transversal intersections between the stable and unstable manifolds of infinity which guarantee the existence of a symbolic dynamics that  creates the  oscillatory orbits. The main achievement is to rigorously prove the existence of these orbits without assuming the mass ratio small since then this transversality  can not be checked by means of classical perturbation theory respect to the mass ratio. Since our method is valid for all values of mass ratio, we are able to detect a curve in the parameter space, formed by the mass ratio and the Jacobi constant, where cubic homoclinic tangencies between the invariant manifolds of infinity appear.

\end{abstract}

\begin{center}
\Large{Mouvements oscillatoires dans le probl\`eme plan circulaire restreint des trois corps}
\end{center}
\vspace{0.6cm}
\small{\textbf{R\'esum\'e} : Dans cet article, nous \'etudions le probl\`eme restreint des trois corps,
qui mod\'elise le mouvement d'un corps de masse nulle sous
l'influence des forces de gravitation newtonienne cr\'e\'ees par deux
autres corps, appel\'es les primaires, qui eux se d\'eplacent le long
d'orbites k\'epl\'eriennes circulaires. Dans un syst\`eme de coordonn\'ees
convenable, ce syst\`eme poss\`ede deux degr\'es de libert\'e et l'\'energie
conserv\'ee est habituellement appel\'ee la constante de Jacobi. En 1980,
J. Llibre et C. Sim\'o~\cite{LlibreS80} ont d\'emontr\'e l'existence de
mouvements oscillatoires dans le probl\`eme plan restreint des trois
corps, c'est-\`a-dire d'orbites qui sortent de n'importe quelle r\'egion
born\'ee mais qui rentrent une infinit\'e de fois dans une certaine r\'egion
born\'ee fix\'ee. Pour d\'emontrer ce r\'esultat, les auteurs avaient besoin
de supposer que le rapport des masses des deux primaires est
exponentiellement petit par rapport \`a la constante de Jacobi. Dans le
pr\'esent travail, nous g\'en\'eralisons ce th\'eor\`eme \`a toute valeur des
rapports de masses. Pour obtenir de tels mouvements, nous montrons
que, quel que soit le rapport des masses, si la constante de Jacobi
est assez grande, il existe des intersections transverses des vari\'et\'es
stable et instable de l'infini, ce qui garantit l'existence d'une
dynamique symbolique, puis celle de mouvements oscillatoires. Le
principal r\'esultat est de prouver rigoureusement l'existence de ces
orbites sans supposer que le rapport de masses est petit, puisqu'alors
la transversalit\'e ne peut pas \^etre v\'erifi\'ee par les m\'ethodes de la
th\'eorie classique des perturbations relativement au rapport de
masses. Comme notre m\'ethode est valable pour toutes les valeurs des
rapports de masses, nous parvenons \`a d\'etecter une courbe, dans
l'espace des param\'etres, c'est-\`a-dire dans l'espace des rapports de
masses et de la constante de Jacobi, sur laquelle apparaissent des
tangences homoclines cubiques entre les vari\'et\'es invariantes de
l'infini. }

\tableofcontents

\section{Introduction}\label{sec:intro}
The  restricted   three body problem models the motion of three
bodies,  one of them massless, under the Newtonian gravitational
force. Since one of the bodies is massless, it does not cause any
influence on the other two, \emph{the primaries}. Thus, the two
primaries are governed by the classical Keplerian two body problem.
Let us assume that these two bodies, whose motion lies on a plane,
perform circular orbits and that the massless body moves in the same
plane. This problem is known as the restricted planar circular three
body problem (RPC3BP from now on). If one normalizes the total mass
of the system to be one, the RPC3BP depends on one parameter $\mu$
which measures the quotient between the masses of the two primaries
and therefore satisfies $\mu\in [0,1/2]$. Then, taking the
appropriate units, the RPC3BP  is Hamiltonian with respect to
\begin{equation}\label{def:Ham3BP:Cartesian}
H(q,p,t;\mu)= \frac{\|p\|^2}{2} - \frac{1-\mu}{\left\|q+\mu
      q_0(t)\right\|}-\frac{\mu}{\left\|q-(1-\mu)q_0(t)\right\|}
\end{equation}
where $q,p\in\RR^2$ and $-\mu q_0(t)$ and $(1-\mu)q_0(t)$,
$q_0(t)=(\cos t,\sin t)$, are the positions of the primaries. This
Hamiltonian has two and a half degrees of freedom. It has a
conserved quantity, usually called Jacobi constant, defined as
\begin{equation}\label{def:JacobiConstant:cartesian}
 \JJ(q,p,t;\mu)=H(q,p,t;\mu)-(q_1p_2-q_2p_1).
\end{equation}


The purpose of this paper is to study the existence of oscillatory
orbits for the RPC3BP, that is, orbits such that
\[
 \limsup_{t\rightarrow \pm\infty}\|q\|=+\infty\,\,\,\text{ and }\,\,\, \liminf_{t\rightarrow \pm\infty}\|q\|<+\infty.
\]
When $\mu=0$, the motion of the massless body is only influenced by
one of the primaries and therefore it satisfies Kepler's laws. In
particular, oscillatory motions cannot exist. In this paper we show
that oscillatory orbits do exist for any value of the mass ratio
$\mu\in (0,1/2]$.

The study of oscillatory motions was started by Chazy (see
\cite{ArnoldKN88}). In 1922, he gave a complete classification of all
possible states that a three body problem can approach as time tends
to infinity (see Section~2.4 of~\cite{ArnoldKN88}). For the
restricted three body problem (either planar or spatial, circular or
elliptic) the possible final states are reduced to four:
\begin{itemize}
 \item $H^\pm$ (hyperbolic):  $\|q(t)\|\rightarrow\infty$ and $\|\dot q(t)\|\rightarrow c>0$ as $t\rightarrow\pm\infty$.
\item $P^\pm$ (parabolic): $\|q(t)\|\rightarrow\infty$ and $\|\dot q(t)\|\rightarrow 0$ as $t\rightarrow\pm\infty$.
\item $B^\pm$ (bounded): $\limsup_{t\rightarrow \pm\infty}\|q\|<+\infty$.
\item $OS^\pm$ (oscillatory): $\limsup_{t\rightarrow \pm\infty}\|q\|=+\infty$ and  $\liminf_{t\rightarrow \pm\infty}\|q\|<+\infty$.
\end{itemize}
Examples of all these types of motion, except the oscillatory ones, were
already known by Chazy. The first to prove the existence of
oscillatory motions was Sitnikov in~\cite{Sitnikov60}. In his paper
he considered the restricted spatial three body problem with mass
ratio $\mu=1/2$ and the three bodies disposed in a certain symmetric
configuration,  called now \emph{the Sitnikov example}.
Later, Moser~\cite{Moser01} gave a new proof. His approach  to prove
the existence of such motions was to consider the invariant
manifolds of infinity and to prove that they intersect transversally.
Then, he established the existence of symbolic dynamics close to these invariant
manifolds which lead to the existence of oscillatory motions.

In the planar setting, the first result was by Sim\'o and
Llibre~\cite{SimoL80}. Following the same approach as
in~\cite{Moser01}, they proved the existence of oscillatory motions
for the RPC3BP for small enough values of~$\mu$. One of the main
ingredients of their proof, as in~\cite{Moser01}, was the study of
the transversality of the intersection of the invariant manifolds of
infinity. For~$\mu=0$, the system is integrable and the stable and
unstable invariant manifolds of infinity coincide. Then, they
applied the classical Poincar\'e--Melnikov Theory
\cite{Poincare90,Melnikov63} to ensure that, for~$\mu$ small enough,
the invariant manifolds do not coincide anymore and intersect
transversally. Nevertheless, to be able to compute the Melnikov
function, they considered the Jacobi constant~$\JJ$
(see~\eqref{def:JacobiConstant:cartesian}) large enough. For these
values of the Jacobi constant, the Melnikov function is
exponentially small with respect to~$\JJ$ and therefore they were
only able to prove the transversality of the invariant manifolds
provided~$\mu$ was exponentially small with respect to the Jacobi
constant. This allowed them to prove the existence of oscillatory
motions for small enough~$\mu$. Note that the orbits that they
obtained have large Jacobi constant, which implies that they are far
from the primaries and, therefore, they are far from collisions.
Their result was extended by Xia~\cite{Xia92} using the
real-analyticity of the invariant manifolds of infinity. He claimed
that these invariant manifolds intersect transversally for any mass
ratio $\mu\in (0,1/2]$ except for a finite number of values, and
thus obtaining the existence of oscillatory motions for any mass
ratio except for these values (see also~\cite{Moser01} for a similar
argument for the Sitnikov problem).

Following the same approach, in~\cite{MartinezP94} some formal
computations were performed to show that, for the restricted planar
\emph{elliptic} three body problem, the invariant manifolds of
infinity intersect transversally for arbitrarily small mass ratio. A
rigorous computation of the Melnikov function is done
in~\cite{DelshamsKRS12}, assuming the eccentricity small enough.
The existence of oscillatory
motions has also been proven for the (non necessarily restricted)
collinear three body problem~\cite{LlibreS80}. All the mentioned
works follow the approach initially developed in~\cite{Moser01},
that is, they relate the oscillatory motions to transversal
homoclinic points to infinity and symbolic dynamics. A completely
different approach using Aubry-Mather theory and semi-infinite
regions of instability has been recently developed
in~\cite{GalanteK11, GalanteK10b, GalanteK10c}. This new approach
has allowed the authors to prove the existence of orbits which
initially are in the range of our Solar System and become
oscillatory as time tends to infinity for the RPC3BP with a
realistic mass ratio for the Sun-Jupiter pair.

The mentioned works deal with the problem of the existence of
oscillatory motions in different models of Celestial Mechanics. Once
the existence is known, the  natural question is to measure how
abundant they are. The only result in this direction is the recent
paper~\cite{GorodetskiK12},  in which the authors study the
Hausdorff dimension of the set of oscillatory motions for the
Sitnikov example and the RPC3BP. Using~\cite{Moser01} and
\cite{SimoL80}, they prove that for both problems and a Baire
generic subset of an open set of parameters (the eccentricity of the
primaries in the Sitnikov example and the mass ratio and the Jacobi
constant in the RPC3BP) the Hausdorff dimension is maximal. As pointed
out to us by V. Kaloshin and A. Gorodetski, the present paper and
the techniques developed by them in~\cite{GorodetskiK12} lead to the
prove of the existence of a set of maximal Hausdorff dimension of
oscillatory motions for any value of the mass ratio and a Baire
generic subset of an open set of Jacobi constants.


The purpose of this paper is to improve  the results
of~\cite{SimoL80} and~\cite{Xia92}. We prove, using $\JJ^{-1}$ as a
perturbative parameter,  that the transversality of the invariant
manifolds of infinity holds \emph{for any value of} $\mu\in (0,1/2]$
and the Jacobi constant $\JJ$ big enough. Moreover, we find a curve
in the parameter plane $(\mu,\JJ)$ where the stable and unstable
invariant manifolds of infinity undergo a cubic homoclinic tangency.
Note that in this setting, classical perturbative techniques such as
Poincar\'e--Melnikov Theory do not apply because the difference
between the stable and unstable invariant manifolds is exponentially
small with respect to $\JJ$. That is, we have to face what is
usually called the exponentially small splitting of separatrices
phenomenon. This phenomenon was discovered by
Poincar\'e~\cite{Poincare90} while studying the non-integrability of
the $n$-body problem and he called it the \emph{Fundamental problem
of mechanics}. It has drawn considerable attention in the past
decades but, due to its difficulty, it has essentially only been
considered in toy models~\cite{HolmesMS88,DelshamsS92, Treshev97,
Gelfreich00, GuardiaOS10} or in general systems under hypothesis
that typically do not apply when one wants to study problems from
Celestial Mechanics \cite{DelshamsS97,Gelfreich97a, BaldomaF04,
BaldomaFGS11, Guardia12}. For instance, most of the results deal
with Hamiltonian functions which are essentially either polynomials
or trigonometric polynomials with respect to the state variables.
The present paper is the first one which proves  exponentially small
splitting of separatrices in Celestial Mechanics and also the first
one which deals with an irrational Hamiltonian without assuming
artificial smallness conditions on the parameters of the problem.

The main result of the present paper is the following.
\begin{theorem}[Main Theorem: version 1]\label{th:main:1}
Fix any $\mu\in (0,1/2]$. Then, there exists an orbit $(q(t), p(t))$
of~\eqref{def:Ham3BP:Cartesian} which is oscillatory. Namely, it
satisfies
\[
 \limsup_{t\rightarrow \pm\infty}\|q\|=+\infty\,\,\,\text{ and }\,\,\, \liminf_{t\rightarrow \pm\infty}\|q\|<+\infty.
\]
\end{theorem}
As we have explained the RPC3BP has as a first integral the Jacobi
constant~\eqref{def:JacobiConstant:cartesian} and we take it big
enough. Thus, in fact we have a more precise knowledge of the
oscillatory motions we obtain.
\begin{theorem}[Main Theorem: version 2]\label{th:main:2}
Fix any $\mu\in (0,1/2]$. Then, there exists $J_0>0$ big enough,
such that for any $J>J_0$ there exists an orbit $(q_J(t), p_J(t))$
of~\eqref{def:Ham3BP:Cartesian} in the hypersurface $\JJ(q,p,t;\mu)=J$
which is oscillatory. Namely, it satisfies
\[
 \limsup_{t\rightarrow \pm\infty}\|q_J\|=+\infty\,\,\,\text{ and }\,\,\, \liminf_{t\rightarrow \pm\infty}\|q_J\|<+\infty.
\]
Moreover, if we take $\mu^\ast\in (0,1/2)$, there exists $J^\ast>0$ big enough
such that, for any $J>J^\ast$ and $\mu\in (0,\mu^\ast]$, the previous claims hold.
\end{theorem}
The main difficulty to prove these theorems is to show that the invariant manifolds of infinity
intersect transversally. To state the corresponding result, we first need
 to set up some notation and formally define what  the
invariant manifolds of infinity are. First, writing the system in
rotating coordinates it turns out to be a two degrees of freedom
Hamiltonian system. At each energy level, infinity corresponds to a
periodic orbit with two dimensional stable and unstable manifolds.
Their intersection is studied in
Section~\ref{sec:InvManifoldsOfInfinity}. First,
Theorem~\ref{th:MainDistance}  gives an asymptotic formula for their
difference. Then, in Theorem~\ref{th:SplittingInvManifoldsInfty} we
see that, for any $\mu\in (0,1/2]$ and $\JJ$ big enough, they
intersect transversally. Theorem~\ref{th:SplittingInvManifoldsInfty}
also gives a formula for the area of the lobes between two
consecutive homoclinic points in an associated Poincar\'e map. Then,
Theorems~\ref{th:main:1} and~\ref{th:main:2} follow  from Theorem
\ref{th:SplittingInvManifoldsInfty} using the reasoning
in~\cite{SimoL80}. Furthermore, we are able to detect a bifurcation
curve in the parameter space $(\mu,J)$ where the transversality of
the invariant manifolds is lost since a cubic homoclinic tangency
appears. This fact is stated in Theorem~\ref{th:tangencies}.

 The rest of the paper, that is
Sections~\ref{sec:HJ}--\ref{sec:DiffManifolds}, is devoted to prove
Theorems~\ref{th:MainDistance}, \ref{th:SplittingInvManifoldsInfty} and \ref{th:tangencies}. In
Section~\ref{sec:HJ}, following~\cite{LochakMS03, Sauzin01}, we look
for parameterizations of the invariant manifolds as graphs of
generating functions in a suitable domain using the Hamilton-Jacobi
equation and we state Theorem \ref{th:SplittingViaGeneratingFunctions}, which gives
the difference of these generating functions. Then,  Theorems~\ref{th:MainDistance} and \ref{th:SplittingInvManifoldsInfty} follow easily from Theorem \ref{th:SplittingViaGeneratingFunctions}. The prove of Theorem \ref{th:tangencies} is deferred to the end of Section \ref{sec:DiffManifolds}.

Theorem \ref{th:SplittingViaGeneratingFunctions} is proved in the subsequent
sections. First, in Section~\ref{sec:Separatrix}, we state certain
analytic properties of the unperturbed separatrix which are crucial
to obtain solutions of the Hamilton-Jacobi equation. Then, in
Section~\ref{sec:Manifolds}, we prove the existence of ``solutions''
of the Hamilton-Jacobi equation, which are periodic in one of their
variables, in certain complex domains. The quotation marks  refer to the fact that we do not obtain actual solutions of
the Hamilton-Jacobi equation. Indeed, to proof exponentially small
splitting of separatrices, usually one has to obtain
parameterizations of the invariant manifolds in complex domains
which reach a neighborhood of the singularities of the unperturbed
separatrix (see for instance~\cite{BaldomaFGS11}). This \emph{is not
possible} in the present problem since the parameterizations blow up
before reaching these neighborhoods of the singularities. Instead,
we consider the Fourier series of the generating functions (see
Section~\ref{sec:Manifolds} for more details) and we show that, even
if the parameterizations blow up, their Fourier coefficients are
well defined in the corresponding complex domains. This turns out to
be sufficient to study the exponentially small splitting of
separatrices. Finally, in Section~\ref{sec:DiffManifolds}, we
complete the proof of Theorem
\ref{th:SplittingViaGeneratingFunctions}. Analyzing the difference
between two solutions (as formal Fourier series) of the
Hamilton-Jacobi equation, which correspond to the stable and unstable manifolds, in the complex domains, we deduce
exponentially small estimates for the difference between the
generating functions for real values of the variables. At the end of this section we also prove Theorem \ref{th:tangencies}.

\section{The invariant manifolds of infinity}\label{sec:InvManifoldsOfInfinity}
To study the invariant manifolds of infinity, it is more convenient
to express the Hamiltonian~\eqref{def:Ham3BP:Cartesian} in polar
coordinates. It is given by
\begin{equation}\label{def:HamCircularPolars}
\begin{split}
H(r,\al,y,G,t;\mu)=&\frac{y^2}{2} +\frac{G^2}{2r^2}-\wt U(r,\al-t;\mu)\\
=&\frac{y^2}{2} +\frac{G^2}{2r^2}-\frac{1}{r}-U(r,\al-t;\mu),
\end{split}
\end{equation}
where $(r,\al)$ are the polar coordinates of the configuration space and $(y,G)$ are the symplectic conjugate variables. That is, $y$ is the radial velocity (or momentum) and $G$ is the angular momentum. The function $\wt U$ is the Newtonian potential, which is given by
\[
\wt U(r,\phi;\mu)=\frac{1-\mu}{\left(r^2-2\mu r\cos \phi+\mu^2\right)^{1/2}}+\frac{\mu}{\left(r^2+2(1-\mu)r\cos\phi+(1-\mu)^2\right)^{1/2}}
\]
and therefore $U$ satisfies $U=\OO(\mu)$. Nevertheless, recall that we are considering any $\mu\in(0,1/2]$ and therefore $\mu$ is non necessarily small.

The associated equations are
\begin{equation}\label{eq:Equations:CircularPolar}
\begin{split}
\dot r&=y\\
\dot y&=\frac{G^2}{r^3}-\frac{1}{r^2}+ \pa_r  U(r,\al-t;\mu)\\
\dot \al&=\frac{G}{r^2}\\
\dot G&= \pa_\al  U(r,\al-t;\mu).
\end{split}
\end{equation}
Call
$\Phi_{t,t_0}=(\Phi^r_{t,t_0},\Phi^y_{t,t_0},\Phi^\al_{t,t_0},\Phi^G_{t,t_0})$
to the flow associated to this equation. Then, the stable and
unstable manifolds of infinity are defined as
\begin{equation}\label{def:InvManifoldsNonRot}
\begin{split}
 \WW^s_\infty&=\left\{(r,y,\al,G)\in \RR^2\times\TT\times\RR: \lim_{t\rightarrow+\infty}\Phi^r_{t,t_0}(r,y,\al,G;\mu)=\infty, \lim_{t\rightarrow+\infty}\Phi^y_{t,t_0}(r,y,\al,G;\mu)=0\right\}\\
 \WW^u_\infty&=\left\{(r,y,\al,G)\in \RR^2\times\TT\times\RR: \lim_{t\rightarrow-\infty}\Phi^r_{t,t_0}(r,y,\al,G;\mu)=\infty, \lim_{t\rightarrow-\infty}\Phi^y_{t,t_0}(r,y,\al,G;\mu)=0\right\},
\end{split}
\end{equation}
where $\TT=\RR/(2\pi\ZZ)$. It is known that these invariant manifolds are analytic
(see~\cite{McGehee73}).

It is well known that the RPC3BP possesses a symmetry which, in
polar coordinates, means that the system
\eqref{eq:Equations:CircularPolar} is reversible with respect to the
involution
\begin{equation}\label{def:involution}
\RRR(r,y,\al,G)= (r,-y,-\al,G).
\end{equation}
We will use this fact to obtain symmetric properties of the
parameterizations of the invariant manifolds.

When $\mu=0$, the RPC3BP is reduced to  a central force equation.
This system is autonomous and therefore  has conservation of energy
$H$. Moreover, the angular momentum $G$ is also conserved  and then
the system is integrable. This implies that the invariant manifolds
of infinity coincide  along a \emph{homoclinic manifold},
$\WW^s_\infty= \WW^u_\infty$. This manifold is formed by  a family
of  homoclinic orbits to infinity which perform Keplerian parabolic
orbits. Recall that by ``infinity'' we mean $(r,y)=(+\infty,0)$ and
then, ``infinity'' is foliated by  periodic orbits which can be
parameterized by  the angular momentum $G$ and are of the form
\[
\Lambda_{G_0}=\left\{(r,\al,y,G):r=\infty, y=0,\al\in\TT,G=G_0\right\}.
\]
All the associated homoclinic orbits belong to the  energy level
$H=0$ and can be parameterized by the angular momentum and their
initial condition in the angular variable. We denote these orbits as
\begin{equation}\label{def:UnperturbedHomoclinic}
z_\h(u;G_0,\al_0)=(r_\h(u;G_0),\al_0+\al_\h(u;G_0), y_\h(u;G_0),G_0),
\end{equation}
and we fix the origin of time such that $y_\h(0;G_0)=0$ and $\al_\h(0;G_0)=0$  (see
Figure~\ref{fig:separatrix}), which makes the homoclinic orbit with $\al_0=0$ symmetric, that is,
\[
\RRR( z_\h(u;G_0,0))=z_\h(-u;G_0,0).
\]

Commonly, in the study of this problem,
one considers McGehee coordinates~\cite{McGehee73} $r=2x^{-2}$ which
send infinity to zero. In these new coordinates the system is still
Hamiltonian with a non canonical symplectic form and the origin
becomes a parabolic periodic orbit with stable and unstable
invariant manifolds. Since the symplectic form  is non canonical in
McGehee coordinates, their use  is rather cumbersome.
In the present work we prefer to stick to the original variables, in
which the invariant manifolds~$\WW^{u,s}_\infty$ can be
characterized as above.

\begin{figure}[H]
\begin{center}
\includegraphics[height=4cm]{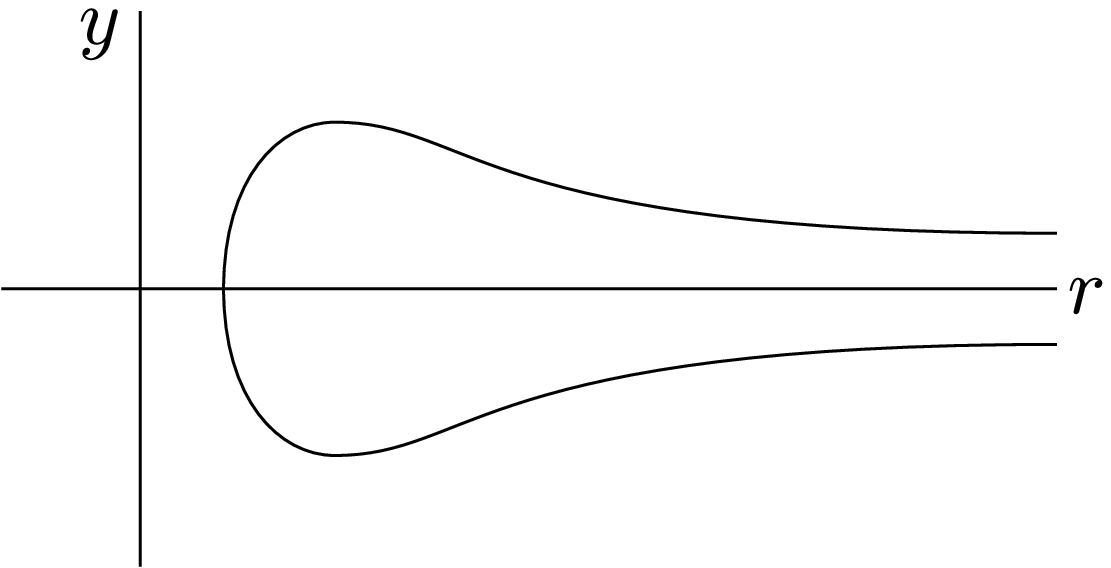}
\end{center}
\caption{Projection on the $(r,y)$ plane of the
separatrix~\eqref{def:UnperturbedHomoclinic} of
system~\eqref{eq:Equations:CircularPolar} with $\mu=0$. It also
correspond to the separatrix of the Poincar\'e map
$\PP_{G_0,\phi_0}$ in
\eqref{def:PoincareMap}.}\label{fig:separatrix}
\end{figure}

Before studying these invariant manifolds, we perform some
rescalings that make the perturbative character more apparent for
big enough angular momentum $G_0$. In these new variables, the
homoclinic~\eqref{def:UnperturbedHomoclinic} will be independent of
$G_0$, which will facilitate the exposition of the results.

\subsection{The RPC3BP as a nearly integrable Hamiltonian System with two time scales}\label{sec:SystemAsFastOscillating}
It is well known that the existence of exponentially small phenomena
usually arises when the system possesses two different time scales.
More concretely, the exponentially small splitting of separatrices
often appears when the system has combined elliptic and hyperbolic
(or parabolic) behavior in such a way that the elliptic one is much
faster than the hyperbolic (or parabolic) one. Even if this is not
obvious looking at Hamiltonian~\eqref{def:HamCircularPolars},
this is also the case in this problem. To make this fact more
apparent we perform a change of variables which is simply a
rescaling. It turns out that the main theorems of this section,
Theorems~\ref{th:MainDistance},~\ref{th:SplittingInvManifoldsInfty} and~\ref{th:tangencies}, are significantly
simpler  in the rescaled variables.

Before showing the rescalings, let us give some heuristic ideas why
this system has two time scales.  Taking arbitrarily high Jacobi
constant implies that the massless body is extremely far from the
primaries, that is, in a neighborhood of infinity in the
configuration space. Moreover, recall that the oscillatory motions
lie close to the invariant manifolds of infinity which are formed by
orbits whose speed tends to zero as its position tends to infinity.
These two facts together imply that the massless body has an
extremely slow motion compared with the primaries. Therefore, after
a suitable rescaling one has a massless body which has a speed of
order one, which is perturbed by the motion of the primaries that
now are rapidly rotating.

Let us recall that, even for
system~\eqref{eq:Equations:CircularPolar} neither $H$ nor $G$ are
preserved as happened for the case $\mu=0$, there is still a
conserved quantity: the Jacobi constant. In polar coordinates, it is
given by
\begin{equation}\label{def:Jacobiconstant}
\JJ(r,\alpha,y,G,t;\mu)=H(r,\alpha,y,G,t;\mu)-G.
\end{equation}
We first fix it to be $\JJ=G_0$. Then, in this invariant
hypersurface we perform the following changes of variables
\begin{equation}\label{def:rescaling}
 r=G_0^2\wt r, \quad y=G_0^{-1}\wt y , \quad \al = \wt \al \quad \text{and}\quad G= G_0 \wt G
\end{equation}
and we rescale time as
\begin{equation}\label{def:RescalingHomo}
t=G_0^3 s,
\end{equation}
obtaining the system
\[
\begin{split}
\frac{d}{ds} \wt r &=\wt y\\
\frac{d}{ds}\wt y &=\frac{\wt G^2}{\wt r^3}-\frac{1}{\wt r^2}+ G_0^4 \pa_ r  U (G_0^2 \wt r,\wt \al-G_0^3 s;\mu)\\
\frac{d}{ds}\wt \al & =\frac{\wt G}{\wt r^2}\\
\frac{d}{ds}\wt G & = G_0^2\pa_\al  U(G_0^2 \wt r,\wt \al-G_0^3 s;\mu).
\end{split}
\]
Now  the two time scales become clear. Indeed, in these variables we
have that $d{\wt y}/ds \sim d{\wt r}/ds\sim 1$, which are the
variables that will define the separatrix, whereas the perturbation
dependence on time is fast. Calling
\[
V (\wt  r, \phi ; \mu, G_0) = G_0^2 U(G_0^2 \wt  r, \phi;\mu),
\]
we have that
\begin{equation}\label{def:PerturbedPotentialscaled}
 V(\wt r,\phi;\mu,G_0)=\frac{1-\mu}{\left(\wt r^2-2(\frac{\mu }{G_0^2})
 \wt r\cos \phi+(\frac{\mu }{G_0^2})^2\right)^{1/2}}
 +\frac{\mu}{\left(\wt r^2+2(\frac{1-\mu }{G_0^2})\wt r\cos\phi+(\frac{1-\mu }{G_0^2})^2\right)^{1/2}}
 -\frac{1} {\wt r},
\end{equation}
and we obtain the system we are going to study
\begin{equation}\label{eq:Equations:Circularescalat}
\begin{split}
\frac{d}{ds} \wt r &=\wt y\\
\frac{d}{ds} \wt y &=\frac{\wt G^2}{\wt r^3}-\frac{1}{\wt r^2}+ \pa_ {\wt r}  V(\wt r,\wt \al-G_0^3 s;\mu, G_0)\\
\frac{d}{ds} \wt \al &=\frac{\wt G}{\wt r^2}\\
\frac{d}{ds}\wt G&=\pa_\al  V(\wt r,\wt \al-G_0^3 s;\mu,G_0).
\end{split}
\end{equation}
Note that $V\sim \mu G_0^{-2}$ and thus, since we are taking $G_0\gg
1$, we are dealing with a \emph{fast oscillating small}
perturbation. When one studies the splitting of separatrices
phenomena in the resonances of nearly integrable Hamiltonian systems
typically the perturbation has the same size as the integrable
unperturbed system (see~\cite{Treshev97, BaldomaFGS11, Guardia12}).
This is not the case in the present problem. This fact will
facilitate the study of the difference between the invariant
manifolds since we will not need to use \emph{inner equations} and
\emph{complex matching techniques} or \emph{continuous averaging}
as were used in those papers.

The rescaling~\eqref{def:rescaling} is conformally symplectic and
therefore the new system~\eqref{eq:Equations:Circularescalat}  is
Hamiltonian with respect to
\begin{equation}\label{def:HamNonRotRescaled}
\wt H(\wt r,\wt \al,\wt y,\wt G,s;\mu,G_0)= \frac{\wt y^2}{2} +\frac{\wt
G^2}{2\wt r^2}-\frac{1}{\wt r}-V(\wt r,\wt \al-G_0^3 s;\mu,G_0).
\end{equation}
The Jacobi constant is now $\JJ= G_0^{-2}\wt H-G_0 \wt G$ and the periodic orbit at infinity is given by $(\wt r,\wt \al, \wt y, \wt G)=(\infty, \wt \al, 0, 1)$, which belongs to  the surface of Jacobi constant
\[
\JJ\left(G_0^2\wt r,\wt \alpha,G_0\ii\wt y,\wt G_0G,G_0^3s;\mu\right)=-G_0.
\]
One of the main advantages of this new set of coordinates is that
the parameterization of the separatrix of the unperturbed
system~\eqref{def:UnperturbedHomoclinic} in the rescaled variables,
\begin{equation}
\label{def:unperturbedhomoclinic}
(\wt r,\wt  \al, \wt y,\wt G)=(\wt r_\h(s), \wt \al_\h (s), \wt
y_\h(s), \wt G_\h(s)),
\end{equation}
is independent of $G_0$ (and also of $\mu$). In particular,  $\wt G_\h (s) \equiv 1$. Moreover, note that after the rescaling we still have
\begin{equation}\label{def:CondicioInicialHomo}
\wt y_\h(0)=0\,\,\,\text{ and }\,\,\,\wt \al_\h(0)=0,
\end{equation}
and therefore
\begin{equation}\label{def:SimetriaHomo}
\wt r_\h(s)=\wt r_\h(-s),\,\,\wt y_\h(s)=-\wt y_\h(-s)\,\,\text{ and }\,\,\wt \al_\h(s)=-\wt \al_\h(-s).
\end{equation}


\subsection{Main result: intersection of the invariant manifolds}
Let the invariant manifolds~$\wt \WW^s_\infty$ and~$\wt
\WW^u_\infty$ be the rescaling of  the invariant
manifolds~$\WW^s_\infty$ and~$\WW^u_\infty$
in~\eqref{def:InvManifoldsNonRot}.  Their coincidence along the
homoclinic manifold~\eqref{def:unperturbedhomoclinic} for $\mu=0$
(equivalently, $G_0=\infty$) is due to the integrability of
system~\eqref{eq:Equations:Circularescalat} for this value of $\mu$.
When $\mu>0$, these manifolds do not longer coincide but they
intersect. To study the transversality of this intersection is the
main goal of this section. Before stating
Theorems~\ref{th:MainDistance},~\ref{th:SplittingInvManifoldsInfty}
and~\ref{th:tangencies}, we first need  to explain how we measure
the transversality of  the invariant manifolds.

Note that the dependence on $s$
in~\eqref{def:HamNonRotRescaled} is only through $\wt\al-G_0^{-3}s$. Thus, one
can eliminate the time dependence in the Hamiltonian by defining a
new angle $\phi=\wt\al-G_0^{-3}s$. Then, we obtain the new Hamiltonian
\begin{equation}\label{def:HamCircularRotating}
\HH(\wt r,\phi,\wt y,\wt G;\mu,G_0)=\frac{\wt y^2}{2}-G_0^3 \wt G +
\frac{\wt G^2}{2\wt r^2}-\frac{1}{\wt r}- V(\wt r,\phi;\mu,G_0),
\end{equation}
which is an autonomous two degrees of freedom Hamiltonian and
therefore has conserved energy $\HH$. The energy conservation
corresponds to the conservation of the Jacobi constant in the
original coordinates. In fact, $\HH(\wt r,\wt \al-G_0^3s,\wt y,\wt
G;\mu,G_0)=G_0^{2}\JJ(G_0^2\wt r,\wt \alpha,G_0\ii\wt y,G_0 \wt
G,G_0^3s;\mu)$. This new set of coordinates is usually called
(rescaled) rotating polar coordinates since they are set in a
rotating frame with respect to the primaries. Namely, in these new
coordinates the primaries remain fixed at the horizontal axis while
the massless body revolves around them, that is, these coordinates
are the polar version of the synodic ones.

In these new coordinates, the invariant manifolds $\wt \WW^s_\infty$
and $\wt \WW^u_\infty$ are three dimensional. If one fixes a level
of energy $\HH(\wt r,\phi,\wt  y,\wt  G;\mu,G_0)=-G_0^3$, the
corresponding invariant manifolds
\[
\begin{split}
\WW^u_{\infty,G_0}&=\wt\WW^u_{\infty}\cap\left\{\HH(\wt r,\phi,\wt  y,\wt  G;\mu,G_0)=-G_0^3\right\}\\
\WW^s_{\infty,G_0}&=\wt\WW^s_{\infty}\cap\left\{\HH(\wt r,\phi,\wt  y,\wt  G;\mu,G_0)=-G_0^3\right\}
\end{split}
\]
are two dimensional.  Moreover,  if one takes the energy high
enough, in a neighborhood of these invariant manifolds one has that
$\dot\phi\neq 0$. This implies that the flow associated to the
Hamiltonian~\eqref{def:HamCircularRotating} restricted to a level of
energy $\HH(\wt r,\phi,\wt  y,\wt  G;\mu,G_0)=-G_0^3$ induces a
Poincar\'e map
\begin{equation}\label{def:PoincareMap}
\begin{split}
\PP_{G_0,\phi_0}:&\{\phi=\phi_0\}\longrightarrow \{\phi=\phi_0+2\pi\}\\
&(r,y)\quad\,\,\,\,\,\mapsto\,\,\,\PP_{G_0,\phi_0}(r,y).
\end{split}
\end{equation}
This Poincar\'e map is two dimensional and is area preserving since
the flow is Hamiltonian (more precisely the  Poincar\'e map
preserves the symplectic form $\Omega=dr\wedge dy$). Now  the
invariant manifolds become  invariant curves  $\ga^{u,s}$ (see
Figure~\ref{fig:splitting}). We focus on the part of these invariant
manifolds with $y>0$ and we consider a  parameterization of
$\ga^{u,s}$ of the form
\begin{equation}\label{def:ParamInvManOriginal}
 \begin{split}
\wt r&=\wt r_\h(v)\\
\wt y&=Y_{\phi_0}^{u,s}(v;\mu,G_0),
 \end{split}
\end{equation}
where $\wt r_\h(v)$ is the $\wt r$ component of the separatrix
parameterization~\eqref{def:unperturbedhomoclinic}. Then, following
\cite{Sauzin01}, to measure the distance between the invariant
manifolds along a section $\wt r=\mathrm{const}$, it suffices to
measure the difference between  the functions $Y_{\phi_0}^{u,s}$.
Note that the used parameterization is equivalent to write the
curves as graphs $\wt y=y^{u,s}(r)$. Nevertheless, we use the
auxiliary parameter $v$ since it simplifies the formulas in
Theorem~\ref{th:MainDistance}. We also  give, in Theorem \ref{th:SplittingInvManifoldsInfty}, a measure
of the area of the lobes that are formed between the invariant
curves since it is a symplectic invariant (see
Figure~\ref{fig:splitting}). Proceeding analogously, one could
easily give an asymptotic formula for the angle between the curves
at homoclinic points. Indeed, in Theorem \ref{th:tangencies}, we show that if one take values of the  parameters in certain curve $\eta$ in the plane $(\mu,G_0)$,  the invariant curves $\ga^{u,s}$ in \eqref{def:ParamInvManOriginal}, besides a transversal intersection,  have a (cubic) homoclinic  tangency.


\begin{figure}[H]
\begin{center}
\includegraphics[height=4cm]{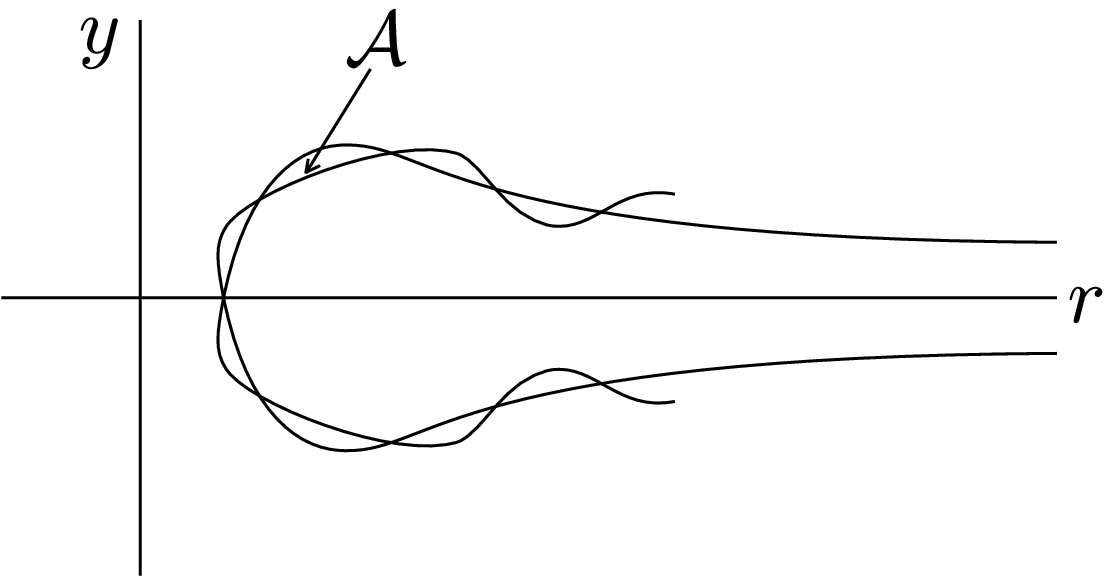}
\end{center}
\caption{Stable and unstable invariant manifolds of infinity for the
Poincar\'e map $\PP_{G_0,\phi_0}$ in~\eqref{def:PoincareMap}.}\label{fig:splitting}
\end{figure}


\begin{theorem}\label{th:MainDistance}
Consider the invariant manifolds of infinity  $\WW^s_{\infty,G_0}$
and $\WW^u_{\infty,G_0}$ of the system associated to
Hamiltonian~\eqref{def:HamCircularRotating} in the level of energy
$\HH(\wt r,\phi,\wt y,\wt G;\mu,G_0) =-G_0^3$ and the corresponding invariant curves  $\ga^{u,s}$  of the
Poincar\'e map $\PP_{G_0,\phi_0}$. Then, there exists $G_0^\ast>0$ such that for any $G_0>G_0^\ast$ and  $\mu\in
(0,1/2]$,
\begin{itemize}
\item  the curves $\ga^{u,s}$    have a parameterization of the
form~\eqref{def:ParamInvManOriginal} and
\item if we fix a section $\wt r=\wt r^\ast$, the distance $d$ between these curves along this section is given by
\[
 \begin{split}
   d=&
\wt y_\h(v^*)\ii\mu (1-\mu)\sqrt{\pi}\Bigg[\frac{1- 2\mu}{2\sqrt{2}}G_0^{3/2 }e^{-\frac{G_0^3}{3}} \sin\left(\phi_0-\wt\al_\h(v^\ast) +G_0^3 v^\ast\right)\\
&+8G_0^{7/2 }e^{-\frac{2G_0^3}{3}} \sin2\left(\phi_0-\wt\al_\h(v^\ast) +G_0^3 v^\ast\right)\\
&+\OO\left( (1-2\mu)G_0e^{-\frac{G_0^3}{3}}+G_0^3e^{-\frac{2G_0^3}{3}} \right)\Bigg],
 \end{split}
\]
where $v^\ast$ is the only $v>0$ such that $\wt  r(v^\ast)=\wt r^\ast$
and $\wt y_\h(v)$ and $\wt\al_\h(v)$ are the $\wt y$ and $\wt
\al$ components of the unperturbed
homoclinic~\eqref{def:unperturbedhomoclinic}.
\end{itemize}
\end{theorem}

This theorem is proven in Section \ref{sec:HJ}. Observe that the distance is exponentially small with
respect to $G_0$, which is taken as a large parameter.
Moreover, note that the size of the first order in the formula of
Theorem~\ref{th:MainDistance} is significantly different depending
whether $\mu\neq 1/2$ or $\mu= 1/2$.  The reason is that
Hamiltonian~\eqref{def:HamCircularRotating} is periodic with respect
to~$\phi$ with period $2\pi$, in the first case, and $\pi$, in the
second one. The physical explanation of the different periodicity
goes as follows. For $\mu\neq 1/2$ (see left picture of
Figure~\ref{fig:Circular}) the primaries rotate around the center of
mass with the same period $2\pi$ performing circles of different
radius. On the contrary, when $\mu=1/2$ both bodies have the same
mass and therefore, they move along the same circle where the two
bodies are placed in diametrally opposed points (see right picture
of Figure~\ref{fig:Circular}). Therefore, in the case $\mu=1/2$ the
period of the system is the half of the period of the primaries
around the center of mass. As usually happens in exponentially small
splitting of separatrices phenomena, the smaller the period of the
fast perturbation, the smaller the distance between the manifolds
(see for instance~\cite{Neishtadt84}).
\begin{figure}[H]
\begin{center}
\includegraphics[height=4cm]{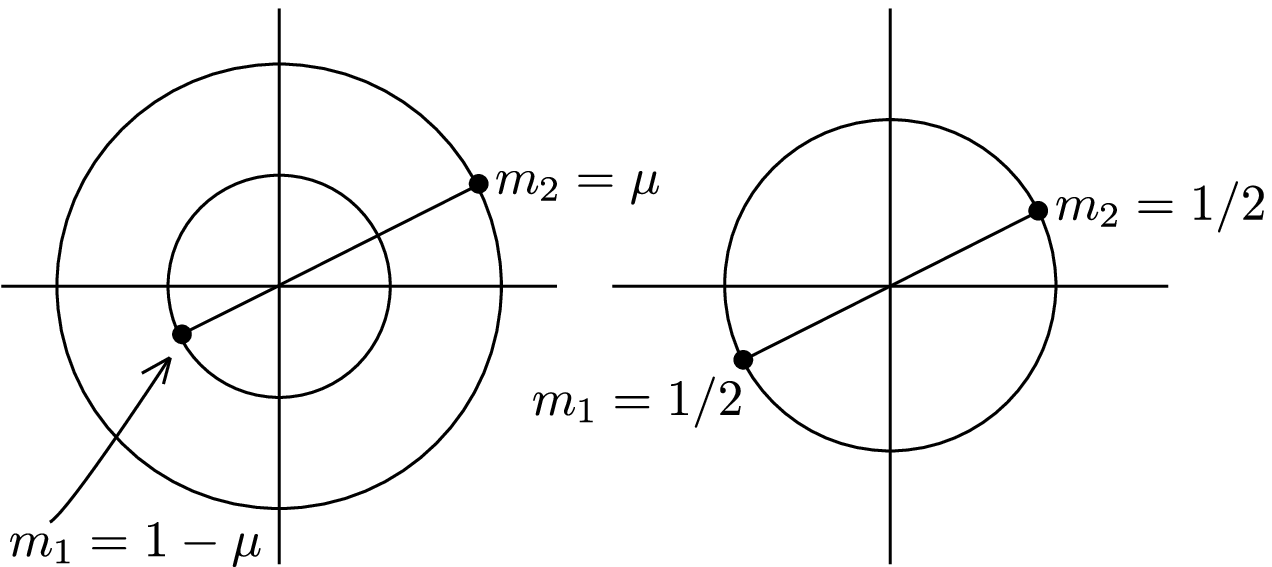}
\end{center}
\caption{The motion of the primary bodies in the cases $\mu\neq 1/2$ and $\mu=1/2$ respectively.}\label{fig:Circular}
\end{figure}

Looking at the formula given in Theorem \ref{th:MainDistance}, the zeros  of the distance are given, up to first order, by the zeros of the function
\[
f(x)=(1-2\mu)\sin x+16\sqrt{2}G_0^2e^{-\frac{G_0^3}{3}}\sin 2x\,\,\text{ where }x=\phi_0-\wt\al_\h(v) +G_0^3 v.
\]
The number of zeros for $x\in [0,2\pi)$ and their nondegeneracy depends strongly on the relation between the parameters $\mu$ and $G_0$. If we fix $\mu\neq 1/2$, and we take $G_0>0$ big enough, $f(x)\sim (1-2\mu)\sin x$ and therefore we will have two nondegenerate zeros, which give rise to two transversal homoclinic points. On the contrary, for $\mu=1/2$ and $G_0$ big enough, $f(x)\sim 16\sqrt{2}G_0^2e^{-G_0^3/3}\sin 2x$ and we will have four nondegenerate zeros. These facts are stated in Theorem \ref{th:SplittingInvManifoldsInfty}. Clearly, between these two regimes the system undergoes a global bifurcation where one of the two transversal intersections becomes a cubic tangency where two new homoclinic points are born. This occurs in a curve $\eta$ in the parameter plane (see Figure \ref{fig:BifPlane}) as is stated in Theorem \ref{th:tangencies}.
\begin{figure}[H]
\begin{center}
\includegraphics[height=5cm]{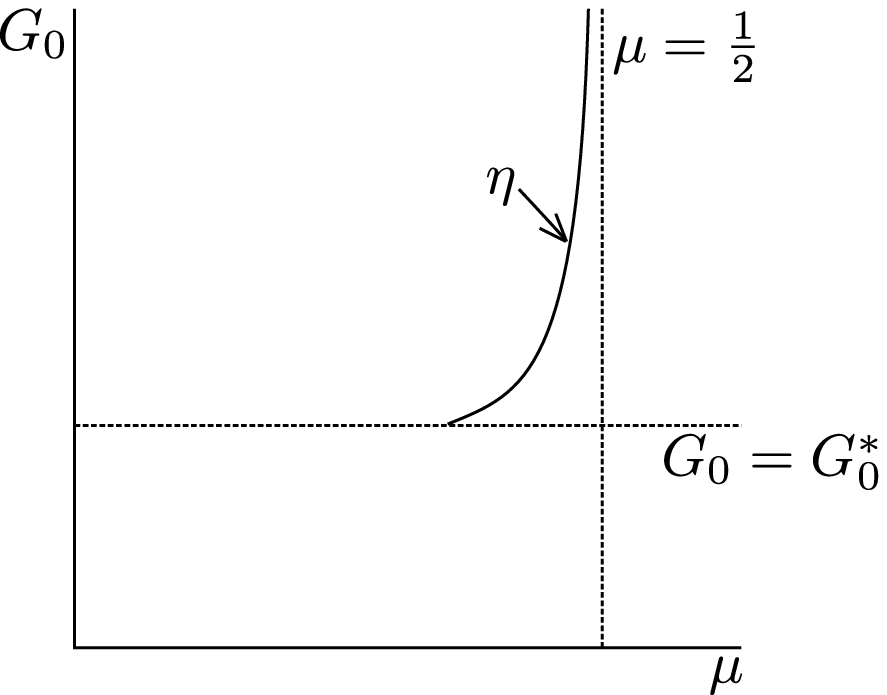}
\end{center}
\caption{Bifurcation curve $\eta$ in the parameter space where the homoclinic tangency is undergone.}\label{fig:BifPlane}
\end{figure}

\begin{theorem}\label{th:SplittingInvManifoldsInfty}
Fix $\mu\in
(0,1/2]$. Then there exists $G^\ast>0$ such that for any $G_0>G^\ast$,
\begin{itemize}
 \item the invariant curves  $\ga^{u,s}$  of the
Poincar\'e map $\PP_{G_0,\phi_0}$ intersect transversally and
\item the area of the lobes between the  invariant curves  $\ga^{u,s}$ between two transversal consecutive homoclinic points is given by
\[
\AAA=\mu (1-\mu)\sqrt{\pi}\left[
\frac{1-2\mu}{\sqrt{2}}G_0^{-3/2 }e^{-\frac{G_0^3}{3}}+ 8G_0^{1/2}
e^{-\frac{2G_0^3}{3}}\right]\left(1+\OO\left(G_0^{-1/2}\right)\right).
\]
\end{itemize}
\end{theorem}

\begin{theorem}\label{th:tangencies}
Let $G_0^\ast>0$ be the constant introduced in Theorem \ref{th:MainDistance}. Then, there exists a curve $\eta$ in the parameter region
\[
(\mu,G_0)\in \left(0,\frac{1}{2}\right]\times (G_0^\ast,+\infty),
\]
of the form
\[
 \mu=\mu^\ast(G_0)=\frac{1}{2}-16\sqrt{2}G_0^2 e^{-\frac{G_0^3}{3}} \left(1+\OO\left(G_0^{-1/2}\right)\right),
\]
such that, for $(\mu,G_0)\in\eta$,
\begin{itemize}
\item the invariant curves  $\ga^{u,s}$  of the
Poincar\'e map $\PP_{G_0,\phi_0}$  have a cubic homoclinic tangency and a transversal homoclinic point and
\item the area of the lobes between the  invariant curves  $\ga^{u,s}$ between the homoclinic tangency and a consecutive transversal homoclinic point is given by
\[
\AAA=10\sqrt{\pi}G_0^{1/2 }e^{-2\frac{G_0^3}{3}}\left(1+\OO\left(G_0^{-1/2}\right)\right).
\]
\end{itemize}
\end{theorem}

Theorem \ref{th:SplittingInvManifoldsInfty} is proven in Section \ref{sec:HJ}. The proof of Theorem \ref{th:tangencies} is deferred to the end of Section \ref{sec:DiffManifolds}.

\subsection{From transversal homoclinic points to oscillatory motions}\label{sec:ProofOscillatory}
The existence of oscillatory motions in the RPC3BP was obtained by
Llibre and Sim\'{o} in~\cite{SimoL80} for large values of  $G_0$ and
exponentially small values of the mass ratio~$\mu$, concretely for
$\mu =\OO(e^{-{G_0}^3/3})$. Their arguments follow the ones
developed by Moser in the so called Sitnikov problem
in~\cite{Moser01}. The proof requires, as a first step, to control
the local behavior near infinity. This is  done using McGehee
coordinates~\cite{McGehee73}, in which  infinity becomes a parabolic
critical point at the origin, and then the so-called Shilnikov
coordinates to study the local map near zero.


The local behavior of the system at infinity turns out to be the
same for all values of the parameter $\mu \in [0,1/2]$ and $G_0$ big
enough. The reason is that when one considers
system~\eqref{eq:Equations:CircularPolar}, makes the change of
variables $\phi=\al-t$ to express it in rotating coordinates,
performs the McGehee change of coordinates $r=2x^{-2}$, uses the
energy reduction and reparameterizes time $s=s(\theta)$ to have
$d\phi/d\theta=1$, one obtains (see~\cite{SimoL80}) the system
\[
 \begin{split}
\frac{dx}{d\theta}&=\frac{x^3y}{4}+K\frac{x^7y}{32}+\OO_{10}(x,y)\\
\frac{dy}{d\theta}&=\frac{x^4}{4}-K^2\frac{x^6}{32}+3K\frac{x^6y^2}{16}-\la(\theta;\mu)x^8+\OO_{10}(x,y)
 \end{split}
\]
where $K=\JJ-\mu (1-\mu)$ and  $ \la(\theta;\mu)=\frac{3}{32}\mu(1-\mu)(1-3\cos^2\theta)$.
Since the dependence on $\mu$ only appears in the higher order terms, the local behavior is the same regardless the value of the parameter $\mu\in [0,1/2]$. Therefore, the techniques developed by Moser in \cite{Moser01} that were applied in \cite{SimoL80} for the case $\mu$ small enough, are also valid in our setting.

The second step of the proof is to show that the invariant stable and unstable manifolds of infinity
intersect transversally. It was in this step that the smallness of
the parameter~$\mu$ was needed in~\cite{SimoL80}. Theorem~\ref{th:SplittingInvManifoldsInfty} guarantees this
transversality for all values of~$\mu$.

Hence combining this fact with the local results of~\cite{SimoL80}
we obtain Theorem~6.1 of~\cite{SimoL80} for any value of $\mu \in
(0,1/2]$, which shows  that there exists a Cantor set in the phase
space where certain return map associated to
system~\eqref{eq:Equations:CircularPolar} is conjugated with a shift
of infinitely many symbols. From this result,
Theorems~\ref{th:main:1} and~\ref{th:main:2} follow directly.

Another way of proving Theorem~\ref{th:main:1} and~\ref{th:main:2}
would be to consider the Poincar\'e map~\eqref{def:PoincareMap} in
McGehee coordinates, which has a parabolic fixed point at the origin
with transversal homoclinic points. Then, as Moser  mentions
in~\cite{Moser01} for the Sitnikov problem, one could adapt the
proof of  the classical Birkhoff-Smale theorem to parabolic fixed
points with transversal homoclinic intersections to obtain a Cantor
set of the phase space with dynamics topologically conjugated to a
shift of two symbols. Using this idea and proving a $\CCC^2$-version
of the Lambda Lemma, Gorodetski and Kaloshin in the recent
preprint~\cite{GorodetskiK12}, prove that  the Hausdorff dimension
of the set of initial conditions which lead to oscillatory motions
is maximal. Their proof also needs the same requirements about the
parameters~$\mu$ and~$G_0$ of~\cite{SimoL80} because their reasoning
uses the transversality of the invariant manifolds proved
in~\cite{SimoL80}. Private conversations with the authors indicate
that, using Theorem~\ref{th:SplittingInvManifoldsInfty}, their
results can  be extended for a larger set of the parameters.

\section{Parameterizations of the invariant manifolds as generating functions}\label{sec:HJ}
As we have explained in Section~\ref{sec:SystemAsFastOscillating},
to prove Theorems~\ref{th:MainDistance},
\ref{th:SplittingInvManifoldsInfty} and \ref{th:tangencies} we  consider
the rescaled system in rotating coordinates given by Hamiltonian
\eqref{def:HamCircularRotating}. The main two advantages of this
choice are the following. First, the unperturbed separatrix is
independent of the parameter $G_0$ and, second, working in rotating
coordinates our system has two degrees of freedom (and thus
conserved energy) instead of two and a half as happens in the
original variables.

To study the difference between the invariant manifolds of infinity,
we follow the approach proposed in~\cite{LochakMS03, Sauzin01}. That
is, we take advantage of the fact that the invariant manifolds are
Lagrangian and therefore they can be locally parameterized as graphs
of generating functions which are solutions of the so-called
Hamilton-Jacobi equation.

We  consider the Hamilton-Jacobi equation associated to
Hamiltonian~\eqref{def:HamCircularRotating} and  we look for
functions $S(\wt r,\phi;\mu,G_0)$ such that
\[
(\wt y,\wt G)=\left(\pa_{\wt r} S(\wt r,\phi;\mu,G_0),\pa_\phi S(\wt r,\phi;\mu,G_0)\right)
\]
parameterize the invariant manifolds as a graph. Then, the Hamilton-Jacobi equation reads
\begin{equation}\label{eq:HJ:original}
 \HH(\wt r,\phi,\pa_{\wt r} S,\pa_\phi S;\mu,G_0)=-G_0^3.
\end{equation}
Recall that we put $-G_0^3$ in the right hand side since it is the level of energy where we are looking for the invariant manifolds.

For the unperturbed Hamiltonian, that is, considering $V=0$
in~\eqref{def:HamCircularRotating}, this equation simply reads
\[
 \frac{1}{2}\left(\pa_{\wt r }S\right)^2-G_0^3 \pa_\phi S+\frac{1}{2 \wt r^2}\left(\pa_\phi S\right)^2-\frac{1}{\wt r}
 =-G_0^3.
\]
It has a solution of the form
\begin{equation}\label{def:GeneratingFunctionUnperturbed}
 S_0(\wt r,\phi)=\phi+f(\wt r),
\end{equation}
where $f$ is any solution of
\[
 \frac{1}{2}\left(\pa_{\wt r} f\right)^2+\frac{1}{2\wt r^2}-\frac{1}{\wt r}=0.
\]
In~\cite{LochakMS03, Sauzin01}, the authors deal with unperturbed
separatrices which can be written globally as graphs. This is not
possible in the present problem, as can be clearly seen in
Figure~\ref{fig:separatrix}, since the separatrix has a turning
point at $(\wt r,\wt y)=(1/2,0)$. Thus, we will deal with this
equation for $\wt r\neq 1/2$. This fact will carry some technical problems
while proving the existence of the invariant manifolds in certain
domains later in Section~\ref{sec:Manifolds}.

We look for solutions of~\eqref{eq:HJ:original} close to
\eqref{def:GeneratingFunctionUnperturbed}. To this end, we  write $S=S_0+S_1$ and
then the equation for $S_1$ becomes
\[
 \pa_{\wt r} f\pa_{\wt r} S_1+\frac{1}{2}\left(\pa_{\wt r} S_1\right)^2- G_0^3\pa_\phi S_1
 +\frac{1}{\wt r^2}\pa_\phi S_1+\frac{1}{2\wt r^2}\left(\pa_\phi S_1\right)^2- V(\wt r,\phi;\mu,G_0)=0.
\]
To look for solutions of this equation we proceed as in
\cite{LochakMS03, Sauzin01} and we reparameterize the
variables~$(\wt r,\phi)$ through the unperturbed
separatrix~\eqref{def:unperturbedhomoclinic}. Namely, we consider
the change
\begin{equation}\label{def:ChangeThroughHomo}
(\wt r,\phi)=(\wt r_\h(v),\xi+\wt\al_\h(v)).
\end{equation}
We define the new generating function
\begin{equation}\label{def:GeneratingFunctionT}
T(v,\xi;\mu,G_0)=S(\wt r_\h(v),\xi+ \wt \al_\h(v);\mu,G_0),
\end{equation}
which can be correspondingly written as $T=T_0+T_1$ where
\[
T_0 (v,\xi ) = S_0(\wt r_\h(v),\xi + \wt \al_\h(v))
\]
and
\[
T_1 (v,\xi;\mu,G_0) = S_1(\wt r_\h(v),\xi + \wt \al_\h(v);\mu,G_0),
\]
whose associated Hamilton-Jacobi equation is
\begin{equation}\label{eq:HJ:Rescaled}
\pa_v  T_1-G_0^{3} \pa_\xi    T_1+\frac{1}{2 \wt y_\h^2}\left(\pa_v
T_1-\frac{1}{\wt r_\h^2}  \pa_\xi T_1\right)^2+\frac{1}{2 \wt
r_\h^2}\left(\pa_\xi T_1\right)^2- V\left(\wt r_\h,\xi+ \wt
\al_\h;\mu,G_0\right)=0.
\end{equation}
Note that the change of variables~\eqref{def:ChangeThroughHomo}
implies that we are looking for parameterizations of the stable and
unstable invariant manifolds of the form
\begin{equation}\label{def:ParamViaHJ}
\begin{split}
\wt r&= \wt r_\h(v)\\
\wt y&= \wt y_\h(v)+\wt y_\h(v)\ii\left(\pa_v T_1^{u,s}(v,\xi;\mu,G_0)-\wt r_\h (v)^{-2}\pa_\xi T_1^{u,s}(v,\xi;\mu,G_0)\right)\\
\phi &= \xi+ \wt \al_\h(v)\\
\wt G&= 1+\pa_\xi  T_1^{u,s}(v,\xi;\mu,G_0),
\end{split}
\end{equation}
where $T_1^{u,s}$ are solutions of equation~\eqref{eq:HJ:Rescaled}
with asymptotic boundary conditions for the unstable manifold
\[
\begin{split}
\lim_{v\rightarrow -\infty}&\wt y_\h\ii(v)\pa_v T_1^u(v,\xi;\mu,G_0)=0 \\
\lim_{v\rightarrow -\infty}&\pa_\xi T_1^u(v,\xi;\mu,G_0)=0,
\end{split}
\]
and analogous ones for the stable manifold taking $v\rightarrow +\infty$. By its definition  in \eqref{def:PerturbedPotentialscaled}, $V$ has the symmetry property
\begin{equation}\label{def:symmetry:RescaledPotentia}
 V (r,-\theta)= V (r,\theta).
\end{equation}
Taking into account this fact and \eqref{def:SimetriaHomo}, one can easily see that if $T_1(v,\xi)$ is a solution of \eqref{eq:HJ:Rescaled}, $-T_1(-v,-\xi)$ is also a solution. Thus, since $T^u_1$ has to satisfy the just mentioned asymptotic boundary conditions and $T_1^s$ the opposite ones, the generating functions parameterizing the invariant manifolds must satisfy
\begin{equation}\label{def:Symmetry:GeneratingFunction}
 T_1^s(v,\xi)=-T_1^u(-v,-\xi).
\end{equation}
This means that if one is able to prove the existence of, for instance, the unstable invariant manifold, the existence of the stable one is guaranteed by the symmetry.

Theorems~\ref{th:MainDistance} and
\ref{th:SplittingInvManifoldsInfty} will be deduced
from the study of the difference  of the generating functions
$T_1^u$ and $T_1^s$ and from the difference between their
derivatives given in
Theorem~\ref{th:SplittingViaGeneratingFunctions} below. The proof of  Theorem \ref{th:tangencies} is postponed to the end of Section \ref{sec:DiffManifolds}, since we will need a slight modification of these generating functions to deduce the existence of tangencies.

It is worth remarking that, in view
of~\eqref{def:CondicioInicialHomo}  and the second equation
of~\eqref{def:ParamViaHJ},  the parameterizations are undefined at
$v=0$ and therefore it seems impossible to have the
parameterizations of both the stable and unstable manifold defined
in a common domain. Nevertheless, in Section~\ref{sec:Manifolds}, we
overcome this problem by using auxiliary parameterizations of the
unstable manifold which allow us, later on, to recover
parameterizations of both invariant manifolds of the
form~\eqref{def:ParamViaHJ} in a common compact domain contained in
$\{v>0\}\times\TT$, which in the original variables corresponds to
$y>0$ (see Figure~\ref{fig:splitting}).

Before stating Theorem~\ref{th:SplittingViaGeneratingFunctions}, we
define the function that will give the first asymptotic order
between the generating functions $T_1^{u,s}$. This function is
closely related to the so-called Poincar\'e function or Melnikov
potential. We define
\begin{equation}\label{def:FirstOrder}
L(v,\xi;\mu,G_0)=\int_{-\infty}^{+\infty} V(\wt r_\h(v+s),\xi-G_0^3
s+ \wt \al_\h(v+s);\mu,G_0)ds,
\end{equation}
where~$V$ is the rescaled perturbed Newtonian potential defined
in~\eqref{def:PerturbedPotentialscaled}. Note that the classical
Poincar\'e function is the first order in~$\mu$ of this function,
that is
\[
\int_{-\infty}^{+\infty}\left.\pa_\mu V(\wt
r_\h(v+s),\xi-G_0^3 s+ \wt \al_\h(v+s);\mu,G_0)\right|_{\mu=0} ds.
\]
This was the function considered as a first order of the difference
between the invariant manifolds of infinity in~\cite{SimoL80}, since
they were considering perturbative methods in $\mu$. On the
contrary, in the present paper $\mu$ can take any value
$\mu\in(0,1/2]$ and therefore we deal with the modified Poincar\'e
function~\eqref{def:FirstOrder}. The following proposition, whose
proof is deferred to Appendix~\ref{app:proofofMelnikovintegral},
gives estimates for this function.
\begin{proposition}\label{prop:Melinkov}
The function $L(v,\xi;\mu,G_0)$ satisfies
\[
L(v,\xi;\mu,G_0)=L^{[0]}(\mu,G_0)+2\sum_{\ell=1}^{+\infty}L^{[\ell]}(\mu,G_0)\cos\ell\left(\xi +G_0^3 v\right),
\]
where
\[
\begin{split}
L^{[1]}(\mu,G_0)=&-\mu (1-\mu)\sqrt{\pi}\frac{1-2\mu}{4\sqrt{2}}G_0^{-3/2}e^{-\frac{G_0^3}{3}}\left( 1+ \OO\left(G_0^{-2}\right)\right) \\
L^{[2]}(\mu,G_0)=&-2\mu (1-\mu)\sqrt{\pi}G_0^{1/2} e^{-\frac{2G_0^3}{3}}\left(1+\OO\left(G_0^{-{1/2}}\right)\right)\\
L^{[\ell]}(\mu,G_0)=&\left(G_0^{\ell-3/2}e^{-\frac{\ell G_0^3}{3}}\right), \,\,\text{ for }\ell\geq 3.
\end{split}
\]
\end{proposition}

\begin{theorem}\label{th:SplittingViaGeneratingFunctions}
There exist $0<v_-<v_+$, $G_0^\ast>0$ and $K>0$ such that, for any
$G_0>G_0^\ast$ and $\mu\in(0,1/2]$, the invariant manifolds of infinity  have
parameterizations of the form~\eqref{def:ParamViaHJ} for $(v,\xi)\in
(v_-,v_+)\times\TT$.

Moreover, the corresponding generating functions
satisfy
\[
 \left|T_1^u(v,\xi)-T_1^s(v,\xi)-L(v,\xi)-E\right|\leq K \mu^2 \left(1-2\mu\right)G_0^{-2}e^{\dps\tfrac{G_0^{-3}}{3}}+K G_0^{-1/2}\mu^2 e^{\dps\tfrac{2G_0^{-3}}{3}},
\]
for a constant $E\in\RR$, which might depend on $\mu$ and $G_0$, and
\[
\left|\pa_v^m\pa_\xi^ n T_1^u(v,\xi)-\pa_v^m\pa_\xi^ n T_1^s(v,\xi)-\pa_v^m\pa_\xi^ n L(v,\xi)\right|\leq K \mu^2 \left(1-2\mu\right)G_0^{-2+3m}e^{\dps\tfrac{G_0^{-3}}{3}}+K G_0^{-1/2+3m}\mu^2 e^{\dps\tfrac{2G_0^{-3}}{3}}
\]
for $0<m+n\leq 2$, $0\le m,n$, where we have omitted the dependence
on $\mu$ and $G_0$ of $T^{u,s}_1$, $L$ and $E$ to simplify the
notation.
 \end{theorem}

Sections~\ref{sec:Separatrix}, \ref{sec:Manifolds}
and~\ref{sec:DiffManifolds} will be devoted to prove this theorem.

\subsection{Proof of Theorems~\ref{th:MainDistance} and \ref{th:SplittingInvManifoldsInfty}}
We use the estimates obtained in Proposition~\ref{prop:Melinkov} and
Theorem~\ref{th:SplittingViaGeneratingFunctions} to prove
Theorems~\ref{th:MainDistance}
and~\ref{th:SplittingInvManifoldsInfty}. The first step is to deduce
parameterizations~\eqref{def:ParamInvManOriginal} of the stable and
unstable invariant curves of infinity $\gamma^{u,s}$ of the
Poincar\'e map $\PP_{G_0,\phi_0}$ in~\eqref{def:PoincareMap} from
the generating functions $T_1^{u,s}$ obtained in
Theorem~\ref{th:SplittingViaGeneratingFunctions}. Recall that, on
the one hand, the generating functions $T_1^{u,s}$ give us a
parameterization of the invariant manifolds of infinity associated
to Hamiltonian~\eqref{def:HamCircularRotating} of the
form~\eqref{def:ParamViaHJ}. On the other hand, the Poincar\'e
section we are considering is $\phi=\phi_0$, which has natural
coordinates $(r,y)$ (see~\eqref{def:PoincareMap}). This implies that
to obtain the parameterization~\eqref{def:ParamInvManOriginal}  of
the invariant curves $\gamma^{u,s}$, one has to impose $\xi+\wt
\al_\h(v)=\phi_0$ obtaining
\begin{equation}\label{def:InvariantManifoldsPoincare}
\begin{split}
 Y_{\phi_0}^{u,s}(v;\mu,G_0)=&\wt y_\h(v)\\
&+\wt y_\h(v)\ii\left(\pa_v T_1^{u,s}(v,\xi;\mu,G_0)-\wt r_\h (v)^{-2}\pa_\xi T_1^{u,s}(v,\xi;\mu,G_0)\right)\Big|_{\xi=\phi_0-\wt\al_\h(v)}.
\end{split}
\end{equation}
Using the estimates of Proposition \ref{prop:Melinkov} and Theorem
\ref{th:SplittingViaGeneratingFunctions}, one has that
\[
 \begin{split}
   Y_{\phi_0}^{s}(v;\mu,G_0)- Y_{\phi_0}^{u}(v;\mu,G_0)=&\wt y_\h(v)\ii\left(\pa_v T_1^{s}(v,\xi;\mu,G_0)-\pa_v T_1^{u}(v,\xi;\mu,G_0)\right)\Big|_{\xi=\phi_0-\wt\al_\h(v)}\\
&+\OO\left(\mu G_0^{-3/2}e^{-\dps\tfrac{G_0^{-3}}{3}}\right)\\
=&\wt y_\h(v)\ii \pa_vL\left(v,\phi_0-\wt\al_\h(v);\mu,G_0\right)\\
&+\OO\left(\mu^2 (1-2\mu) G_0e^{-\dps\tfrac{G_0^{-3}}{3}}+G_0^{5/2}\mu^2 e^{\dps\tfrac{2G_0^{-3}}{3}}\right).
\end{split}
\]
Now taking into account Proposition~\ref{prop:Melinkov}, we obtain the formula of the distance in
Theorem~\ref{th:MainDistance}.

To prove Theorem \ref{th:SplittingInvManifoldsInfty}, we first fix $\mu\in (0,1/2)$. Then, taking $G_0>G^\ast$, $G^*=G^*(\mu)$ big enough,
\[
\begin{split}
   Y_{\phi_0}^{s}(v;\mu,G_0)- Y_{\phi_0}^{u}(v;\mu,G_0)=&
-\wt y_\h(v)\ii\mu (1-\mu)\sqrt{\pi}\frac{ 2\mu-1}{2\sqrt{2}}G_0^{3/2 }e^{-\frac{G_0^3}{3}} \sin\left(\phi_0-\wt\al_\h(v) +G_0^3 v\right)\\
&+\OO\left(\mu G_0e^{-\frac{G_0^3}{3}} \right).
\end{split}
\]
This gives the formula of the distance for fixed $\mu\in (0,1/2)$. Proceeding analogously, one obtains the
following formula for $\mu=1/2$ and $G_0$ big enough,
\[
\begin{split}
   Y_{\phi_0}^{s}(v;1/2,G_0)- Y_{\phi_0}^{u}(v;1/2,G_0)=&
\wt y_\h(v)\ii\sqrt{\pi}G_0^{7/2 }e^{-\frac{2G_0^3}{3}} \sin2\left(\phi_0-\wt\al_\h(v) +G_0^3 v\right)\\
&+\OO\left(G_0^{3}e^{-\frac{2G_0^3}{3}} \right).
\end{split}
\]
We prove the transversality of the invariant manifolds and compute the area of the lobes for the case $\mu\neq 1/2$. The homoclinic points of infinity are just given by the equation
\[
 Y_{\phi_0}^{s}(v;\mu,G_0)- Y_{\phi_0}^{u}(v;\mu,G_0)=0.
\]
To locate the zeros of this equation we use the just obtained first order of the difference of the function $Y_{\phi_0}^{u,s}(v)$. The zeros of this first order correspond to the zeros of $\sin\left(\phi_0-\wt\al_\h(v) +G_0^3 v\right)$ and thus are  $v^0_k\in\RR$ such that
\begin{equation}\label{def:zerosMelnikov}
 \phi_0-\wt\al_\h(v^0_k)+G_0^{3}v^0_k=k\pi,\,\,\,k\in\ZZ,
\end{equation}
which are $G_0^{-3}$-close to each other. These zeros are non degenerate since for $G_0$ big enough $\wt\al_\h'(v)+G_0^3\neq 0$.

Using the estimates of Proposition~\ref{prop:Melinkov} and Theorem
\ref{th:SplittingViaGeneratingFunctions} and taking $G_0$ large
enough, one can apply the  Implicit Function Theorem to show that
the function   $Y_{\phi_0}^{s}(v)- Y_{\phi_0}^{u}(v)$ has
nondegenerate zeros $v_k^1\in\RR$ which satisfy
\begin{equation}\label{def:zerosImplicitFunctionTheorem}
 v_k^1=v^0_k+\OO\left(G_0^{-7/2}\right),
\end{equation}
and therefore the curves $\ga^u$ and $\ga^s$ intersect transversally at the corresponding homoclinic points.

To compute the asymptotic formula for the area of the lobes, we
proceed as in~\cite{BaldomaFGS11} and we consider the change of
coordinates
\[
 \left\{
\begin{aligned}
 r&=\wt r_\h(v)\\
 y&=\wt y_\h(v)\ii w.
\end{aligned}\right.
\]
Since $\pa_v\wt r_\h(v)=\wt y_\h(v)$, we have that $dr\wedge dy=dv\wedge dw$. We compute the area of the lobes in this new system of coordinates. Now the invariant curves are parameterized as a graph as
\[
w=\wt y_\h(v)Y^{u,s}_{\phi_0}(v;\mu,G_0).
\]
Consider two consecutive zeros which we denote by $v_k^1$ and
$v_{k+1}^1$. Then, using the definition of the function
$Y_{\phi_0}^{u,s}$ in~\eqref{def:InvariantManifoldsPoincare} and
that $\wt\al_\h'(v)=1/\wt r_\h^{2}(v)$ (see equation
\eqref{eq:Equations:Circularescalat} with $\mu=0$), the area of the
lobes is given by
\[
\begin{split}
\AAA=&\Bigg|\int_{v_k^1}^{v_{k+1}^1}\wt y_\h(v)\left(Y_{\phi_0}^{s}(v;\mu,G_0)- Y_{\phi_0}^{u}(v;\mu,G_0)\right)dv\Bigg|\\
=&\Bigg|\int_{v_k^1}^{v_{k+1}^1}\Big(\left(\pa_v T_1^{s}(v,\xi;\mu,G_0)-\wt r_\h (v)^{-2}\pa_\xi T_1^{u}(v,\xi;\mu,G_0)\right)\\
&\qquad\qquad-\left(\pa_v T_1^{s}(v,\xi;\mu,G_0)-\wt r_\h (v)^{-2}\pa_\xi T_1^{u}(v,\xi;\mu,G_0)\right)\Big)\Big|_{\xi=\phi_0-\wt\al_\h(v)}dv\Bigg|\\
=&\Big|\left( T_1^{s}(v_{k+1}^1,\phi_0-\wt\al_\h(v_{k+1}^1);\mu,G_0)- T_1^{u}(v_{k+1}^1,\phi_0-\wt\al_\h(v_{k+1}^1);\mu,G_0)\right)\\
&-\left( T_1^{s}(v_{k}^1,\phi_0-\wt\al_\h(v_{k}^1);\mu,G_0)- T_1^{u}(v_{k}^1,\phi_0-\wt\al_\h(v_{k}^1);\mu,G_0)\right)\Big|.
\end{split}
\]
Then, applying Theorem~\ref{th:SplittingViaGeneratingFunctions} we
have that
\[
 \AAA=L\left(v_{k+1}^1,\phi_0-\wt\al_\h(v_{k+1}^1);\mu,G_0\right)-L\left(v_{k}^1,\phi_0-\wt\al_\h(v_{k}^1);\mu,G_0\right)+\OO\left(\mu^2G_0^{-2}e^{\dps\tfrac{G_0^{-3}}{3}}\right).
\]
To obtain the formula for the area of the lobes in
Theorem~\ref{th:SplittingInvManifoldsInfty}, it is enough to take
into account the estimates in Proposition~\ref{prop:Melinkov}, that
the zeros $v_{k}^1$ satisfy~\eqref{def:zerosImplicitFunctionTheorem}
and that the zeros of the first order $v_k^0$ satisfy
\eqref{def:zerosMelnikov}.

 The case $\mu=1/2$ can be done analogously but taking into account that now the zeros of the first order $v_k^0$ are solutions of
\[
  \phi_0-\wt\al_\h(v^0_k)+G_0^{3}v^0_k=k\frac{\pi}{2},\,\,\,k\in\ZZ.
\]

\section{The integrable system and the homoclinic solution}\label{sec:Separatrix}
To study the existence of the solutions of the Hamilton-Jacobi
equation~\eqref{eq:HJ:Rescaled}, we need to recall some analytic
properties of the functions $\wt r_\h$, $\wt y_\h$ and $\wt \al_\h$,
which are solution of
\[
\begin{split}
\frac{d}{dv} \wt r &=\wt y\\
\frac{d}{dv} \wt y &=\frac{1}{\wt r^3}-\frac{1}{\wt r^2}\\
\frac{d}{dv} \wt \al &=\frac{1}{\wt r^2},
\end{split}
\]
close to its singularities and as $|v|\rightarrow +\infty$.

Next lemma (see~\cite{SimoL80, MartinezP94}) gives the homoclinic
solution of infinity of these equations
satisfying~\eqref{def:CondicioInicialHomo} after a
reparameterization of time.
\begin{lemma}
\label{lem:homoclinicexplicitexpressions} Let $\tau(v)$ be the
unique analytic function defined by
\[
v = \frac{1}{2}\left(\frac{1}{3} \tau^3 + \tau \right),
\]
with the convention that $\tau$ is real for real values of $v$, that
is,
\[
\tau(v) = \left(3v+\sqrt{9v^2+1}\right)^{1/3} - \left(3v+\sqrt{9v^2+1}\right)^{-1/3}.
\].
Then, the rescaled homoclinic~\eqref{def:unperturbedhomoclinic}
satisfying~\eqref{def:CondicioInicialHomo} has the following
properties:
\begin{enumerate}
\item $\wt r(v) = \wh r(\tau(v))$, where
\[
\wh r(\tau) = \frac{1}{2}(\tau^2+1).
\]
\item
$\wt \alpha(v) = \wh \alpha(\tau(v))$, where
\begin{equation}
\label{def:whalpha}
\wh \alpha(\tau) = 2 \arctan(\tau) = -i \log \left( \frac{i-\tau}{i+\tau}\right).
\end{equation}
\end{enumerate}
\end{lemma}
As a consequence, we can deduce the behavior for the homoclinic
close to its singularities and as $|v|\rightarrow +\infty$. Note
that these singularities are in fact zeros of the function $\wt r_\h$
and therefore can be seen as collisions of the parabolic orbit of
the two body problem, which occur for complex values of time.

\begin{corollary}\label{coro:HomoInfinity}
The rescaled homoclinic~\eqref{def:unperturbedhomoclinic} with
initial conditions~\eqref{def:CondicioInicialHomo}
 behaves as
\[\wt r_\h (v)\sim  3 v^{2/3}, \quad
\wt y_\h (v)\sim  2 v^{-1/3}, \quad \text{ and }\quad
\wt \al_\h (v)-\al_0\sim  \frac{1}{3} v^{-1/3}
\]as $|v|\rightarrow +\infty$.
\end{corollary}

Now we give the properties of the homoclinic close to its singularities.
Note that $\tau = \pm i$ if and only if $v = \pm i/3$. Furthermore,
\begin{equation}
\label{eq:singularitiesofthechangeofvariables}
\tau \pm i \sim (v\pm i/3)^{1/2}.
\end{equation}

\begin{corollary}\label{coro:HomoSing}
The rescaled homoclinic~\eqref{def:unperturbedhomoclinic} with
initial conditions~\eqref{def:CondicioInicialHomo}
 behaves as
\[
\wt r_\h(v)\sim C\left( v\mp\frac{i}{3}\right)^{1/2},\quad \wt
y_\h(v)\sim \frac{C}{2}\left( v\mp \frac{i}{3}\right)^{-1/2}, \quad
\wt\al_\h(v)\sim \pm\frac{i}{2}\ln \left( v-\frac{i}{3}\right),
\]
where $C^2=\pm 2i$, and
\[
e^{i \wt \al_\h (v)}\sim \frac{v-\frac{i}{3}}{v+\frac{i}{3}}.
\]
\end{corollary}

\section{Existence of the invariant manifolds}\label{sec:Manifolds}

We devote this section to obtain parameterizations of the invariant
manifolds of infinity as solutions of
equation~\eqref{eq:HJ:Rescaled}. Introducing the linear operator
\begin{equation}\label{def:Outer:DiffOperator}
\LL(h)=\pa_v h-G_0^{3} \pa_\xi h
\end{equation}
equation~\eqref{eq:HJ:Rescaled} becomes
\[
 \LL(T_1) = -\frac{1}{2 \wt
y_\h^2}\left(\pa_v T_1-\frac{1}{\wt r_\h^2}  \pa_\xi
T_1\right)^2-\frac{1}{2\wt r_\h^2}\left(\pa_\xi T_1\right)^2+\wh
U (v,\xi+\wt\al_\h).
\]
where
\[
 \wh U (v,\theta) = V\left(\wt r_\h(v),\theta;\mu, G_0\right),
\]
and~$V$ is the function defined
in~\eqref{def:PerturbedPotentialscaled}. From now on, we omit the dependence on the parameters $\mu$ and $G_0$.
To look for solutions of this equation  we  split the potential $\wh U$ as $\wh U=\wh U_0+\wh U_1$ where
\begin{equation}
\label{def:whU0} \wh U_0 (v,\theta) =
-\frac{\mu}{2}(1-\mu-3(1-\mu)\cos^2 \theta) \frac{1}{G_0^4 \wt r_\h
(v)^3}
\end{equation}
and
\begin{equation}
\label{def:whU1}
\wh U_1  = \wh U - \wh U_0.
\end{equation}

We consider two different left inverses of $\LL$ given by
\begin{align}
\GG^u(h)(v,\xi)=\int_{-\infty}^0 h(v+s, \xi-G_0^3s)ds\label{def:Outer:IntegralOperator}\\
\GG^s(h)(v,\xi)=\int_{+\infty}^0 h(v+s, \xi-G_0^3s)ds\notag,
\end{align}
whenever these expressions are well defined. Using these operators and
the map
\begin{equation}
\label{def:mapA} A:(v,\xi)\mapsto (v,\xi+\wt \alpha (v)),
\end{equation}
we introduce the functions
\begin{equation}
\label{def:Q0} Q_0^\ast  = \GG^\ast(\wh U_0 \circ A),\,\,\,\ast=u,s
\end{equation}
We will see later that $Q_0^\ast$ are indeed well defined in suitable domains. Finally, we
introduce the new unknown $Q^\ast$ satisfying
\begin{equation}\label{def:Q}
 T_1^\ast = Q_0^\ast + Q^\ast.
\end{equation}
The equation for $Q^\ast$ is
\begin{equation}\label{eq:Q}
\LL(Q^\ast)= \FF^\ast (Q^\ast)
\end{equation}
where $\FF^\ast$ is given by
\begin{equation}\label{def:Outer:OperatorF}
\FF^\ast(h)=
-\frac{1}{2 \wt y_\h^2}\left(\pa_v Q_0^\ast+\pa_v h-\frac{1}{\wt
r_\h^2} \left(\pa_\xi Q_0^\ast+ \pa_\xi h\right)\right)^2\\
-\frac{1}{2\wt r_\h^2}\left(\pa_\xi Q_0^\ast+\pa_\xi h\right)^2+ \wh
U_1\left(v,\xi+\wt \al_\h(v)\right).
\end{equation}
We devote the rest of the section to study solutions of these
equations in suitable domains. First in
Section~\ref{sec:ExistenceManifolds:Description} we give the main
ideas of the proof and define the complex domains were we will look
for the solutions of equations~\eqref{eq:Q}. Then, in Section
\ref{sec:PerturbativePotential} we analyze the functions $\wh U_0$
and $\wh U_1$. Finally, in Sections
\ref{sec:ManifoldsInfinity}-\ref{sec:InvManifold:ExtensionUnstable} we obtain
parameterizations of the invariant manifolds in the different
domains  specified in Section
\ref{sec:ExistenceManifolds:Description}.

\subsection{Description of the proof}\label{sec:ExistenceManifolds:Description}
The classical way to study exponentially small splitting of
separatrices would be in this setting  to look for the functions
$Q^u$ and $Q^s$ as solutions of equations~\eqref{eq:Q} in a certain
complex common domain $D\times\TT_\sigma$, where $D\subset\CC$ is a domain which reaches a neighborhood of size
$\OO(G_0^{-3})$ of the singularities of the unperturbed separatrix,
namely $v=\pm i/3$ and
\[
\TT_\sigma=\left\{\xi\in \CC/(2\pi\ZZ): |\Im \xi|<\sigma\right\}.
\]However,  to solve equations~\eqref{eq:Q} we
have to face two different problems.

The first one is that equation~\eqref{eq:Q} becomes singular at
$v=0$  because $\wt y_\h(0)=0$ (see
\eqref{def:CondicioInicialHomo}). To overcome this problem we
proceed  as was done in~\cite{BaldomaFGS11} (see also
\cite{Guardia12}). We look for solutions of equations~\eqref{eq:Q}
in a complex domain for the variable $v$, which was called in that paper
\emph{boomerang domain} due to its shape.
The important features of this domain is that contains both an
interval of the real line of width independent of $G_0$ and points
at a distance $\OO(G_0^{-3})$ of the singularities $v=\pm i/3$, and
that it does not contain $v=0$.  This domain can be seen in
Figure~\ref{fig:DomRaro} and is defined as follows
\begin{equation}\label{def:DominisRaros}
\begin{split}
D_{\kk,\de}=
&\left\{v\in\CC;\right.\left. |\Im v|<\tan\beta_1\Re
v+1/3-\kk G_0^{-3},
|\Im v|<-\tan\beta_1\Re v+1/3-\kk G_0^{-3},\right.\\
& \left.|\Im v|>-\tan\beta_2\Re v+1/6-\delta\right\},
\end{split}
\end{equation}
where  $\kappa\in (0,1/3)$, $\de\in (0,1/12)$  and
$\beta_1,\beta_2\in (0,\pi/2)$  are fixed  independently of~$G_0$.
Therefore, this domain is non empty provided $G_0>1$.

\begin{figure}[H]
\begin{center}
\includegraphics[height=6cm]{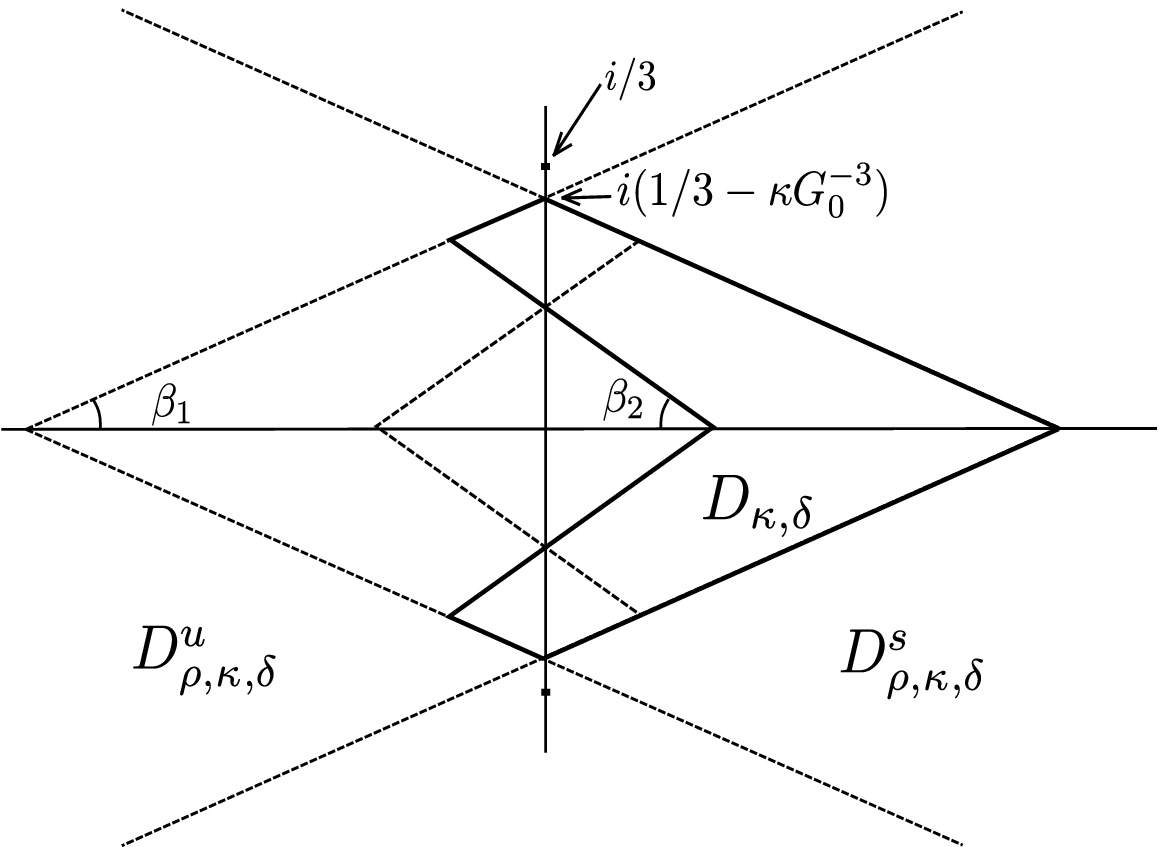}
\end{center}
\caption{The domains  $D_{\kk,\de}$ defined in
\eqref{def:DominisRaros}.}\label{fig:DomRaro}
\end{figure}
To obtain the parameterizations of the invariant manifolds of
infinity  using  equations~\eqref{eq:Q}, we need  to impose the
asymptotic conditions
\begin{equation}\label{def:DecayAtInfinity}
\begin{split}
\lim_{\Re v\rightarrow -\infty}\wt y_\h\ii(v)\pa_v Q^u(v,\xi)=0 ,\quad \lim_{\Re v\rightarrow -\infty}\pa_\xi Q^u(v,\xi)=0 \\
\lim_{\Re v\rightarrow +\infty}\wt y_\h\ii(v)\pa_v Q^s(v,\xi)=0, \quad \lim_{\Re v\rightarrow +\infty}\pa_\xi Q^s(v,\xi)=0.
\end{split}
\end{equation}
However, these conditions do not have any meaning in the domain
$D_{\kk,\de}$ since this domain is bounded. Therefore, to prove the
existence of the parameterizations of the invariant manifolds in the
domain $D_{\kk,\de}$, one has to start with different domains were
these asymptotic conditions make sense and, then, one has to find a
way to extend them analytically to the domain $D_{\kk,\de}$.

Thus, the first step is to look for solutions of
equations~\eqref{eq:Q} in the domains
\begin{equation}\label{def:DomainInfinity}
\begin{split}
D^{u}_{\infty,\rr}&=\{v\in\CC; \Re v<-\rr\}\\
D^{s}_{\infty,\rr}&=\{v\in\CC; \Re v>\rr\},
\end{split}
\end{equation}
for some $\rr>0$, where one can impose the asymptotic
conditions~\eqref{def:DecayAtInfinity}.  This is done in
Theorem~\ref{thm:Infinity:ExistenceManifolds} of
Section~\ref{sec:ManifoldsInfinity}.

To analytically extend the  invariant  manifolds to reach $D_{\kk,\de}$ we have to face the problem that these parameterizations become undefined at  $v=0$. To overcome it, the second step is to first analytically extend them to the domains
\begin{equation}\label{def:DomainOuter}
\begin{split}
D^{u}_{\rr',\kk,\de}=\Big\{v\in\CC;& |\Im v|<-\tan \beta_1\Re v+1/3-\kk G_0^{-3},
|\Im v|>\tan \beta_2\Re v+1/6-\de,\\
&\Re v>-\rr'\Big\}\\
\dps D^{s}_{\rr',\kk,\de}=\{v\in\CC; &-v\in D^{\out,u}_{\rr,\kk}\},
\end{split}
\end{equation}
which do not contain $v=0$ (see Figure~\ref{fig:OuterDomains}). We
choose $\rr'>\rr$ so that $D^{u,s}_{\rr',\kk,\de}\cap
D^{u,s}_{\infty,\rr}\neq\emptyset$, since then we can perform the
analytic extension procedure. Note that $ D_{\kk,\de}\subset
D^{s}_{\rr,\kk,\de}$. This is not true for the unstable manifold.


\begin{figure}[H]
\begin{center}
\includegraphics[height=6cm]{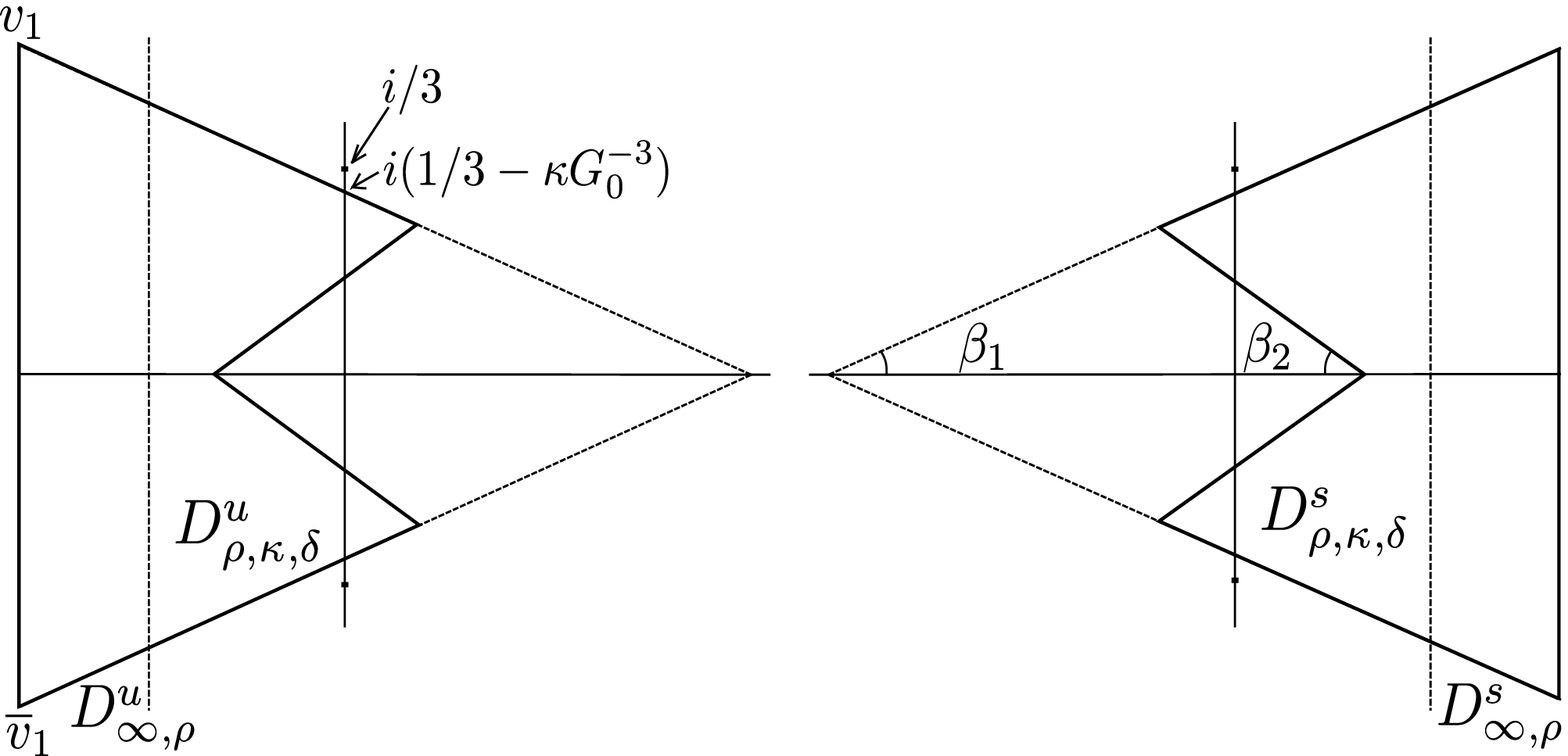}
\end{center}
\caption{The domains  $D^{u}_{\rr,\kk}$ and $D^{s}_{\rr,\kk}$ defined in
\eqref{def:DomainOuter}.}\label{fig:OuterDomains}
\end{figure}
When one tries to perform this analytical extension procedures, the
second problem arises, which, as  far as the authors know, had not
appeared before in any exponentially small splitting of separatrices
problem. It is well known that the solutions of~\eqref{eq:Q} become
very large  in the upper and lower vertices of the domains
$D^{u,s}_{\rr,\kk,\de}$,  since they reach $\OO(G_0^{-3})$
neighborhoods of the  singularities of the unperturbed separatrix.
However, in the exponentially small splitting of separatrices
problems that have been studied in the literature, using suitable
weighted Banach spaces, one is able to prove that, even if they
become big, they can be analytically extended to the whole domains
$D^{u,s}_{\rr,\kk,\de}$. This is not possible for
equations~\eqref{eq:Q}  since the solutions of these equations blow
up before reaching these points. Indeed, if one looks for the
analytic continuation of the functions
$Q^{u,s}(v,\xi)=\sum_{\ell\in\ZZ}Q_{u,s}^{[\ell]}(v)e^{i\ell \xi}$,
their Fourier coefficients $Q_{u,s}^{[\ell]}(v)$ grow exponentially
in $\ell$ near the singularities $v=\pm i/3$ and therefore one
cannot obtain these analytic extensions. This technical difficulty
can be overcome since to study the exponentially small splitting of
separatrices phenomena one does not need to extend the
functions~$Q^{u,s}(v,\xi)$ but it suffices to extend their Fourier
coefficients $Q_{u,s}^{[\ell]}(v)$. If the functions~$Q^{u,s}$ are
well defined up to neighborhoods of the singularities, these two
analytic extension procedures are equivalent, but in the present
paper they are not. This implies that the extension theorem that we
consider in this paper, which is explained in
Section~\ref{sec:ManifoldsOuter}, is rather different from the ones
in the literature (see for instance~\cite{BaldomaFGS11}). Indeed, we
will consider the Fourier coefficients of $Q^{u,s}$ as sequences of
functions and we will study their extension in a Banach space of
sequences of functions. This Banach space will be endowed with a
weighted norm that will allow us to obtain good estimates for each
Fourier coefficient. With this procedure we are able to analytically
extend the Fourier coefficients of the solutions of
equations~\eqref{eq:Q}  to the domains $D^{u,s}_{\rr',\kk,\de}$ even
if the corresponding Fourier series are not convergent. This is
shown in Theorem~\ref{thm:outer:ExistenceManifolds} in
Section~\ref{sec:ManifoldsOuter}.

Once the analytical extension of the Fourier coefficients
of~$Q^{u,s}(v,\xi)$ to $D^{u,s}_{\rr',\kk,\de}$ has been obtained,
it only remains to complete the analytical extension procedure for
the unstable manifold. Indeed, as we have already pointed out
$D_{\kk,\de}\subset D^{s}_{\rr,\kk,\de}$ and therefore, we have the
Fourier coefficients of the parameterization of the  stable manifold
in the boomerang domain~\eqref{def:DominisRaros}, where we will
study the difference between the manifolds. For the unstable
manifold we still need to reach the points in $ D_{\kk,\de}$ that do
not belong to $D^{u}_{\rr,\kk,\de}$. That is, we need to
analytically extend the Fourier coefficients of $Q^u$ to the domain
\begin{equation}\label{def:Domain:UnstableLastPiece}
\begin{split}
\wt D_{\kk,\de}=
\Big\{v\in\CC; &|\Im v|<-\tan\beta_1\Re
v+1/3-\kk G_0^{-3}, |\Im v|>-\tan\beta_2\Re v+1/6-\delta\\
& |\Im v|<\tan\beta_2\Re v+1/6+\delta\Big\}.
\end{split}
\end{equation}
(see Figure~\ref{fig:DomExtFinal}).

\begin{figure}[H]
\begin{center}
\includegraphics[height=6cm]{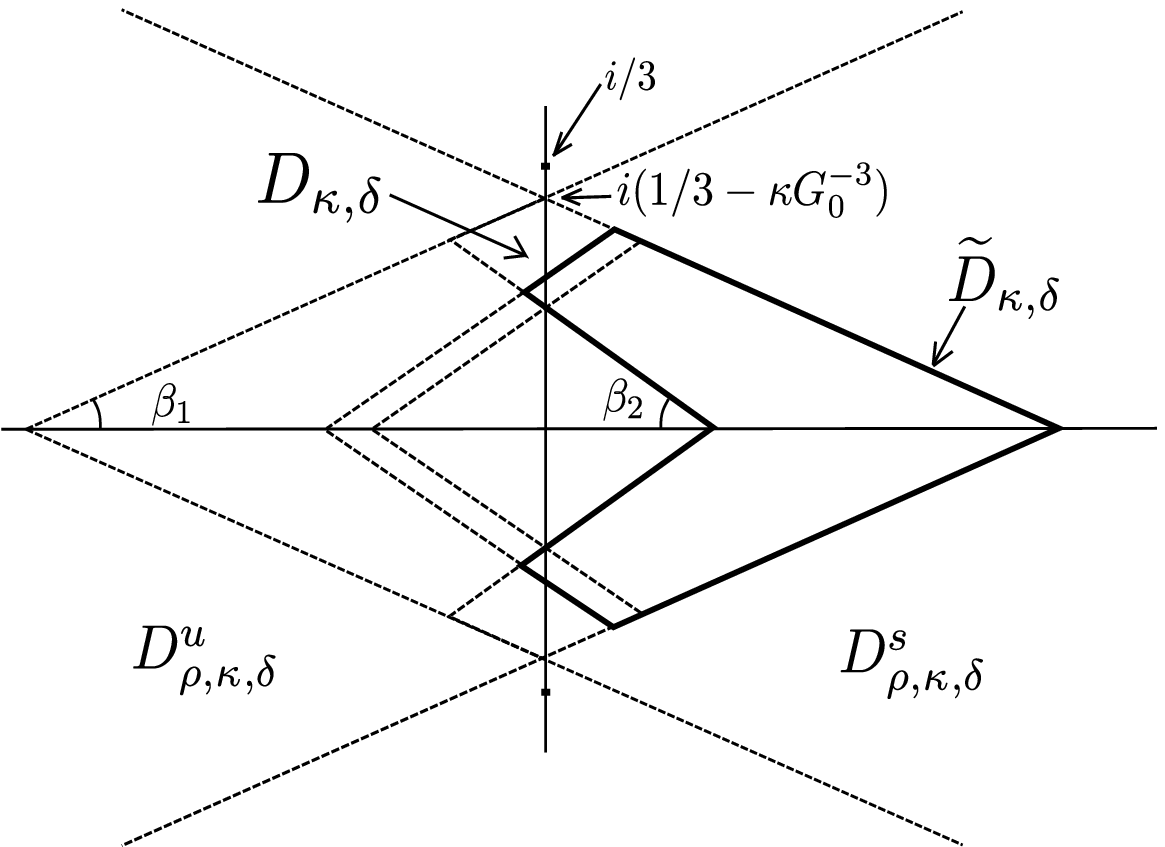}
\end{center}
\caption{The domain  $\wt D_{\kk,\de}$ defined in
\eqref{def:Domain:UnstableLastPiece}.}\label{fig:DomExtFinal}
\end{figure}

Indeed, it is easy to see that
\begin{equation}\label{def:DomainsInclusions}
 D_{\kk,\de}\subset D^u_{\rr,\kk,\de}\cup\wt D_{\kk,\de}.
\end{equation}
To obtain the analytic extension of the Fourier coefficients of
$Q^u$ to the domain $\wt D_{\kk,\de}$ we proceed as follows.  First,
we switch to the classical flow parameterization of the invariant
manifolds~\cite{DelshamsS92}, which is well defined for $v=0$. This
parameterization is of the form
\begin{equation}\label{def:ParamFlow}
\left(\wt r,\wt y,\phi,\wt  G\right)=\Gamma^u(v,\xi),
\end{equation}
where $\Gamma^u=\Gamma_0+\Gamma_1^u$, $\Gamma_0(v,\xi)=(\wt
r_\h(v), \wt y_\h(v), \xi+\wt\al_\h(v),1)$ is the parameterization
of the unperturbed homoclinic orbit and $\Gamma^u(v+t,\xi-G_0^3t)$
is a solution of the flow associated to the
Hamiltonian~\eqref{def:HamCircularRotating}.

We look for this parameterization in the domain
\begin{equation}\label{def:DominiParamFlux}
\begin{split}
D^\fl_{\kk,\de}=
\Big\{v\in\CC;& |\Im v|<-\tan\beta_1\Re
v+1/3-\kk G_0^{-3},\\
&|\Im v|<\tan\beta_2\Re v+1/6+\de\Big\},
\end{split}
\end{equation}
which can be seen in Figure~\ref{fig:DomFlux}. Finally we will
switch back to the original parameterization~\eqref{def:ParamViaHJ}
in the domain $\wt D_{\kk,\de}=D^\fl_{\kk,\de}\cap D_{\kk,\de}$
defined in~\eqref{def:Domain:UnstableLastPiece}, which does not
contain $v=0$ and therefore equation~\eqref{eq:Q} is well defined.
Since the inclusion~\eqref{def:DomainsInclusions} is satisfied,
after this procedure we will have the Fourier coefficients of the
parameterization of the unstable manifold $Q^u$ defined in the whole
domain  $D_{\kk,\de}$.

\begin{figure}[H]
\begin{center}
\includegraphics[height=6cm]{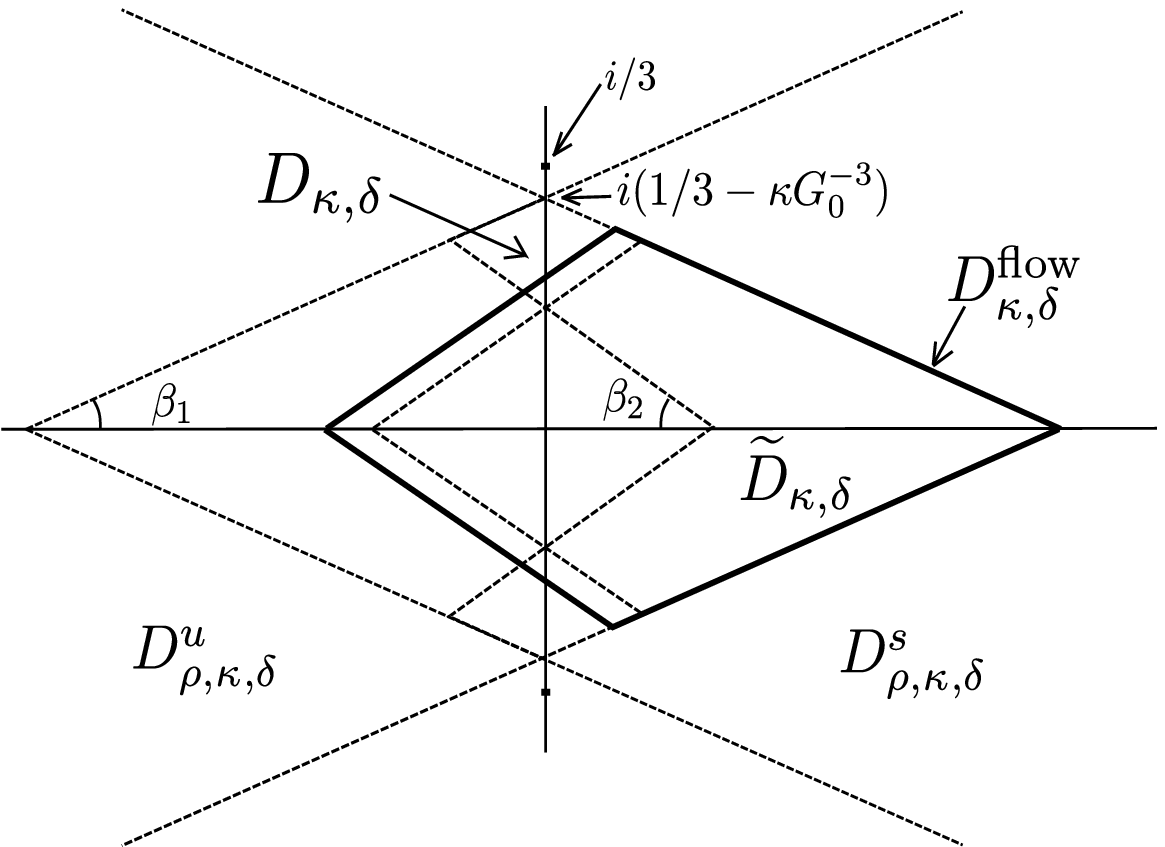}
\end{center}
\caption{The domain  $ D^\fl_{\kk,\de}$ defined in
\eqref{def:Domain:UnstableLastPiece}.}\label{fig:DomFlux}
\end{figure}

This analytical extension procedure through the flow
parameterization is explained in detail in
Section~\ref{sec:InvManifolds:Flow}. It has three steps
\begin{enumerate}
\item First, we  derive a parameterization of the form~\eqref{def:ParamFlow} in  the  domain
$D^{u}_{\rr,\kk,\de}$ by performing a suitable change of variables
to  the parameterization~\eqref{def:ParamViaHJ}, which are already
known in this domain thanks to
Theorem~\ref{thm:outer:ExistenceManifolds}. This is done in Theorem
\ref{th:ExistenceManifolds:Summary}.
\item  Once we have obtained this parameterization~\eqref{def:ParamFlow}
in the domain $D^{u}_{\rr,\kk,\de}$, we  extend it, using the flow,
to the whole domain~$D^\fl_{\kk,\de}$. This is done in
Proposition~\ref{prop:extensionbytheflow}.
\item Finally, using again a change of variables, the last step is to switch back to the
parameterization~\eqref{def:ParamViaHJ} in the domain~$\wt
D_{\kk,\de}=D^\fl_{\kk,\de}\cap D_{\kk,\de}$. This is done in
Proposition~\ref{prop:FromParamToHJ}.
\end{enumerate}

\noindent \textbf{Notation.} \emph{To simplify the notation, in the forthcoming sections, we will denote by $K$ any constant
independent of $\mu$ and $G_0$ to state all the bounds.}

\subsection{The perturbative potential}\label{sec:PerturbativePotential}
We devote this section to obtain estimates for the functions $\wh
U_0$ and $\wh U_1$ defined in~\eqref{def:whU0} and~\eqref{def:whU1},
respectively, in the domains defined in
Section~\ref{sec:ExistenceManifolds:Description}. This estimate will
be used in the subsequent Sections
\ref{sec:ManifoldsInfinity}-\ref{sec:InvManifold:ExtensionUnstable}.


First we state the following technical lemma, whose proof is
straightforward taking into account Corollary~\ref{coro:HomoSing}.
\begin{lemma}\label{lemma:CotaRadi}
For any $v\in D$, where $D$ is any of the
domains~\eqref{def:DominisRaros}, \eqref{def:DomainInfinity},
\eqref{def:DomainOuter}, \eqref{def:Domain:UnstableLastPiece},
\eqref{def:DominiParamFlux} and $G_0$ big enough, the following
bound is satisfied
\[
\left|\frac{1}{G_0^2\wt r_\h(v)}\right| \le \frac{K}{G_0^{1/2}},
\]
where $\wt r_\h$ is the $\wt r$ component of the homoclinic
orbit~\eqref{def:unperturbedhomoclinic}.
\end{lemma}

\begin{lemma}\label{lemma:whUminuswhU0bound}
For any $(v,\theta)\in D\times\TT$, where $D$ is any of the
domains~\eqref{def:DominisRaros},~\eqref{def:DomainInfinity},
\eqref{def:DomainOuter}, \eqref{def:Domain:UnstableLastPiece},
\eqref{def:DominiParamFlux} and $G_0$ big enough, the functions $\wh
U_0$ and  $\wh U_1$ satisfy
\[
\begin{split}
 | \wh U_0 (v,\theta)| &< \mu \frac{K}{G_0^4 |\wt r_\h(v)|^3}\\
| \wh U_1 (v,\theta)| &< \mu \frac{K}{G_0^6 |\wt r_\h(v)|^4}.
\end{split}
\]
\end{lemma}

\begin{proof}
The prove of the first bound is straightforward. For the second one
it is enough to  bound the remainder in Taylor's formula using
Lemma~\ref{lemma:CotaRadi} and to apply  Schwarz's lemma to obtain
\[
 \left|\wh U (v,\theta)  +\frac{1}{2}
\mu (1-\mu-3(1-\mu) \cos^2 \theta) \frac{1}{G_0^4 \wt r_\h
(v)^3}\right| \le \mu G_0^2 \frac{K}{|G_0^2\wt r_\h(v)|^4}.
\]
\end{proof}

\begin{corollary} \label{lem:whUsingularites} Denoting by $h^{[\ell]}(v)$ the $\ell$-th Fourier
coefficient of a function $h(v,\theta)$, the following bounds hold
\begin{enumerate}
\item
$|\wh U_0^{[\ell]}(v)| \le  \mu \frac{K}{G_0^4 |\wt r_\h(v)|^3}$, for
$|\ell|=0,2$, $\wh U_0^{[\ell]}(v) =0$, for $|\ell|\neq 0,2$, and
\item
$|\wh U_1^{[\ell]}(v)| \le\mu \frac{K}{G_0^6 |\wt r_\h(v)|^4}$, for
$\ell \in \ZZ$.
\end{enumerate}
\end{corollary}



We finish this section by giving estimates for $\wt\al_\h(v)$, which
will be crucial to extend the Fourier coefficients of the
parameterization of the invariant manifolds close to the
singularities of the unperturbed separatrix.

\begin{lemma} \label{lem:eikalphasingularities} Let $\wt \al_\h$ be the $\wt \al$ component of the homoclinic
orbit~\eqref{def:unperturbedhomoclinic}. Let $\ell\in \ZZ$. Then,
for $G_0>0$ big enough the function $e^{i\ell\wt \alpha_\h(v)}$
satisfies that
\begin{enumerate}
\item For $v\in D$, where $D$ is any of the domains~\eqref{def:DominisRaros},~\eqref{def:DomainOuter},~\eqref{def:Domain:UnstableLastPiece},~\eqref{def:DominiParamFlux}
\[
\begin{split}
\left|e^{i\ell\wt \alpha_\h(v)}\right| &\leq\left|K\frac{v-i/3}{v+i/3}\right|^{|\ell|/2}\,\,\,\text{for }\ell>0\\
\left|e^{i\ell\wt \alpha_\h(v)}\right| &\leq\left|K\frac{v+i/3}{v-i/3}\right|^{|\ell|/2}\,\,\,\text{for }\ell<0.
\end{split}
\]
\item
For $v \in D_{\infty,\rr}^{u,s}$ (see~\eqref{def:DomainInfinity}),
\[
\left|e^{i\ell\wt \alpha_\h(v)}\right| \le K^{|\ell|}.
\]
\end{enumerate}
\end{lemma}
\begin{proof}
The claim follows directly from the fact that,
by~\eqref{def:whalpha},
\[
e^{i\ell\wt \alpha_\h(v(\tau))} = \left(
\frac{i-\tau(v)}{i+\tau(v)}\right)^{\ell}
\]
and~\eqref{eq:singularitiesofthechangeofvariables}.
\end{proof}

\subsection{The invariant manifolds close to infinity}\label{sec:ManifoldsInfinity}
We devote this section to prove the existence of the invariant
manifolds  in the domains $D^{u,s}_{\infty, \rr}\times\TT_\sigma$,
where $D^{u,s}_{\infty, \rr}$ are the domains defined
in~\eqref{def:DomainInfinity}. To this end we will look for
solutions of equation~\eqref{eq:Q} satisfying the asymptotic
conditions~\eqref{def:DecayAtInfinity}. To obtain these solutions,
we set up a fixed point argument inverting the operator $\LL$
in~\eqref{def:Outer:DiffOperator} in suitable Banach spaces for
functions defined in $D_{\infty,\rr}^{u,s}\times\TT_\sigma$.  We do it for the unstable manifold. The existence of the stable one is given by the symmetry property \eqref{def:Symmetry:GeneratingFunction}.

We first define  norms for functions $h:D_{\infty,\rr}^{u}\rightarrow \CC$ as
\[
 \|h\|_{\nu}=\sup_{v\in D_{\infty,\rr}^{u}}\left|v^\nu h(v)\right|.
\]
Then, for functions $h:D_{\infty,\rr}^{u}\times\TT_\sigma\rightarrow \CC$, we define
\[
 \|h\|_{\nu,\sigma}=\sum_{\ell\in\ZZ}\left\|h^{[\ell]}\right\|_{\nu}e^{|\ell|\sigma},
\]
and the functional space
\[
\ZZZ_{\nu,\rr,\sigma}=\left\{h:D_{\infty,\rr}^{u}\times\TT_\sigma\rightarrow \CC:\, \text{real-analytic},\,\|h\|_{\nu,\sigma}<\infty\right\}.
\]
It can be checked that it is a Banach space for any fixed $\nu\geq
0$. Moreover, since equation~\eqref{eq:Q} involves both $Q$ and its
derivatives $\pa_v Q$ and $\pa_\xi Q$ we will need to deal with
Banach spaces which control at the same time the norms of a function
and its derivatives. To this end we define the norm
\[
 \lln h \rrn_{\nu,\sigma}= \|h\|_{\nu,\sigma}+ \|\pa_v h\|_{\nu+1,\sigma}+G_0^{3} \|\pa_\xi h\|_{\nu+1,\sigma},
\]
and the corresponding Banach space
\begin{equation}\label{def:Infinity:Banach:C1}
\wt \ZZZ_{\nu,\rr,\sigma}=\left\{h:D_{\infty,\rr}^{u}\times\TT_\sigma\rightarrow \CC:\, \text{real-analytic},\, \lln h \rrn_{\nu,\sigma}<\infty\right\}.
\end{equation}
Note that
\[
 \wt \ZZZ_{\nu,\rr,\sigma}\subset \ZZZ_{\nu,\rr,\sigma}.
\]
Once we have defined the suitable Banach spaces, we can state the main theorem of this section.
\begin{theorem}\label{thm:Infinity:ExistenceManifolds}
Fix constants $\rr_0>0$ and $\sigma_0>0$. Then, for $G_0>0$ big
enough, there exists a unique solution $Q^u$ of equation~\eqref{eq:Q}
in $\wt \ZZZ_{5/3,\rr_0,\sigma_0}$ satisfying~\eqref{def:DecayAtInfinity}.
Moreover, it satisfies
\[
 \lln Q^u\rrn_{5/3,\sigma_0}\leq b_0\mu G_0^{-6},
\]
for a constant $b_0>0$ independent of $\mu$ and $G_0$.

Furthermore, if we define the function
\begin{equation}\label{def:HalfMelnikov:unst}
 L_1^u(v,\xi)=\int_{-\infty}^0 \wh U_1 (v+s,\xi-G_0^3 s+\wt \al_\h(v+s)) \,ds,
\end{equation}
where $\wt\al_\h$ and $\wh U_1$ are defined
in~\eqref{def:unperturbedhomoclinic} and~\eqref{def:whU1},
respectively, we have that $L_1^u\in\wt \ZZZ_{5/3,\rr_0,\sigma_0}$,
satisfies
\begin{equation}
\label{HalfMelnikov:boundunstinfinity} \|L_1^u\|_{5/3,\sigma_0}\leq
K\mu G_0^{-6}
\end{equation}
and
\begin{equation}\label{def:outer:uns:ApproxHalfMelnikovinfinity}
 \|Q^u-L_1^u\|_{5/3,\sigma_0}\leq K\mu^2G_0^{-8}.
\end{equation}
\end{theorem}
The first step to prove this theorem is to  invert operator $\LL$ in
equation~\eqref{eq:Q} to set up a fixed point argument for $Q^u$. To
this end we consider the operator $\GG=\GG^u$ defined in
\eqref{def:Outer:IntegralOperator}. Since we are only dealing with the proof of existence of the parameterization of the unstable manifold, from now on, we omit the superindex $u$ wherever there is no risk of confusion.

Next lemma, whose proof is analogous to the one of Lemma 5.5 of \cite{GuardiaOS10}, gives properties
about how the operator $\GG$ acts on the Banach spaces
$\ZZZ_{\nu,\rr,\sigma}$.
\begin{lemma}\label{lemma:Infinity:PropertiesOfG}
The operator $\GG=\GG^u$, defined in~\eqref{def:Outer:IntegralOperator},
when considered acting on the spaces $\ZZZ_{\nu,\rr,\sigma}$ satisfies the
following properties.
\begin{enumerate}
\item For any $\nu>1$,  $\GG:\ZZZ_{\nu,\rr,\sigma}\rightarrow \ZZZ_{\nu-1,\rr,\sigma}$ is
well defined and linear continuous. Moreover,  $\LL\circ\GG=\mathrm{Id}$.
\item If $h\in\ZZZ_{\nu,\rr,\sigma}$ for some $\nu> 1$, then
\[
\left\|\GG(h)\right\|_{\nu-1,\rr,\sigma}\leq
K\|h\|_{\nu,\sigma}.
\]
\item If $h\in\ZZZ_{\nu,\rr,\sigma}$ for some $\nu\geq 1$, then
$\pa_v\GG(h)\in
\ZZZ_{\nu,\rr,\sigma}$ and
\[
\left\|\pa_v\GG(h)\right\|_{\nu,\sigma}\leq
K\|h\|_{\nu,\sigma}.
\]
\item If $h\in\ZZZ_{\nu,\rr,\sigma}$ for some $\nu\geq 1$, then
$\pa_\xi\GG(h)\in
\ZZZ_{\nu,\rr,\sigma}$ and
\[
\left\|\pa_\xi\GG(h)\right\|_{\nu,\sigma}\leq
KG_0^{-3}\|h\|_{\nu,\sigma}.
\]
\item From the previous three statements,  one can conclude that if
$h\in\ZZZ_{\nu,\rr,\sigma}$ for some $\nu> 1$, then $\GG(h)\in\wt \ZZZ_{\nu-1,\rr,\sigma}$ and
\[
 \lln \GG(h) \rrn_{\nu-1,\sigma}\leq
K \| h \|_{\nu,\sigma}.
\]
\end{enumerate}
\end{lemma}
We look for a fixed point of the operator
\begin{equation}\label{def:Outer:CompleteOperator}
\wt \FF=\GG\circ \FF
\end{equation}
where $\GG$ and $\FF$ are the operators defined
in~\eqref{def:Outer:IntegralOperator}
and~\eqref{def:Outer:OperatorF} respectively, in the
space~$\wt\ZZZ_{5/3,\rho_0,\sigma_0}$ defined
in~\eqref{def:Infinity:Banach:C1}.
Theorem~\ref{thm:Infinity:ExistenceManifolds} is a straightforward
consequence of the following proposition.

\begin{proposition}\label{prop:Infinity:FixedPoint}
 Let us fix $\rr_0>0$ and $\sigma_0>0$. Then, if $G_0$ is big enough there exists a constant~$b_0>0$
 such that the operator~$\wt\FF$ in~\eqref{def:Outer:CompleteOperator} has a fixed
 point~$Q^u\in B(b_0 \mu G_0^{-6})\subset\wt \ZZZ_{5/3,\rr_0,\sigma_0}$.
 Moreover, the function $L_1^u$ defined in~\eqref{def:HalfMelnikov:unst}
 satisfies~\eqref{HalfMelnikov:boundunstinfinity} and $Q^u$
 satisfies~\eqref{def:outer:uns:ApproxHalfMelnikovinfinity}.
\end{proposition}
Before proving this proposition, we state a technical lemma which
gives properties of the functions~$Q_0$ and~$\wh U_1$ defined
in~\eqref{def:Q0} and~\eqref{def:whU0}.

\begin{lemma}\label{lemma:Infinity:BoundsQ0andU1}
 The functions $Q_0$ and $\wh U_1$ satisfy that $Q_0 \in\wt\ZZZ_{1,\rr_0,\sigma_0}$ and
 $\wh U_1\circ A \in \ZZZ_{8/3,\rr_0,\sigma_0}$, where $A$ is the map defined in~\eqref{def:mapA}. Moreover,
\[
 \begin{split}
  \lln Q_0\rrn_{1,\sigma_0}&\leq K\mu G_0^{-4}\\
  \| \wh U_1 \circ A\|_{8/3,\sigma_0}&\leq K\mu G_0^{-6}.
 \end{split}
\]
\end{lemma}
\begin{proof}
For the statements referred to $Q_0$, it is enough to use the bound
for $\wh U_0$ in Lemma~\ref{lemma:whUminuswhU0bound}
 and
Lemma~\ref{lemma:Infinity:PropertiesOfG}. The statements for $\wh
U_1\circ A$ follow from Corollary~\ref{lem:whUsingularites} and
Lemma~\ref{lem:eikalphasingularities}.
\end{proof}

\begin{proof}[Proof of Proposition~\ref{prop:Infinity:FixedPoint}]
First, we see that $\wt\FF=\GG\circ\FF$ is well defined from $\wt
\ZZZ_{5/3,\rr_0,\sigma_0}$ to itself. Indeed if $h\in \wt
\ZZZ_{5/3,\rr_0,\sigma_0}$, using Lemma~\ref{lemma:Infinity:BoundsQ0andU1}
and the properties of the separatrix parameterization given in
Corollary~\ref{coro:HomoInfinity}, one can easily see that
$\FF(h)\in\ZZZ_{8/3,\rr_0,\sigma_0}$. Therefore, using the last statement
of Lemma~\ref{lemma:Infinity:PropertiesOfG}, we obtain that
$\GG\circ\FF(h)\in \wt \ZZZ_{5/3,\rr_0,\sigma_0}$.

Now we bound $\wt\FF(0)=\GG\circ\FF(0)$. By the definition of~$\FF$
in~\eqref{def:Outer:OperatorF}, we have that
\begin{equation}\label{def:Infinity:F(0)}
 \FF(0)= \frac{1}{2 \wt y_\h^2}\left(\pa_v Q_0-\frac{1}{\wt
r_\h^2} \pa_\xi Q_0\right)^2-\frac{1}{2\wt r_\h^2}\left(\pa_\xi
Q_0\right)^2+ \wh U_1\circ A.
\end{equation}
Thus, applying Lemma~\ref{lemma:Infinity:BoundsQ0andU1}
and Corollary~\ref{coro:HomoInfinity}, we have that $\FF(0)\in
\ZZZ_{8/3,\rr_0,\sigma_0}$ and satisfies $\|\FF(0)\|_{8/3,\sigma_0}\leq
K\mu G_0^{-6}$. Then, applying the last statement of
Lemma~\ref{lemma:Infinity:PropertiesOfG}, we obtain that there
exists a constant $b_0>0$ such that
\[
\lln \wt \FF(0)\rrn_{5/3,\sigma_0}\leq \frac{b_0}{2}\mu G_0^{-6}.
\]
Now we have to prove that the operator $\wt\FF$ is contractive in $B(b_0 \mu
G_0^{-6})\subset \wt \ZZZ_{5/3,\rr_0,\sigma_0}$. Take $h_1,h_2\in B(b_0 \mu
G_0^{-6})\subset \wt \ZZZ_{5/3,\rr_0,\sigma_0}$.  Using the last statement of
Lemma~\ref{lemma:Infinity:PropertiesOfG}, we have that
\[
\lln\wt \FF(h_2)-\wt \FF(h_1)\rrn_{5/3,\sigma_0}\leq K\left\|\FF(h_2)-\FF(h_1)\right\|_{8/3,\sigma_0}.
\]
Thus, we only need to bound the right hand side of this formula. To this end, we write it as
\begin{equation}\label{eq:Infinity:Lipschitz}
 \begin{split}
  \FF(h_2)-\FF(h_1)=& \frac{1}{2 \wt y_\h^2}\left(\left(2\pa_v Q_0+\pa_v h_1+\pa_v h_2\right)
  -\frac{1}{ \wt r_\h^2}(2\pa_\xi Q_0+\pa_\xi h_1+\pa_\xi h_2)\right)\left(\pa_v h_2-\pa_v h_1\right)\\
&- \frac{1}{2 \wt y_\h^2\wt r^2}\left(\left(2\pa_v Q_0+\pa_v h_1+\pa_v
h_2\right)-\frac{1}{ \wt r_\h^2}
(2\pa_\xi Q_0+\pa_\xi h_1+\pa_\xi h_2)\right)\left(\pa_\xi h_2-\pa_\xi h_1\right)\\
&+\frac{1}{2 \wt r_\h^2}\left(2\pa_\xi Q_0+\pa_\xi h_1+\pa_\xi
h_2\right)\left(\pa_\xi h_2-\pa_\xi h_1\right).
 \end{split}
\end{equation}
Using the bounds for $Q_0$ obtained in
Lemma~\ref{lemma:Infinity:BoundsQ0andU1},
Corollary~\ref{coro:HomoInfinity} and that $h_1,h_2\in B(b_0 \mu
G_0^{-6})\subset \wt \ZZZ_{5/3,\rr_0,\sigma_0}$, we obtain that
\[
\left\|\FF(h_2)-\FF(h_1)\right\|_{8/3,\sigma_0}\leq K\mu
G_0^{-4}\lln h_2-h_1 \rrn_{5/3,\sigma_0}.
\]
Then, since $h_1,h_2\in \wt \ZZZ_{5/3,\rr_0,\sigma_0}$, we can conclude that
\[
\lln\wt \FF(h_2)-\wt \FF(h_1)\rrn_{5/3,\sigma_0}\leq
K\left\|\FF(h_2)-\FF(h_1)\right\|_{8/3,\sigma_0}\leq K\mu
G_0^{-4}\lln h_2-h_1 \rrn_{5/3,\sigma_0},
\]
and therefore, taking $G_0$ big enough, it is contractive and has a
unique  fixed point in~$B(b_0 \mu G_0^{-6})\subset\wt
\ZZZ_{5/3,\rr_0,\sigma_0}$. This completes the proof of the first claim.

Since $L_1^u = \GG (\wh U_1 \circ A)$,
inequality~\eqref{HalfMelnikov:boundunstinfinity} follows from
Lemma~\ref{lemma:Infinity:BoundsQ0andU1} and the second statement of
Lemma~\ref{lemma:Infinity:PropertiesOfG}.

Finally, to prove
inequality~\eqref{def:outer:uns:ApproxHalfMelnikovinfinity} we note
that, since $Q^u$ is a fixed point of $\wt \FF$,
\begin{equation}
\label{bound:QuminusLu}
 \|Q^u-L_1^u\|_{5/3,\sigma_0} \le \|\wt \FF(Q^u)-\wt \FF(0)\|_{5/3,\sigma_0} +
\|\wt \FF(0)-L_1^u\|_{5/3,\sigma_0}.
\end{equation}
The first term in the right hand side can be bounded using that $\wt
\FF$ is Lipschitz,
\[
\|\wt \FF(Q^u)-\wt \FF(0)\|_{5/3,\sigma_0} \le \lln \wt \FF(Q^u)-\wt
\FF(0)\rrn_{5/3,\sigma_0} \le K\mu G_0^{-4} \lln Q^u
\rrn_{5/3,\sigma_0} \le K  \mu^2 G_0^{-10}.
\]
To obtain a bound of the second term in the right hand side
of~\eqref{bound:QuminusLu} we first note that,
by~\eqref{def:Infinity:F(0)} and
Lemma~\ref{lemma:Infinity:BoundsQ0andU1},
\[
\|\FF(0) - \wh U_1 \circ A\|_{8/3,\sigma_0} = \left\|\frac{1}{2 \wt
y_\h^2}\left(\pa_v Q_0-\frac{1}{\wt r_\h^2} \pa_\xi Q_0\right)^2
+\frac{1}{2\wt r_\h^2}\pa_\xi Q_0^2 \right\|_{8/3,\sigma_0} \le K^2
\mu^2 G_0^{-8}.
\]
The claim follows then applying the second statement of
Lemma~\ref{lemma:Infinity:PropertiesOfG} to $\wt \FF(0)-L_1^u = \GG
(\FF(0) - \wh U_1 \circ A)$.
\end{proof}

\subsection{The invariant manifolds close to the singularities}\label{sec:ManifoldsOuter}

As we have explained in
Section~\ref{sec:ExistenceManifolds:Description}, once one has
obtained the existence of the invariant manifolds close to infinity,
the usual next step when one wants to study the exponentially small
splitting of separatrices phenomena is to extend the
parameterizations of the invariant manifolds to the domains
$D^{u,s}_{\rr,\kk,\de}$ defined in~\eqref{def:DomainOuter}. That is,
up to points at  distance $\OO(G_0^{-3})$ of the singularities of
the separatrix~\eqref{def:unperturbedhomoclinic} at $v= \pm i/3$.
This is not possible in this problem since the parameterizations of
the invariant manifolds blow up before reaching these points.
Indeed, if one looks at the analytic continuation of the function
$L_1^{u}(v,\xi)=\sum_{\ell\in\ZZ}L_1^{[\ell]}(v)e^{ik\xi}$ defined
in \eqref{def:HalfMelnikov:unst}, which is the first order
approximation of the function  $Q^u(v,\xi)$ obtained in
Proposition~\ref{prop:Infinity:FixedPoint}, one can see that its
Fourier coefficients $L_1^{[\ell]}$ grow exponentially in $\ell$.
Therefore, one cannot obtain this analytic extension and neither the
one of $Q^u$. We overcome this difficulty analytically extending the
Fourier coefficients $Q^{[\ell]}$ instead of the function $Q^u$. It
turns out that this is sufficient to study the exponentially small
splitting of separatrices. To this end, we consider the Fourier
coefficients of $Q^u$ as a sequence of functions and we study their
extension in a Banach space of sequence of functions. We endow the
Banach space with a weighted norm that allows us to obtain good
estimates for each Fourier coefficient.



The structure of this section goes as follows. First in
Section~\ref{sec:outer:WeightedFourierBanach} we define the Banach
spaces for the Fourier coefficients and their sequences. We also
show that these spaces have an algebra-like structure with respect
to the classical product of Fourier series. Note that the ``Fourier
series'' we are dealing with are formal and therefore this algebra
structure is not obvious. Then, in
Section~\ref{sec:Outer:FixedPoint} we set up a fixed point argument
and prove the existence of the Fourier coefficients of the
parameterization of the invariant manifolds in the
domains~$D^{u,s}_{\rr,\kk,\de}$. We deal only with the unstable
manifold since the existence of the stable one is given by the
symmetry~\eqref{def:Symmetry:GeneratingFunction}.

\subsubsection{Weighted Fourier norms and Banach spaces}\label{sec:outer:WeightedFourierBanach}
We devote this section to study the weighted Fourier norms. To this
end we first define norms for the Fourier coefficients, that is, for
functions $h:D^u_{\rr,\kk,\de}\rightarrow \CC$. We define
\begin{equation}\label{def:Norma:Weighted}
 \|h\|_{\nu_-,\nu_+}=\sup_{v\in
D^{u}_{\rr,\kk,\de}}\left|(v-i/3)^{\nu_+}(v+i/3)^{\nu_-} h(v)\right|
\end{equation}
and the corresponding Banach space
\begin{equation}\label{def:BanachCoefs}
\XX_{\nu_-,\nu_+,\rr,\kk,\de}=\left\{h:D^{u}_{\rr,\kk,\de}\rightarrow \CC: \text{analytic}, \|h\|_{\nu_-,\nu_+}<\infty  \right\}.
\end{equation}
Next step is to define a weighted norm for sequences of functions.
To denote a sequence $\{h^{[\ell]}\}_{\ell\in\ZZ}$ we keep the
Fourier series notation
\[
 h(v,\xi)=\sum_{\ell\in \ZZ}h^{[\ell]}(v) e^{i\ell\xi}
\]
but we want to stress that this series is a \emph{formal series} and
that $h(v,\xi)$ is not necessarily a function. That is, each Fourier
coefficient is an analytic function defined in $D^{u}_{\rr,\kk,\de}$
but its sum does not need to be  convergent for any
$\xi\in\TT_\sigma$. We define the following norm for these sequences
of Fourier coefficients\footnote{An equivalent norm is defined by
\[
\|h\|_{\nu,\sigma}=\sum_{\ell\in\ZZ}\left\|h^{[\ell]} e^{-i \ell
\wt \alpha_\h} \right\|_{\nu} e^{|\ell|\sigma}, \quad \text{ where }
\quad \|h\|_{\nu} = \sup_{v\in
D^{u}_{\rr,\kk,\de}}\left|(v-i/3)^{\nu}(v+i/3)^{\nu} h(v)\right|
\] and $\wt\al_\h$ has been defined in~\eqref{def:unperturbedhomoclinic}.
}
\begin{equation}\label{def:Norma:Weighted-Fourier}
 \|h\|_{\nu,\sigma}=\sum_{\ell\in\ZZ}\left\|h^{[\ell]}\right\|_{\nu+\ell/2,\nu-\ell/2} e^{|\ell|\sigma}
\end{equation}
and the corresponding Banach space
\begin{equation}\label{def:BanachFourier:0}
\YY_{\nu,\rr,\kk,\de,\sigma}=\left\{h(v,\xi)=\sum_{\ell\in \ZZ}h^{[\ell]}(v)
e^{i\ell\xi}: h^{[\ell]}\in \XX_{\nu+\ell/2,\nu-\ell/2,\rr,\kk,\de},
\|h\|_{\nu,\kk,\sigma}<\infty\right\}.
\end{equation}
Next lemma gives algebra properties for these Banach spaces. Its proof is straightforward.

\begin{lemma}\label{lemma:banach:AlgebraProps}
The spaces $\YY_{\nu,\rr,\kk,\de,\sigma}$ satisfy the following properties:
\begin{itemize}
 \item If $h\in \YY_{\nu,\rr,\kk,\de,\sigma}$ and $g\in \YY_{\eta,\rr,\kk,\de,\sigma}$, then the formal product of Fourier series $hg$ defined as usual by
\[
 (hg)^{[\ell]}(v)=\sum_{k\in\ZZ}h^{[k]} g^{[\ell-k]}
\]
satisfies that $hg\in \YY_{\nu+\eta,\rr,\kk,\de,\sigma}$ and
\[
 \|hg\|_{\nu+\eta,\sigma}\leq \|h\|_{\nu,\sigma}\|g\|_{\eta,\sigma}.
\]
\item If $h\in \YY_{\nu,\rr,\kk,\de,\sigma}$, then $h\in \YY_{\nu+\eta,\rr,\kk,\de,\sigma}$ with $\eta>0$ and
\[
 \|h\|_{\nu+\eta,\sigma}\leq K\|h\|_{\nu,\sigma}.
\]
\item If $h\in \YY_{\nu,\rr,\kk,\de,\sigma}$, then $h\in \YY_{\nu-\eta,\rr,\kk,\de,\sigma}$ with $\eta>0$ and
\[
 \|h\|_{\nu-\eta,\sigma}\leq K G_0^{3\eta}\|h\|_{\nu,\sigma}.
\]
\end{itemize}
\end{lemma}

In Section~\ref{sec:ManifoldsInfinity}, we  needed to control at the
same time the size of functions $h\in\ZZZ_{\nu,\rr,\sigma}$ and the
size of their derivatives $\pa_v h$ and $\pa_\xi h$. This fact has
prompted us to deal with the Banach
spaces~\eqref{def:Infinity:Banach:C1}. To perform the Extension
Theorem~\ref{thm:outer:ExistenceManifolds} we will also need to
control the size of the derivatives for sequences of Fourier
coefficients $h\in\YY_{\nu,\rr,\kk,\de,\sigma}$. The derivatives of
sequences  are defined in the natural way
\begin{equation}\label{def:DerivativeFourierSeries}
 \begin{split}
   \pa_v h(v,\xi)&=\sum_{\ell\in \ZZ}\pa_v  h^{[\ell]}(v) e^{i\ell\xi},\\
 \pa_\xi h(v,\xi)&=\sum_{\ell\in \ZZ}(i\ell) h^{[\ell]}(v) e^{i\ell\xi}.
 \end{split}
\end{equation}
Thus, we define the Banach space of formal Fourier series
\[
\wt \YY_{\nu,\rr,\kk,\de,\sigma}=\left\{h(v,\xi)=\sum_{\ell\in
\ZZ}h^{[\ell]}(v) e^{i\ell\xi}: h^{[\ell]}\in
\XX_{\nu+\ell/2,\nu-\ell/2,\rr,\kk,\de}, \lln
h\rrn_{\nu,\sigma}<\infty\right\},
\]
where
\[
 \lln h\rrn_{\nu,\sigma}=\| h\|_{\nu,\sigma}+\| \pa _vh\|_{\nu+1,\sigma}+G_0^3\|\pa_\xi h\|_{\nu+1,\sigma}.
\]
\subsubsection{The fixed point argument}\label{sec:Outer:FixedPoint}
In Theorem~\ref{thm:Infinity:ExistenceManifolds} we have obtained
the existence of the function $Q^u$ in the domain
$D^{u}_{\infty,\rr_0}\times\TT_{\sigma_0}$ and, therefore, we have
obtained the existence of its Fourier coefficients $Q^{[\ell]}$ in
the domain $D^{u}_{\infty,\rr_0}$. We devote this section to obtain
the analytic continuation of these Fourier coefficients to the
domain $D^u_{\rr_1,\kk,\de}$ defined in~\eqref{def:DomainOuter}, where
we choose $\rr_1>\rr_0$ so that the domains $D^{u}_{\infty,\rr_0}$
and $D^u_{\rr_1,\kk,\de}$ overlap. The main theorem of this section is
the following.

\begin{theorem}\label{thm:outer:ExistenceManifolds}
Consider $\rr_0$ and $\sigma_0$ the constants given by
Theorem~\ref{thm:Infinity:ExistenceManifolds}, $\rr_1>\rr_0$,
$\kk_0,\de_0>0$ and the function  $Q^u$ obtained in Theorem
\ref{thm:Infinity:ExistenceManifolds}. Then, for $G_0$ big enough,
its Fourier coefficients  can be analytically extended  to
$D^u_{\rr_1,\kk_0,\de_0}$ and the sequence given by the Fourier series
\[
 Q^u(v,\xi)=\sum_{\ell\in\ZZ}Q^{[\ell]}(v)e^{i\ell\xi}
\]
satisfies that $Q^u\in \wt \YY_{1,\rr_1,\kk_0,\de_0,\sigma_0}$ and
\[
 \lln Q^u\rrn_{1,\sigma_0}\leq b_1\mu G_0^{-6}
\]
for a constant $b_1>0$ independent of $\mu$ and $G_0$.

Moreover, the Fourier coefficients of the function $L_1^u$ defined
in~\eqref{def:HalfMelnikov:unst} can be analytically extended to $D_{\rr_1,\kk_0,\de_0}^u$ and its sequence belongs to $\YY_{1,\rr_1,\kk_0,\de_0,\sigma_0}$. Furthermore, it satisfies
\begin{equation}
\label{bound:HalfMelnikov:unst} \|L_1^u\|_{1,\sigma_0}\leq K\mu
G_0^{-6}
\end{equation}
and
\begin{equation}\label{def:outer:uns:ApproxHalfMelnikov}
 \|Q^u-L_1^u\|_{1,\sigma_0}\leq K\mu^2G_0^{-8}.
\end{equation}
\end{theorem}

We devote the rest of the section to prove this theorem. To this end,
we derive a fixed point argument from equation~\eqref{eq:Q}. We
follow the approach considered in~\cite{BaldomaFGS11}. We
consider an operator~$\wt\GG$ which acts on $\YY_{\nu,\rr,\kk,\de,\sigma}$ and is  a left inverse of the operator~$\LL$
in~\eqref{def:Outer:DiffOperator}. $\wt\GG$ is  defined acting on
the Fourier coefficients. We consider $v_1, \bar v_1\in \CC$ the
vertices of the domain $D^{u}_{\rr,\kk,\de}$ (see
Figure~\ref{fig:OuterDomains}, the bar denotes the complex conjugate). Then, we define $\wt\GG$ as
\begin{equation}\label{def:Outer:operadorG}
\wt\GG(h)(v,\xi)=\sum_{\ell\in\ZZ}\wt\GG(h)^{[\ell]}(v)e^{i\ell\xi},
\end{equation}
where its Fourier coefficients are given by
\begin{align*}
\dps\wt\GG(h)^{[\ell]}(v)&= \int_{v_1}^v e^{i\ell G_0^3
(v-t)}h^{[\ell]}(t)\,dt& \text{ for }\ell< 0\\
\dps\wt\GG(h)^{[0]}(v)&=\int_{- \rr}^v h^{[0]}(t)\,dt&
\\
\dps\wt\GG (h)^{[\ell]}(v)&=\int_{\bar v_1}^v e^{i\ell G_0^3
(v-t)}h^{[\ell]}(t)\,dt& \text{ for }\ell>0.\\
\end{align*}
Observe that the definition of the operator $\wt\GG$ depends on the
domain, since  it involves the vertices $v_1$, $\bar v_1$ and also
$\rr$. Next lemma, whose proof is analogous to Lemma~5.5
in~\cite{GuardiaOS10}, gives properties of this operator acting on
the space $\YY_{\nu,\rr,\kk,\de,\sigma}$.
\begin{lemma}\label{lemma:Outer:PropertiesOfG}
The operator $\wt\GG$  in~\eqref{def:Outer:operadorG} satisfies the
following properties.
\begin{enumerate}
\item If $h\in \YY_{\nu,\rr,\kk,\de,\sigma}$ for some $\nu\geq 0$, then $\wt\GG(h)\in
\YY_{\nu,\rr,\kk,\de,\sigma}$, $\LL\circ\wt \GG (h)=h$ and
\[
\|\wt \GG(h)\|_{\nu,\sigma}\leq K\|\wt h\|_{\nu,\sigma}.
\]
\item If $h\in\YY_{\nu,\rr,\kk,\de,\sigma}$ for some $\nu>1$, then $\wt \GG(h)\in
\YY_{\nu-1,\rr,\kk,\de,\sigma}$ and
\[
\left\|\wt \GG(h)\right\|_{\nu-1,\sigma}\leq K\|h\|_{\nu,\sigma}.
\]
\item If $h\in\YY_{\nu,\rr,\kk,\de,\sigma}$ for some $\nu\geq 1$, then $\pa_v\wt \GG(h)\in
\YY_{\nu,\rr,\kk,\de,\sigma}$ and
\[
\left\|\pa_v\wt \GG(h)\right\|_{\nu,\sigma}\leq K\|h\|_{\nu,\sigma}.
\]
\item If $h\in\YY_{\nu,\rr,\kk,\de,\sigma}$ for some $\nu\geq 1$, then $\pa_\xi\wt \GG(h)\in
\YY_{\nu,\rr,\kk,\de,\sigma}$ and
\[
\left\|\pa_\xi\wt \GG(h)\right\|_{\nu,\sigma}\leq KG_0^{-3}\|h\|_{\nu,\sigma}.
\]
\item From the previous three statements, one can conclude that if $h\in\YY_{\nu,\rr,\kk,\de,\sigma}$ for some $\nu>1$, then $\wt \GG(h)\in
\wt \YY_{\nu-1,\rr,\kk,\de,\sigma}$ and
\[
\lln\wt \GG(h)\rrn_{\nu-1,\sigma}\leq K\|h\|_{\nu,\sigma}.
\]
\end{enumerate}
\end{lemma}
We use the operator $\wt \GG$ to set up a fixed point argument from
equation~\eqref{eq:Q}. Nevertheless, since we want to obtain the
analytic continuation of the Fourier coefficients obtained in
Theorem~\ref{thm:Infinity:ExistenceManifolds}, we impose certain
initial conditions. Since the operator~$\wt\GG$ involves
integration from different initial points depending on the harmonic,
we impose different initial conditions for each harmonic. We
define
\[
 F(v,\xi)=\sum_{\ell<0}e^{i\ell G_0^3 (v-v_1)}Q^{[\ell]}(v_1)e^{i\ell\xi}+Q^{[0]}(-\rr_1)
 +\sum_{\ell>0}e^{i\ell G_0^3 (v-\ol v_1)}Q^{[\ell]}(\ol v_1)e^{i\ell\xi}.
\]
Recall that, since
$\rr_1>\rr_0$,  $v_1,\ol v_1,-\rr_1\in D^{u}_{\infty,\rr_0}$  and therefore, by
Theorem~\ref{thm:Infinity:ExistenceManifolds}, $F$  is
already known. Moreover, it is straightforward to see that $F\in \wt
\YY_{0,\rr_1,\kk_0,\de_0,\sigma_0}$ and
\begin{equation}\label{def:outer:InitialConditionBound}
 \lln F\rrn_{0,\sigma_0}\leq K\mu G_0^{-6}.
\end{equation}
With this prescribed initial condition, it can be seen that the
solutions of equation
\[
 Q(v,\xi)=F(v,\xi)+\wt\GG \circ \FF(Q)(v,\xi),
\]
understood as an equation of (not necessarily convergent) Fourier
series with analytic coefficients, give the analytic continuation of
the Fourier coefficients $Q^{[\ell]}(v)$ obtained in
Proposition~\ref{prop:Infinity:FixedPoint}. Note that it is not
obvious that this equation is well defined for $Q\in \wt
\YY_{\nu,\rr_1,\kk_0,\de_0,\sigma_0}$ for any $\nu>0$. Nevertheless, the algebra
properties stated in Lemma~\ref{lemma:banach:AlgebraProps} ensure
that it is the case.

We look for a fixed point of the operator
\[
\SSS(h)=F+\wt\GG \circ \FF(h).
\]
Next proposition gives its existence and uniqueness. From it, using
that the function $Q$ of
Theorem~\ref{thm:Infinity:ExistenceManifolds} is also a fixed point
of $\SSS$ in the overlapping domain,
one deduces Theorem~\ref{thm:outer:ExistenceManifolds}.
\begin{proposition}\label{prop:outer:fixedpoint}
Consider $\rr_0$ and $\sigma_0$ the constants given by
Theorem~\ref{thm:Infinity:ExistenceManifolds}, $\rr_1>\rr_0$,
$\kk_0>0$ and $\de_0>0$. Then, for $G_0$ big enough, there exists a constant
$b_1>0$ independent of $G_0$ and $\mu$ such that the operator $\SSS$
has a unique fixed point $Q^u\in B(b_1\mu G_0^{-6})\subset\wt
\YY_{1,\rr_1,\kk_0,\de_0,\sigma_0}$.

Moreover, the (not necessarily convergent) Fourier series $L_1^u$ in~\eqref{def:HalfMelnikov:unst}
satisfies~\eqref{bound:HalfMelnikov:unst} and $Q^u$
satisfies~\eqref{def:outer:uns:ApproxHalfMelnikov}.
\end{proposition}

To prove this proposition we first state the following technical lemma.
\begin{lemma}\label{lemma:outer:boundsU1andQ0}
The functions $Q_0$ and $\wh U_1$, defined in~\eqref{def:Q0} and
\eqref{def:whU1} respectively, satisfy that  $Q_0\in
\wt\YY_{1/2,\rr_1,\kk_0,\de_0,\sigma_0}$ and  $\wh U_1\circ A\in
\YY_{2,\rr_1,\kk_0,\de_0,\sigma_0}$, where $A$ is the map
in~\eqref{def:mapA}.  Moreover,
\[
 \begin{split}
  \lln Q_0\rrn_{1/2,\sigma_0}&\leq K\mu G_0^{-4}\\
  \| \wh U_1 \circ A\|_{2,\sigma_0}&\leq K\mu G_0^{-6}.
 \end{split}
\]
\end{lemma}

\begin{proof}
For the first bound, we can write the function $Q_0$
in~\eqref{def:Q0} as
\[
 Q_0(v,\xi)=\wt Q_0(v,\xi)+\wt\GG\left(\wh U_0\circ A\right)(v,\xi)
\]
where
\[
\wt Q_0(v,\xi)= \sum_{\ell<0}e^{i\ell G_0^3
(v-v_1)}{Q_0}^{[\ell]}(v_1)e^{i\ell\xi}+{Q_0}^{[0]}(-\rr_1)
 +\sum_{\ell>0}e^{i\ell G_0^3 (v-\ol v_1)}{Q_0}^{[\ell]}(\ol
 v_1)e^{i\ell\xi}.
\]
We bound the two terms. For the first one, let us point out that
$v_1, \ol v_1,-\rr_1\in D_{\infty,\rr_0}^u$ and therefore we can
bound the function $Q_0$ in these points thanks to
Lemma~\ref{lemma:Infinity:BoundsQ0andU1}. Then, we can deduce  for
the function $\wt Q_0$ the  bound $ \lln \wt
Q_0\rrn_{1/2,\sigma_0}\leq K\mu G_0^{-4}$. For the second
term, we use Corollary~\ref{lem:whUsingularites} and Lemma
\ref{lem:eikalphasingularities} to obtain  $\| \wh U_0 \circ
A\|_{3/2,\sigma_0}\leq K\mu G_0^{-4}$. Then, applying Lemma
\ref{lemma:Outer:PropertiesOfG} we obtain that $ \lln \wt\GG(\wh
U_0\circ A)\rrn_{1/2,\sigma_0}\leq K\mu G_0^{-4}$.

For the second bound it is enough to apply
Corollary~\ref{lem:whUsingularites} and Lemma
\ref{lem:eikalphasingularities}.
\end{proof}

\begin{proof}[Proof of Proposition~\ref{prop:outer:fixedpoint}]
First, we see that $\SSS$ is well defined from $\wt
\YY_{1,\rr_1,\kk_0,\de_0,\sigma_0}$ to itself. Indeed if $h\in \wt
\YY_{1,\rr_1,\kk_0,\de_0,\sigma_0}$, using
Lemmas~\ref{lemma:banach:AlgebraProps}
and~\ref{lemma:outer:boundsU1andQ0} and the properties of the
separatrix parameterization given
in~\eqref{def:unperturbedhomoclinic}, one can easily see that
$\FF(h)\in \YY_{2,\rr_1,\kk_0,\de_0,\sigma_0}$, where $\FF$ is the operator
defined in~\eqref{def:Outer:OperatorF}. Therefore using the last
statement of Lemma~\ref{lemma:Outer:PropertiesOfG}, we obtain that
$\wt\GG\circ\FF(h)\in \wt \YY_{1,\rr_1,\kk_0,\de_0,\sigma_0}$. Then, using
also~\eqref{def:outer:InitialConditionBound}, we can deduce that
$\SSS(h)\in \wt \YY_{1,\rr_1,\kk_0,\de_0,\sigma_0}$.

We bound the first iteration
\[
\SSS(0)(v,\xi)=F(v,\xi)+\wt \GG\circ\FF(0)(v,\xi).
\]
Recall that $\FF(0)$ has been defined in~\eqref{def:Infinity:F(0)}.
Applying Lemmas~\ref{lemma:banach:AlgebraProps}
and~\ref{lemma:outer:boundsU1andQ0}  and
Corollary~\ref{coro:HomoSing}, it is easy to see that $\FF(0)\in
\YY_{2,\sigma_0}$ and $\|\FF(0)\|_{2,\sigma_0}\leq K\mu
G_0^{-6}$. Then, applying the last statement of
Lemma~\ref{lemma:Outer:PropertiesOfG} and taking into
account~\eqref{def:outer:InitialConditionBound}, we obtain that
there exists a constant $b_1>0$ such that
\[
 \|\SSS(0)\|_{1,\sigma_0}\leq \frac{b_1}{2}\mu G_0^{-6}.
\]
Now we prove that the operator~$\SSS$ is contractive in $B(b_1 \mu
G_0^{-6})\subset \wt \YY_{1,\rr_1,\kk_0,\de_0,\sigma_0}$. Take $h_1,h_2\in B(b_1 \mu
G_0^{-6})\subset \wt \YY_{1,\rr_1,\kk_0,\de_0,\sigma_0}$ and recall that
\[
 \SSS(h_2)-\SSS(h_1)=\wt \GG\left(\FF(h_2)-\FF(h_1)\right).
\]
We start by bounding $\FF(h_2)-\FF(h_1)$. To this end, we use
formula~\eqref{eq:Infinity:Lipschitz}. Then, using
Lemmas~\ref{lemma:banach:AlgebraProps}
and~\ref{lemma:outer:boundsU1andQ0}, the properties of the
separatrix parameterization given in Corollary~\ref{coro:HomoSing} and
that $h_1,h_2\in B(b_0 \mu G_0^{-6})\subset \wt
\YY_{1,\rr_1,\kk_0,\de_0,\sigma_0}$, we obtain
\[
\left\|\FF(h_2)-\FF(h_1)\right\|_{2,\sigma_0}\leq K\mu
G_0^{-5/2}\lln h_2-h_1 \rrn_{1,\sigma_0}.
\]
Then, since $h_1,h_2\in  \wt \YY_{1,\rr_1,\kk_0,\de_0,\sigma_0}$ and using
Lemma~\ref{lemma:Outer:PropertiesOfG}, we can conclude that
\[
\lln\SSS(h_2)-\SSS(h_1)\rrn_{1,\sigma_0}\leq K \|
\FF(h_2)-\FF(h_1) \|_{1,\sigma_0}\leq K\mu G_0^{-5/2}\lln
h_2-h_1 \rrn_{1,\sigma_0},
\]
and therefore, taking $G_0$ large enough, it is contractive and has
a unique fixed point in  $B(b_1 \mu G_0^{-6})\subset\wt
\YY_{1,\rr_1,\kk_0,\de_0,\sigma_0}$. This completes the first part of the
proposition.

To prove the second part, we  introduce
\[
F_1(v,\xi) = \sum_{\ell<0}e^{i\ell G_0^3
(v-v_1)}{L_1^u}^{[\ell]}(v_1)e^{i\ell\xi}+{L_1^u}^{[0]}(-\rr_1)
 +\sum_{\ell>0}e^{i\ell G_0^3 (v-\ol v_1)}{L_1^u}^{[\ell]}(\ol
 v_1)e^{i\ell\xi},
\]
where $L_1^u$ was defined in~\eqref{def:HalfMelnikov:unst}.
By~\eqref{HalfMelnikov:boundunstinfinity}, $F_1 \in \wt
\YY_{1,\rr_1,\kk_0,\de_0,\sigma_0}$ and
\[
 \lln F_1\rrn_{1,\sigma_0}\leq K\mu G_0^{-6}.
\]
Then, since their Fourier coefficients coincide, we have that
\[
L_1^u = F_1 + \wt \GG(\wh U_1 \circ A).
\]
Hence, inequality~\eqref{bound:HalfMelnikov:unst} follows from
the estimate for $F_1$ just obtained,
Lemma~\ref{lemma:outer:boundsU1andQ0} and the second statement of
Lemma~\ref{lemma:Outer:PropertiesOfG}.

Finally, to prove
inequality~\eqref{def:outer:uns:ApproxHalfMelnikov} we note that,
since $Q^u$ is a fixed point of $\SSS$,
\[
\|Q^u-L_1^u\|_{1,\sigma_0} \le \|\SSS(Q^u)-\SSS(0)\|_{1,\sigma_0} +
\|\SSS(0)-L_1^u\|_{1,\sigma_0}.
\]
The first term in the right hand side can be bounded using that
$\SSS$ is Lipschitz,
\[
\begin{split}
\|\SSS(Q^u)-\SSS(0)\|_{1,\sigma_0} \le &\lln
\SSS(Q^u)-\SSS(0)\rrn_{1,\sigma_0}\\
 \le& K\mu G_0^{-5/2} \lln Q^u
\rrn_{1,\sigma_0} \\\le& K b_1 \mu^2 G_0^{-17/2}.
\end{split}
\]
To obtain a bound of the second term in the right hand side
 we note that
\[
\SSS(0)-L_1^u = F-F_{1}+\wt\GG (\FF(0) - \wh U_1 \circ A).
\]
Then,
by~\eqref{def:outer:uns:ApproxHalfMelnikovinfinity},
\[
\|F - F_1\|_{1,\sigma_0} \le K^2 \mu^2 G_0^{-8}.
\]
and, by Lemma~\ref{lemma:outer:boundsU1andQ0},
\[
\|\FF(0) - \wh U_1 \circ A\|_{2,\sigma_0} = \left\|\frac{1}{2
\wt y_\h^2}\left(\pa_v Q_0-\frac{1}{\wt r_\h^2} \pa_\xi Q_0\right)^2
+\frac{1}{2\wt r_\h^2}\pa_\xi Q_0^2 \right\|_{2,\sigma_0} \le
K^2 \mu^2 G_0^{-8}.
\]
The claim follows then applying the second statement of
Lemma~\ref{lemma:Outer:PropertiesOfG} to $\wt\GG (\FF(0) - \wh U_1
\circ A)$.
\end{proof}

\subsection{Extension of the parameterization of the unstable manifold by the flow}\label{sec:InvManifold:ExtensionUnstable}
From Theorem~\ref{thm:outer:ExistenceManifolds} we  have  a formal parameterization of the  unstable invariant
manifold as not necessarily convergent Fourier series in the domain
$D^{u,s}_{\rr_1,\kk_0,\de_0}$.
This parameterization is given by
\begin{equation}
\label{eq:parameterizationsbyHamiltonJacobi} \wt \Gamma^u(v,\xi) =
\begin{pmatrix}
\wt r \\ \wt y \\ \wt \phi \\ \wt G
\end{pmatrix} =
\begin{pmatrix}
\wt r_\h(v) \\
\wt y_\h(v)^{-1}\left(\pa_v T^u(v,\xi) - \wt r_\h(v)^{-2}
\pa_\xi T^u(v,\xi)\right) \\
\xi + \wt \al_\h(v) \\
\pa_\xi T^u(v,\xi)
\end{pmatrix},
\end{equation}
where
\[
T^u = T_0 + Q_0^u + Q^u,
\]
with
\begin{equation}\label{def:T0} T_0 (v,\xi ) = S_0(\wt r_\h(v),\xi + \wt \al(v)),
\end{equation}
$S_0$ was introduced in~\eqref{def:GeneratingFunctionUnperturbed},
$Q_0^u$ in~\eqref{def:Q0} and $Q^u$ is given by
Theorem~\ref{thm:outer:ExistenceManifolds}. Analogously for $\wt
\Gamma^s$.

To compute its difference is necessary to have the parameterizations
of both manifolds defined in a common (real) domain. However, since
$\wt y_\h(0) = 0$, it is no possible to extend these
parameterizations to a common domain containing a real interval.  To
overcome this difficulty we extend the unstable manifold using a
different parameterization as has been explained in
Section~\ref{sec:ExistenceManifolds:Description}.

\subsubsection{From Hamilton-Jacobi parameterizations to
parameterizations invariant by the
flow}\label{sec:InvManifolds:Flow} The first step is to look for  a
change of variables of the form
\[
\Id
+g:(v,\xi)\mapsto (v+g_1(v,\xi),\xi+g_2(v,\xi)),\,\,g=(g_1,g_2),
\]
 in
such a way that, applied
to~\eqref{eq:parameterizationsbyHamiltonJacobi}, $\wh \Gamma^u = \wt \Gamma^u \circ (\Id +g)$
satisfies
\[
 \Phi_t(\wh \Gamma^u (v,\xi)) = \wh
\Gamma^u (v+t,\xi-G_0^3 t),
\]
where $\Phi_t$ is the flow associated to the
Hamiltonian~\eqref{def:HamCircularRotating}. Denoting by $X$ the
vector field generated by this Hamiltonian, this equation is
equivalent to
\begin{equation}
\label{invariancebytheflow2} \LL(\wh \Gamma^u) = X \circ \wh
\Gamma^u,
\end{equation}
where the operator $\LL$  defined in~\eqref{def:Outer:DiffOperator}
is understood to be acting on each component of~$\wh \Gamma^u$. In
fact, writing $X = X_0 + X_1$, being~$X_0$ the vector field for
$\mu=0$, and denoting by~$\Gamma_0$ the parameterization defined by
equations~\eqref{eq:parameterizationsbyHamiltonJacobi} replacing
$T^u$ by $T_0$, it is an immediate computation to see that
\begin{equation}
\label{invariancebytheunperturbedflow} \LL(\Gamma_0) = X_0 \circ
\Gamma_0.
\end{equation}
Hence,  the condition one needs to impose on $g$ to ensure that $\wh \Gamma^u = \wt
\Gamma^u \circ (\Id +g)$ satisfies~\eqref{invariancebytheflow2} is
\begin{equation}
\label{eq:equationforCg} \pa_v \wt \Gamma^u \circ (\Id +g) (1+\LL
(g_1)) + \pa_\xi \wt \Gamma^u \circ (\Id +g) (-G_0^3+ \LL (g_2)) = X
\circ \wt \Gamma^u \circ(\Id +g).
\end{equation}
We remark that, unlike the Hamilton-Jacobi equation, this is a
system of four partial differential equations. However, the
symplectic structure implies that it is only necessary to solve two
of them, which we choose to be the first and the third ones. Then, the other two
are also fulfilled.

By the definition of $T_0$ in~\eqref{def:T0},
\eqref{eq:parameterizationsbyHamiltonJacobi} and
using~\eqref{invariancebytheunperturbedflow},
equation~\eqref{eq:equationforCg} is equivalent to
\begin{equation}
\label{eq:hamiltonjacobitoparameterization} \LL(g) = \wt F \circ
(\Id +g), \quad \wt F = \begin{pmatrix} A^u \\ B^u \end{pmatrix}
\end{equation}
where
\[
\begin{aligned}
A^u & = \frac{1}{\wt y_\h^2 }\left(\pa_v T^u_1-\frac{1}{\wt r_\h^2}\pa_\xi
T_1^u\right) \\
B^u & = \left[\frac{1}{\wt r_\h^2}\pa_\xi T_1^u - \frac{1}{\wt r_\h^2 \wt
y_\h^2 }\left(\pa_v T_1^u-\frac{1}{\wt r_\h^2}\pa_\xi T_1^u\right)\right]
\end{aligned}
\]
and
\[ T_1^u = Q_0^u+ Q^u.\]

We emphasize that equations~\eqref{eq:equationforCg} or their
equivalent~\eqref{eq:hamiltonjacobitoparameterization} have to be
understood at the level of formal Fourier series, since $T^u$ and,
consequently, $\wt \Gamma^u$, are this kind of objects. In
particular, the change $\Id +g$ will be a formal Fourier series.
Hence, it is necessary to give a meaningful definition of the
composition of formal Fourier series. This will be accomplished
taking into account that the formal change of variables~$\Id +g$ is
close to the identity and using Taylor's formula.

In order to find the change $\Id +g$, we define for $g =(g_1,g_2)
\in \YY_{\nu,\rr,\kk,\de,\sigma}\times \YY_{\nu,\rr,\kk,\de,\sigma}$, where $\YY_{\nu,\rr,\kk,\de,\sigma}$ is the Banach space defined
in~\eqref{def:BanachFourier:0}, the norm
\[ \|g\|_{\nu,\sigma} =
\|g_1\|_{\nu,\sigma}+ G_0^{-3} \|g_2\|_{\nu,\sigma}.
\]

The following two technical lemmas summarize the properties of the
composition of formal Fourier series we need. The second is an immediate consequence of the first.

\begin{lemma}\label{lemma:banach:CauchyEstimates}
Fix constants $\sigma'<\sigma$, $\de'<\de$, $\rr'<\rr$ and $\kk'>\kk$ such that $(\log
\kk'-\log \kk)/2 < \sigma -\sigma'$ and take $h\in
\YY_{\nu,\rr,\kk,\de,\sigma}$. Its derivatives, as defined
in~\eqref{def:DerivativeFourierSeries}, satisfy  $\pa_v^m \pa_\xi^n
h\in \YY_{\nu,\rr',\kk',\de',\sigma'}$ and
\[
 \|\pa_v^m \pa_\xi^n h\|_{\nu,\sigma'}\leq \left(\frac{\kk'}{\kk}\right)^\nu
 \frac{G_0^{3m}m!n!}{(\kk'-\kk)^m(\sigma-\sigma'-(\log
\kk'-\log \kk)/2)^n}
 \|h\|_{\nu,\sigma}.
\]
\end{lemma}

\begin{proof}
Fix constants $\sigma'<\sigma$, $\de'<\de$, $\rr'<\rr$ and $\kk'>\kk$ such that $(\log
\kk'-\log \kk)/2 < \sigma -\sigma'$ and take $h\in
\YY_{\nu,\rr,\kk,\de,\sigma}$. By~\eqref{def:DerivativeFourierSeries},
\[
(\pa_v^m \pa_\xi^n h)^{[\ell]} = (i\ell)^n \pa_v^m  h^{[\ell]},
\]
and, by Cauchy's formula,
\[
\pa_v^m  h^{[\ell]}(v) = \frac{m!}{2\pi i} \int_{\gamma_v}
\frac{h^{[\ell]}(z)}{(z-v)^{m+1}}\, dz,
\]
where $\gamma_v$ is a curve with index $1$ with respect to $v$ in
$D^{u}_{\rr,\kk,\de}$. In particular, since we will measure the norm of
this function in $D^{u}_{\rr,\kk,\de}$ and for $G_0$ big enough, one can choose that $\de-\de'>(\kk'-\kk)G_0^{-3}$ and $\rr-\rr'>(\kk'-\kk)G_0^{-3}$, and we can take
$\gamma_v:[0,2\pi]\to \CC:t\mapsto v+ (\kk'-\kk)G_0^{-3} e^{it}$.
Hence, by~\eqref{def:Norma:Weighted}
\[
\begin{aligned}
\left\|(\pa_v^m \pa_\xi^n h)^{[\ell]}\right\|_{\nu+\ell/2,\nu-\ell/2} & =
\sup_{v\in
D^{u}_{\rr',\kk',\de'}}\left|(v-i/3)^{\nu-\ell/2}(v+i/3)^{\nu+\ell/2}
(\pa_v^m \pa_\xi^n h)^{[\ell]}(v)\right| \\
& \le \frac{\ell^n
m!}{(\kk'-\kk)^{m}G_0^{-3m}}\left(\frac{\kk'}{\kk}\right)^{\nu+|\ell|/2}\left\|h^{[\ell]}\right\|_{\nu+\ell/2,\nu-\ell/2},
\end{aligned}
\]
where we have used that
\[
\sup_{v\in D^{u}_{\rr',\kk',\de'}} \left| \frac{v\pm i/3}{v+
(\kk'-\kk)G_0^{-3} e^{it} \pm i/3} \right| \le \frac{\kk'}{\kk}.
\]
Then, by~\eqref{def:Norma:Weighted-Fourier}, since $h\in
\YY_{\nu,\rr,\kk,\de,\sigma}$,
\[
\begin{aligned}
\left \|\pa_v^m \pa_\xi^n h\right\|_{\nu,\sigma'} & \leq
 \sum_{\ell\in\ZZ}\left\|(\pa_v^m \pa_\xi^n h)^{[\ell]}\right\|_{\nu+\ell/2,\nu-\ell/2}
 e^{|\ell|\sigma'} \\
 & \leq
 \frac{m!}{(\kk'-\kk)^{m}G_0^{-3m}}\left(\frac{\kk'}{\kk}\right)^{\nu}
 \|h\|_{\nu,\sigma}\sum_{\ell\in\ZZ}
\ell^n
\left(\frac{\kk'}{\kk}\right)^{|\ell|/2}e^{|\ell|(\sigma'-\sigma)}
\\
& \le 2\frac{m!n!G_0^{3m}}{(\kk'-\kk)^{m}(\sigma-\sigma'-(\log
\kk'-\log \kk)/2)^n}\left(\frac{\kk'}{\kk}\right)^{\nu}
 \|h\|_{\nu,\sigma},
 \end{aligned}
\]
where we have used that, for $b>0$, $\sum_{\ell \ge 0} \ell^n
e^{-b\ell} \le 2n!/b^n$.
\end{proof}

\begin{lemma}\label{lemma:banach:Compositionparameterization}
Given $g=(g_1,g_2)$, let $\Id+g:(v,\xi) \mapsto
(v+g_1(v,\xi),\xi+g_2(v,\xi))$. We define the formal composition of
formal Fourier series
\[
\begin{split}
 h\circ (\Id +g) (v,\xi) &= h(v+g_1(v,\xi),\xi+g_2(v,\xi))\\
&=\sum_{m=0}^\infty \frac{1}{m!}\sum_{n=0}^m \binom{m}{n}
 \pa_v^{m-n}  \pa_\xi^{n} h(v,\xi) g_1^{m-n}(v,\xi) g_2^{n}(v,\xi).
\end{split}
\]
Fix constants $\sigma'<\sigma$, $\rr'<\rr$, $\de'<\de$ and $\kk'>\kk$ such that $(\log
\kk'-\log \kk)/2 < \sigma -\sigma'$. Let $\theta =
\min\{\kk'-\kk,\sigma-\sigma'-(\log \kk'-\log \kk)/2\}>2\eta>0$.
Then,
\begin{itemize}
\item If $h\in \YY_{\nu,\rr,\kk,\de,\sigma}$, $g = (g_1,g_2) \in \YY_{0,\rr',\kk',\de',\sigma'}\times \YY_{0,\rr',\kk',\de',\sigma'}$ and
$\|g\|_{0,\sigma'}\leq \eta G_0^{-3}$,
we have that
 $\wt h=h \circ (\Id +g) $ satisfies  $\wt h\in \YY_{\nu,\rr',\kk',\de',\sigma'}$ and
\[
 \|\wt h\|_{\nu,\sigma'}\leq \left(\frac{\kk'}{\kk}\right)^\nu
 \left(1-\frac{2\eta}{\theta}\right)^{-1}
 \|h\|_{\nu,\sigma}.
\]
\item Moreover, if $\|g\|_{0,\sigma'},\|\wh g\|_{0,\sigma'}\leq \eta
G_0^{-3}$, then $f=h \circ (\Id +g)-h \circ (\Id +\wh g)$ satisfies
\[
 \|f\|_{\nu,\sigma'}\leq \frac{2G_0^3}{\delta}
 \left(\frac{\kk'}{\kk}\right)^\nu \left(1-\frac{2\eta}{\theta}\right)^{-2}
 \|h\|_{\nu,\sigma}\|g-\wh g\|_{0,\sigma'}.
\]
\end{itemize}
\end{lemma}

\begin{theorem}\label{th:ExistenceManifolds:Summary}
Let $\rr_1$, $\de_0$ and $\kk_0$ and  $\sigma_0$ be the constants given by
Theorem~\ref{thm:outer:ExistenceManifolds}. Let $\sigma_1<\sigma_0$, $\rr_2<\rr_1$, $\de_1<\de_0$
and $\kk_1>\kk_0$  such that $(\log \kk_1-\log \kk_0)/2 < \sigma_0
-\sigma_1$ be fixed. Then, for $G_0$ big enough, there exists a (not
necessarily convergent) Fourier series $g =(g_1,g_2) \in
\YY_{0,\rr_2,\kk_1,\de_1,\sigma_1}\times \YY_{0,\rr_2,\kk_1,\de_1,\sigma_1}$ satisfying
\[
\|g\|_{0,\sigma_1}\leq b_2\mu G_0^{-4},
\]
where $b_2>0$ is a constant independent of $\mu$ and $G_0$, such
that
\[\wh \Gamma^u = \wt \Gamma^u \circ (\Id + g),
\]
satisfies~\eqref{invariancebytheflow2}.
\end{theorem}

\begin{proof}
The function $g$ is found as a solution of
equation~\eqref{eq:hamiltonjacobitoparameterization}. The following
auxiliary lemma provides the appropriate bounds on the function $\wt
F$ introduced in~\eqref{eq:hamiltonjacobitoparameterization}.

\begin{lemma}\label{lem:boundsABparameterization} $\wt F \in \YY_{1/2,\rr_1,\kk_0,\de_0,\sigma_0}\times \YY_{1/2,\rr_1,\kk_0,\de_0,\sigma_0}$
and $\|\wt F\|_{1/2,\sigma_0} \le K \mu G_0^{-4}$.
\end{lemma}
\begin{proof}[Proof of Lemma~\ref{lem:boundsABparameterization}]
It follows directly from the bound on $Q_0^u$ in
Lemma~\ref{lemma:outer:boundsU1andQ0} and the bound on $Q^u$ in
Theorem~\ref{thm:outer:ExistenceManifolds}.
\end{proof}

Using the operator~$\wt\GG$ in~\eqref{def:Outer:operadorG}, we
rewrite equation~\eqref{eq:hamiltonjacobitoparameterization} as
\begin{equation}
\label{eq:fixedpointeqforg} g = \wt \GG \circ \wt F \circ (\Id +g)
\end{equation}
and we look for $g$ as a fixed point of this operator.

Taking $G_0$ large enough, by
Lemmas~\ref{lem:boundsABparameterization}
and~\ref{lemma:banach:Compositionparameterization}, the map
$g\mapsto \wt F \circ (\Id +g)$ is well defined and sends the ball
$B(K\mu G_0^{-4})\subset\YY_{0,\rr_2,\kk_1,\de_1,\sigma_1}\times
\YY_{0,\rr_2,\kk_1,\de_1,\sigma_1}$ to
$\YY_{0,\rr_2,\kk_1,\de_1,\sigma_1}\times
\YY_{0,\rr_2,\kk_1,\de_1,\sigma_1}$. Also, by
Lemma~\ref{lem:boundsABparameterization},
\[
\left\|\wt F\circ (\Id+g)_{\mid g=0}\right\|_{1/2,\sigma_1} \le \left\|\wt
F\right\|_{1/2,\sigma_0} \le K\mu G_0^{-4},
\]
which implies, by Lemma~\ref{lemma:Outer:PropertiesOfG}, that
\[
\left\|\wt \GG \circ \wt F\right\|_{0,\sigma_1} \le K\mu G_0^{-4}.
\]
It only remains to see that the map defined
by~\eqref{eq:fixedpointeqforg} is Lipschitz. But this is
straightforward, since, by
Lemma~\ref{lemma:banach:Compositionparameterization}, for any $g,\wh
g \in B(K\mu
G_0^{-4})\subset\YY_{0,\rr_2,\kk_1,\de_1,\sigma_1}\times
\YY_{0,\rr_2,\kk_1,\de_1,\sigma_1}$,
\[
\left\|\wt F\circ (\Id+g) - \wt F\circ (\Id+\wh g)\right\|_{1/2,\sigma_1}
\le \wt K \mu G_0^{-1} \left\|g -\wh g\right\|_{0,\sigma_1},
\]
where the constant $\wt K$ depends on the reduction of the domain.
Finally, using the last inequality and
Lemma~\ref{lemma:Outer:PropertiesOfG}, we have that
\[
\left\|\wt \GG (\wt F\circ (\Id+g) - \wt F\circ (\Id+\wh
g))\right\|_{0,\sigma_1} \le   K\wt K \mu G_0^{-1} \left\|g -\wh
g\right\|_{0,\sigma_1}.
\]
\end{proof}

\subsubsection{Analytic extension of the unstable manifold by the flow
parameterization}

Now we can extend the flow parameterization of the unstable manifold
given by Theorem~\ref{th:ExistenceManifolds:Summary}, whose Fourier
coefficients are defined, up to now,  for $v \in
D^{u}_{\rr_2,\kk_1,\de_1}$, to~$D^\fl_{\kk_1,\de_1}$ defined
in~\eqref{def:DominiParamFlux} (see Figure~\ref{fig:DomFlux}). This
extension is obtained using the flow of the Hamiltonian~\eqref{def:HamCircularRotating}.

We look for a parameterization of the unstable manifold $\wh
\Gamma^u$ satisfying~\eqref{invariancebytheflow2}, defined in~$
D^\fl_{\kk_1,\de_1}$ and which coincides with the one given by
Theorem~\ref{th:ExistenceManifolds:Summary}
in~$D^{u}_{\rr_2,\kk_1,\de_1} \cap D^\fl_{\kk_1,\de_1}$.

Using the notations of the previous section, where $X= X_0+ X_1$ is
the vector field generated  by the
Hamiltonian~\eqref{def:HamCircularRotating}, and introducing
$\Gamma_1^u$ such that $\wh \Gamma^u = \Gamma_0 + \Gamma_1^u$, where
$\Gamma_0$ is obtained
from~\eqref{eq:parameterizationsbyHamiltonJacobi} with $T_0$ instead
of $T^u$, and satisfies $\LL(\Gamma_0) =X_0 \circ \Gamma_0$,
equation~\eqref{invariancebytheflow2} becomes
\begin{equation}
\label{def:eqforGamma1} \wt \LL (\Gamma_1^u) = \wt \FF (\Gamma_1^u),
\end{equation}
where
\begin{equation}
\label{def:tildeLL} \wt \LL (\Gamma) = \LL(\Gamma) - DX_0(\Gamma_0)
\Gamma
\end{equation}
and
\begin{equation}
\label{def:tildeFF} \wt \FF (\Gamma) =
X_0(\Gamma_0+\Gamma)-X_0(\Gamma_0) -DX_0(\Gamma_0) \Gamma + X_1
(\Gamma_0+\Gamma).
\end{equation}

We will find a solution of this equation whose initial condition is
given the formal Fourier series obtained by Theorem~\ref{th:ExistenceManifolds:Summary}.
For this reason, we introduce the space of formal Fourier series
\begin{equation}
\label{def:BanachFourier:full}
\ZZZ_{\nu,\kk,\de,\sigma}=\left\{h(v,\xi)=\sum_{\ell\in
\ZZ}h^{[\ell]}(v) e^{i\ell\xi}: h^{[\ell]}\in \left(\wt
\XX_{\nu+\ell/2,\nu-\ell/2,\kk,\de}\right)^4,
\|h\|_{\nu,\sigma}<\infty\right\},
\end{equation}
where
\[\wt \XX_{\nu_-,\nu_+,\kk,\de}=\left\{h:D^\fl_{\kk,\de}\rightarrow \CC:
\|h\|_{\nu_-,\nu_+}<\infty  \right\},
\]
the norm $\|\cdot \|_{\nu_-,\nu_+}$ is the one defined
in~\eqref{def:Norma:Weighted}, but taking the supremum in~$D^\fl_{\kk,\de}$ and,
\[
\|h\|_{\nu,\sigma} = \sum_{i=1}^4 \|h_i\|_{\nu,\sigma},
\]
where the norm $\|\cdot\|_{\nu,\sigma}$ in the right hand side
above is the one in~\eqref{def:BanachFourier:0}. We remark that, since $\pm i /3$ are at a distance $\Oo$ of the
domain $D^\fl_{\kk,\de}$, the norms $\|\cdot\|_{\nu,\sigma}$
are equivalent for all $\nu$. Therefore, we only work with $\nu=0$.

Let $\Psi(v)$ be any fundamental matrix of the linear system
\[
\dot z(v) = DX_0(\Gamma_0(v,\xi)) z(v), \quad v \in
D^\fl_{\kk_1,\de_1}.
\]
 In
fact, the system above does not depend on $\xi$. As a consequence,
the matrix $\Psi$ satisfies
\[ \LL \Psi = DX_0(\Gamma_0) \Psi.
\]
Since $\Gamma_0(v,\xi)$ is well
defined and bounded for $v\in  D^\fl_{\kk_1,\de_1}$ and $D^\fl_{\kk_1,\de_1}$
is bounded, we have that there exists $K>0$ such that
\begin{equation}
\label{bound:fundamentalmatrix} \sup_{v\in D^\fl_{\kk_1,\de_1}}
\max\left\{\|\Psi(v)\|_{0,\sigma},\left\|\Psi(v)^{-1}\right\|_{0,\sigma}\right\} \leq K,
\end{equation}
in any matrix norm.

We choose $v_2$ at the top of the domain $D^\fl_{\kk_1,\de_1}$ and
define $\wt \GG$ like in~\eqref{def:Outer:operadorG}.
Lemma~\ref{lemma:Outer:PropertiesOfG} applies without further
modification.
Then, it is a easy
computation to check that a left inverse of the operator~$\wt \LL$
in~\eqref{def:tildeLL} is
\begin{equation}
\label{def:variationofconstants} \wh \GG (\Gamma) = \Psi \wt \GG
(\Psi^{-1} \Gamma).
\end{equation}

Since we want to obtain an analytic continuation of the
parameterization of the invariant manifold given by
Theorem~\ref{th:ExistenceManifolds:Summary}, we introduce an initial
condition defining
\[
\begin{aligned} \Gamma_1^0(v,\xi) = &
\sum_{\ell <0} \Psi(v) \Psi^{-1}(v_2 ) \Gamma_1^{[\ell]}(v_2)
e^{i\ell G_0^3 (v- v_2)} e^{i\ell \xi}
\\ & + \sum_{\ell >0} \Psi(v) \Psi^{-1}(\bar v_2 ) \Gamma_1^{[\ell]}(\bar v_2)
e^{i\ell G_0^3 (v-\bar v_2)} e^{i\ell \xi} \\
& + \Psi(v) \Psi^{-1}(-\rho_2 ) \Gamma_1^{[0]}(-\rho_2),
\end{aligned}
\]
where $\Gamma_1^{[\ell]}$ are the Fourier coefficients of $\Gamma_1^u$ and therefore
are already known at the points $v_2$, $\bar v_2$ and $-\rho_2$. The following lemma collects the properties we need
about $\Gamma_1^0$.

\begin{lemma}
\label{lem:initialconditionfortheflow} The function $\Gamma_1^0$  satisfies
\begin{itemize}
\item $\wt \LL(\Gamma_1^0) = 0$ and
\item $\Gamma_1^0 \in \ZZZ_{0,\kk_1,\de_1,\sigma_1}$ with
$\|\Gamma_1^0\|_{0,\sigma_1} < K\mu G_0^{-4}$.
\end{itemize}
\end{lemma}
\begin{proof}
 The first statement is straightforward. For the second one, it is enough to write
\[
\begin{split}
 \Gamma_1^u&=\wh\Gamma^u-\Gamma_0\\
&=\left(\wt\Gamma^u\circ(\Id+g)-\Gamma_0\circ(\Id+g)\right)+\left(\Gamma_0\circ(\Id+g)-\Gamma_0\right).
\end{split}
\]
Now, using the formula of $\wt\Gamma^u$ in
\eqref{eq:parameterizationsbyHamiltonJacobi}, the estimates given in
Theorem~\ref{thm:outer:ExistenceManifolds}, the estimates for
$Q_0^u$ in Lemma~\ref{lemma:outer:boundsU1andQ0} and the ones for
$g$ given in Theorem~\ref{th:ExistenceManifolds:Summary},  Lemma
\ref{lemma:banach:Compositionparameterization} for the composition
and the fact that the domain $D^\fl_{\kk_1,\de_1}$  only contains
points at a distance of order one of the singularities $v=\pm i/3$,
one obtains the estimate of the second statement.
\end{proof}

We rewrite equation~\eqref{def:eqforGamma1} using the function
$\Gamma_1^0$ and the operator $\wh \GG$
in~\eqref{def:variationofconstants} as
\begin{equation}
\label{eq:fixedpointequationforflow} \Gamma_1^u = \Gamma_1^0 + \wh \GG
\circ \wt \FF (\Gamma_1^u).
\end{equation}
It is important to remark that the definition of $\wt \FF$
in~\eqref{def:tildeFF} involves the compositions~$X_0 (\Gamma_0 +
\Gamma_1^u)$ and~$X_1 (\Gamma_0 + \Gamma_1^u)$ of formal Fourier
series. Like in the previous section, in
Lemma~\ref{lemma:banach:Compositionparameterization}, this
composition is defined through a Taylor series of the form
\begin{equation}
\label{eq:TaylorX1Gamma} X (\Gamma_0 + \Gamma) = \sum_{k\ge 0}
\frac{1}{k!} D^k X(\Gamma_0) \Gamma^{\otimes k}.
\end{equation}
However, this case is much simpler, since the vector fields $X_0$
and $X_1$ and the function~$\Gamma_0$ are true functions, analytic
and appropriately bounded in the domain~$D^\fl_{\kk_1,\de_1}$.

\begin{lemma}
\label{lem:compositionofvectorfieldnadformalfourierseries} Let
$K>0$. Assume $\Gamma, \wt \Gamma \in \ZZZ_{0,\kk_1,\de_1,\sigma_1}$ with
$\|\Gamma\|_{0,\sigma_1}, \|\wt \Gamma\|_{0,\sigma_1} \le
K\mu G_0^{-4}$. Then there exists $K'
>0$ such that, if $G_0$ is large enough,
\begin{itemize}
\item defining $Y(\Gamma) = X_0 (\Gamma_0 + \Gamma) -X_0(\Gamma_0)
-DX_0(\Gamma_0)\Gamma$, we have that  $Y(\Gamma)\in \ZZZ_{0,\kk_1,\de_1,\sigma_1}$ with
\[
 \|Y(\Gamma)\|_{0,\sigma_1}\le K' \mu G_0^{-4},
\]
\item $X_1 (\Gamma_0 + \Gamma)\in \ZZZ_{0,\kk_1,\de_1,\sigma_1}$ with
$\|X_1 (\Gamma_0 + \Gamma)\|_{0,\sigma_1} \le K'\mu G_0^{-4}$,
\item $\|Y(\Gamma) -Y (\wt \Gamma)\|_{0,\sigma_1} \le K' \mu
G_0^{-4}\|\Gamma -\wt \Gamma \|_{0,\sigma_1}$,
\item
$\|X_1 (\Gamma_0 + \Gamma)-X_1 (\Gamma_0 + \wt
\Gamma)\|_{0,\sigma_1} \le K' \mu G_0^{-4}\|\Gamma -\wt \Gamma
\|_{0,\sigma_1}$.
\end{itemize}
\end{lemma}

\begin{proof}
Cauchy estimates imply that
\[
\| D^k
X_0(\Gamma_0)\|_{0,\sigma} \le \wt K k!, \qquad \| D^k X_1(\Gamma_0)\|_{0,\sigma} \le \wt K \mu
G_0^{-4}k!.
\]
The claims follow then from formula~\eqref{eq:TaylorX1Gamma}.
\end{proof}

Now we can claim
\begin{proposition}
\label{prop:extensionbytheflow} Let $\kk_1$, $\de_1$ and $\sigma_1$
be the constants considered in
Theorem~\ref{th:ExistenceManifolds:Summary}. Then, there exists
$b_3>0$ such that if $G_0$ is large enough, the fixed point
equation~\eqref{eq:fixedpointequationforflow} has a unique solution
$\Gamma_1^u\in B(b_3\mu G_0^{-4}) \subset
\ZZZ_{0,\kk_1,\de_1,\sigma_1}$.
\end{proposition}

\begin{proof}
By Lemmas~\ref{lem:initialconditionfortheflow},
\ref{lem:compositionofvectorfieldnadformalfourierseries} and the
algebra properties of $\ZZZ_{0,\kk_1,\de_1,\sigma_1}$, the map $\wt \KK:\Gamma
\mapsto \Gamma_1^0 + \wt G \circ \wt \FF (\Gamma)$ is well defined
from $B(K\mu G_0^{-4}) \subset \ZZZ_{0,\kk_1,\de_1,\sigma_1}$ to
$\ZZZ_{0,\kk_1,\de_1,\sigma_1}$ for any $K>0$. Also, by
Lemmas~\ref{lem:initialconditionfortheflow},
\ref{lem:compositionofvectorfieldnadformalfourierseries}
and~\ref{lemma:Outer:PropertiesOfG}, and using the
bound~\eqref{bound:fundamentalmatrix} on the fundamental matrix
$\Psi$, we have that
\[
\|\wt \KK(0)\|_{0,\sigma_1} = \|\Gamma_1^0+
\wh \GG(X_1 \circ \Gamma_0)\|_{0,\sigma_1} \le \frac{b_3}{2}\mu G_0^{-4}.
\]
for certain $b_3$ independent of $G_0$ and $\mu$.

Finally,
Lemmas~\ref{lem:compositionofvectorfieldnadformalfourierseries}
and~\ref{lemma:Outer:PropertiesOfG} imply that $\wt \KK$ is
Lipschitz with Lipschitz constant $K\mu G_0^{-4}$.
\end{proof}

\subsubsection{From flow parameterization to Hamilton-Jacobi parameterization}

Finally we apply a change of variables to the parameterization $\wh
\Gamma^u = \Gamma_0 + \Gamma_1^u$ obtained in
Proposition~\ref{prop:extensionbytheflow}, which is invariant by the
flow, in order to obtain an extension of the parameterization of the
form~\eqref{eq:parameterizationsbyHamiltonJacobi} to the domain $\wt
D_{\kk_2,\de_2}$ defined in~\eqref{def:Domain:UnstableLastPiece} for
some $\kk_2>\kk_1$ and $\de_2>\de_1$ (see
Figure~\ref{fig:DomExtFinal}). Equivalently, we obtain an extension
of the Fourier coefficients of~$T^u$, solution of the
Hamilton-Jacobi equation, to this domain. Like in the previous section, since
this change of variables has to be found in a domain which is far
from the singularities $v=\pm i/3$, the procedure is rather simple.

We will find a change of variables of the form $\Id + f$, with $f=
(f_1,f_2)$, such that
\begin{equation}
\label{eq:fromparameterizationtohamiltonjacobi}
\begin{aligned}\pi_1
\circ \wh
\Gamma^u \circ (\Id +f)(v,\xi) & = \wt r_\h(v), \\
\pi_3 \circ \wh \Gamma^u \circ (\Id +f)(v,\xi) & = \wt \al_\h(v) +
\xi,
\end{aligned}
\end{equation}
where $\pi_i$ denotes the projection on the $i$-th component. We
will see that this change of variables is unique under certain
conditions. Therefore, in $D^{u}_{\rr_2,\kk_1,\de_1} \cap \wt
D_{\kk_2,\de_2}$ the change $\Id+f$ is the formal
inverse of the change $\Id+g$ obtained in Theorem~\ref{th:ExistenceManifolds:Summary}. Then,
\begin{equation}\label{def:GeneratingFunctionFromParam}
\begin{aligned}
\pi_2 \circ \wh \Gamma^u \circ (\Id +f)(v,\xi) & = \wt
y_\h(v)^{-1}\left(\pa_v
T^u(v,\xi) - \wt r_\h(v)^{-2} \pa_\xi T^u(v,\xi)\right) \\
\pi_4 \circ \wh \Gamma^u \circ (\Id +f)(v,\xi) & = \pa_\xi T^u(v,\xi).
\end{aligned}
\end{equation}
These equalities provide an analytic extension of the Fourier coefficients of  $T^u$
given by Theorem~\ref{thm:outer:ExistenceManifolds}.

Taking into account that $\pi_1 \circ \Gamma_0(v,\xi) = \wt r_\h(v)$
and $\pi_3 \circ \Gamma_0 (v,\xi)= \wt \al_\h(v)+\xi$ and that $\wh\Gamma^u=\Gamma_0+\Gamma_1^u$,
equations~\eqref{eq:fromparameterizationtohamiltonjacobi} are
equivalent to
\[
 f = \PP(f),
\]
where
\[\label{def:P} \PP(f)(v,\xi) =\begin{pmatrix} -\wt y_\h^{-1}(v)
\left(\wt r_\h(v+f_1(v,\xi))-\wt r_\h(v) - \wt y_\h(v) f_1(v,\xi) -
\pi_1
\circ \Gamma_1^u \circ (\Id +f)(v,\xi)\right) \\
\wt \al_\h(v) - \wt \al_\h(v+f_1(v,\xi)) - \pi_3 \circ \Gamma_1^u
\circ (\Id +f)(v,\xi)
\end{pmatrix}.
\]
Then, defining the space $\wh \ZZZ_{\nu,\kk,\de,\sigma}$ analogously
to~\eqref{def:BanachFourier:full}, with only two components
$f=(f_1,f_2)$ and the same norm, one has
\begin{proposition}\label{prop:FromParamToHJ}
Consider  the constants $\kk_1$, $\de_1$ and $\sigma_1$ given by
Proposition~\ref{prop:extensionbytheflow} and any $\kk_2>\kk_1$,
$\de_2>\de_1$ and $\sigma_2<\sigma_1$. Then,
\begin{itemize}
 \item  There exists $b_4>0$ independent of $G_0$ and $\mu$ such that, if $G_0$ is large enough, the operator $\PP$ has a unique fixed point
 $f=(f_1,f_2) \in \wh \ZZZ_{0,\kk_2,\de_2,\sigma_2}$ with
\[
\|f\|_{0,\sigma_2} \le b_4 \mu G_0^{-4}.
\]
\item The change $\Id+f$ is the inverse of the restriction of
the change given by Theorem~\ref{th:ExistenceManifolds:Summary} to
the domain~$D^{u}_{\rr_2,\kk_1,\de_1} \cap \wt D_{\kk_2,\de_2}$.
\item Moreover, the equation~\eqref{def:GeneratingFunctionFromParam} defines a formal Fourier series of a generating function
which can be written as $T^u=T_0+T_1^u$ where $T_0$ has been defined
in~\eqref{def:T0} and $T_1$ satisfies
\[
 \left\|\pa_v^m\pa_\xi^nT_1^u\right\|_{0,\sigma_2}\leq K\mu G_0^{-4},
\]
for any $0\leq m+n\leq 1$.
\end{itemize}
\end{proposition}
\begin{proof}
The first part follows from rewriting
Lemma~\ref{lem:compositionofvectorfieldnadformalfourierseries} in
this setting and taking into account that
$\|\Gamma_1^u\|_{0,\kk,\sigma} \le K\mu G_0^{-4}$. The second, from the uniqueness statements of the first part and
Theorem~\ref{th:ExistenceManifolds:Summary}.

Finally, the third statement is deduced from equation~\eqref{def:GeneratingFunctionFromParam},
taking into account that the formal generating function~$T^u$ was already defined in
$D^{u}_{\rr_2,\kk_1,\de_1} \cap \wt D_{\kk_2,\de_2}$ thanks to
Theorem~\ref{thm:outer:ExistenceManifolds}.
\end{proof}

With this proposition we have finished the extension procedure and we have
both invariant manifolds parameterized as formal Fourier series, $T^{u,s}$, in the domain
$D_{\kk_2,\de_2}$. We summarize this fact in the following theorem.

\begin{theorem}\label{th:SummaryManifolds}
Let $\kk_2$ and $\de_2$ the constants given by
Proposition~\ref{prop:FromParamToHJ}. Then, there exist formal
Fourier series $T_1^{u,s}$ with Fourier coefficients defined in
$D_{\kk_2,\de_2}$ (the domain defined in~\eqref{def:DominisRaros})  which are solutions of equation
\eqref{eq:HJ:Rescaled} and  satisfy
\[
 \|\pa_v^m\pa_\xi^nT_1^{u,s}\|_{1,\sigma_2}\leq b_5\mu G_0^{-4},
\]
with $0\leq m+n\leq 1$ and  $b_5>0$ a constant independent of $\mu$ and $G_0$.
\end{theorem}
\begin{proof}
It is enough to join the results of
Theorem~\ref{thm:outer:ExistenceManifolds} and Proposition
\ref{prop:FromParamToHJ}. Indeed, in Theorem
\ref{thm:outer:ExistenceManifolds} we have obtained solutions as
formal Fourier series of equation~\eqref{eq:HJ:Rescaled} of the form
$T_1^{u,s}=Q_0^{u,s}+Q^{u,s}$ in the domains
$D_{\rr_1,\kk_0,\de_0}^{u,s}$. Using the fact that
$D_{\kk_2,\de_2}\subset D_{\rr_1,\kk_0,\de_0}^{s}$ and the estimates
obtained in Lemma~\ref{lemma:outer:boundsU1andQ0} and Theorem
\ref{thm:outer:ExistenceManifolds} we obtain the desired estimates
for $T_1^s$. For the unstable manifold, we have to recall that
$D_{\kk_2,\de_2}\subset D_{\rr_1,\kk_0,\de_0}^{s}\cup \wt
D_{\kk_2,\de_2}$. Then, we need also to take into account the
results obtained in Proposition~\ref{prop:FromParamToHJ} and recall
that since the points $\wt D_{\kk_2,\de_2}$ are far from the
singularities $v=\pm i/3$, the norms $\|\cdot\|_{1,\sigma_2}$ and
$\|\cdot\|_{0,\sigma_2}$ are equivalent and satisfy
$\|\cdot\|_{0,\sigma_2}\leq K\|\cdot\|_{1,\sigma_2}$.
\end{proof}

\section{The difference between the manifolds}\label{sec:DiffManifolds}
Once we have obtained the parameterization of the  invariant
manifolds (as formal Fourier series) up to  points $\OO(G_0^{-3})$
close to the singularities $v=\pm i/3$, the next step is to study their
difference. It suffices to study the difference between the
Fourier coefficients of the generating functions $T^{u,s}$ defined
in~\eqref{def:GeneratingFunctionT}. To this end, we define
\begin{equation}\label{def:Difference}
\wt \Delta(v,\xi)=T^s(v,\xi)-T^u(v,\xi).
\end{equation}
Recall that $T^{u,s}(v,\xi)=T_0(v,\xi)+T_1^{u,s}(v,\xi)$  where
$T_0$ is defined in~\eqref{def:T0} and
$T_1^{u,s}$ are the Fourier series obtained in
Theorem~\ref{th:SummaryManifolds}. Therefore, the Fourier
coefficients of  $\wt\Delta$ are defined in $D_{\kk_2,\de_2}$. Recall that
$\wt\Delta$ is not a function but a formal Fourier series that does not need
to be convergent.

Subtracting equation~\eqref{eq:HJ:Rescaled} for $T_1^s(v,\xi)$ and
$T_1^u(v,\xi)$ considered as equations of formal Fourier series, we can
easily see that
\[
\wt \Delta\in \mathrm{Ker} \wt\LL
\]
where $\wt\LL$ is the differential operator
\begin{equation}\label{def:DifferentialOperator:Anulador}
 \wt\LL=\left(1+A(v,\xi)\right)\pa_v -G_0^{3}\left(1+B(v,\xi)\right)\pa_\xi
\end{equation}
with
\begin{equation}\label{def:DiffOperator:AB}
\begin{aligned}
A(v,\xi)=& \frac{1}{2 \wt y_\h^2}\left(\left(\pa_v T_1^s+\pa_v T_1^u\right)
-\frac{1}{ \wt r_\h^2}(\pa_\xi T_1^s+\pa_\xi T_1^u)\right)\\
B(v,\xi)=&\frac{G_0^{-3}}{2 \wt y_\h^2}\left(\left(\pa_v T_1^s+\pa_v T_1^u\right)
-\frac{1}{ \wt r_\h^2}(\pa_\xi T_1^s+\pa_\xi T_1^u)\right)\\
&-\frac{G_0^{-3}}{2 \wt r_\h^2}\left(\pa_\xi T_1^s+\pa_\xi
T_1^u\right).
\end{aligned}
\end{equation}
where $T_1^{u,s}$ are the formal Fourier series given by
Theorem~\ref{th:SummaryManifolds}.

Recall that the equation $\wt\LL\wt \Delta=0$
 is stated as an equation for not necessarily
convergent Fourier series and not as an equation for functions. The
derivatives for Fourier series have been defined in the natural way
in~\eqref{def:DerivativeFourierSeries}.

To study the function $\wt \Delta$ we proceed as in~\cite{Sauzin01}
(see also~\cite{GuardiaOS10}). That is, we consider a near the
identity change of coordinates which conjugates the operator
$\wt\LL$ in~\eqref{def:DifferentialOperator:Anulador} with the
operator $\LL$ in~\eqref{def:Outer:DiffOperator}. Recall that, as we
observed in Section~\ref{sec:InvManifold:ExtensionUnstable}, the
composition of formal Fourier series with a near the identity formal
transformation is well defined under suitable hypotheses (see
Lemma~\ref{lemma:banach:Composition} below).


The structure of this section goes as follows. First, in
Section~\ref{sec:diff:banach}, we introduce a functional setting to
study the difference between the generating functions. This
functional setting  is essentially the same as the one considered in
Section~\ref{sec:outer:WeightedFourierBanach} but referred to
Fourier series with coefficients defined in $D_{\kk,\de}$. In
Section~\ref{sec:diff:straightening} we straighten the operator
$\wt\LL$. Finally, in Section~\ref{sec:diff:ComputingDiff}, we use
this result to prove Theorem~\ref{th:SplittingViaGeneratingFunctions}.

\subsection{Weighted Fourier norms and Banach spaces}\label{sec:diff:banach}
We devote this section to define Banach spaces for Fourier series
with coefficients defined in $D_{\kk,\de}$. First, we define the
Banach spaces for the Fourier coefficients as
\[
\PP_{\nu_-,\nu_+,\kk,\de}=\left\{h:D_{\kk,\de}\rightarrow \CC: \text{analytic}, \|h\|_{\nu_-,\nu_+}<\infty  \right\},
\]
where
\[
 \|h\|_{\nu_-,\nu_+}=\sup_{v\in
D_{\kk,\de}}\left|(v-i/3)^{\nu_+}(v+i/3)^{\nu_-} h(v)\right|.
\]
Note that these definitions are the same
as~\eqref{def:Norma:Weighted} and~\eqref{def:BanachCoefs} but for
functions defined in $D_{\kk,\de}$ instead of $D^u_{\rr,\kk,\de}$.
Now we define the Banach space for Fourier series
\[
\QQQ_{\nu,\kk,\de,\sigma}=\left\{h(v,\xi)=\sum_{\ell\in
\ZZ}h^{[\ell]}(v) e^{i\ell\xi}: h^{[\ell]}\in
\PP_{\nu+\ell/2,\nu-\ell/2,\kk,\de},
\|h\|_{\nu,\sigma}<\infty\right\}.
\]
where
\[
 \|h\|_{\nu,\sigma}=\sum_{\ell\in\ZZ}\left\|h^{[\ell]}\right\|_{\nu+\ell/2,\nu-\ell/2} e^{|\ell|\sigma}.
\]
The Banach space $\QQQ_{\nu,\kk,\de,\sigma}$ satisfies the algebra
properties stated in Lemma~\ref{lemma:banach:AlgebraProps}.
Therefore, from now on in this section, we will refer to this lemma
understanding the properties stated in it as properties referred to
elements of $\QQQ_{\nu,\kk,\de,\sigma}$ instead of elements of
$\YY_{\nu,\kk,\de,\sigma}$. Moreover, in the present section we will
need to take derivatives of and compose Fourier series. To this end
we state the following two technical  lemmas, which are equivalent
to Lemmas~\ref{lemma:banach:CauchyEstimates}
and~\ref{lemma:banach:Compositionparameterization}.

\begin{lemma}\label{lemma:banach:CauchyEstimates:diff}
Fix constants $\sigma'<\sigma$, $\kk'>\kk$ and $\de'>\de$ and take  $h\in
\QQQ_{\nu,\kk,\de,\sigma}$. Its derivatives, as defined
in~\eqref{def:DerivativeFourierSeries}, satisfy
\begin{itemize}
 \item $\pa_v^m h\in \QQQ_{\nu,\kk',\de',\sigma'}$ and
\[
 \|\pa_v^m h\|_{\nu,\sigma'}\leq \left(\frac{\kappa'}{\kappa}\right)^\nu \frac{G_0^{3m}m!}{(\kappa'-\kappa)^m}
 \|h\|_{\nu,\sigma}.
\]
 \item $\pa_\xi h\in \QQQ_{\nu,\kk',\de',\sigma'}$ and
\[
 \|\pa_\xi h\|_{\nu,\sigma'}\leq \frac{1}{\sigma-\sigma'}\|h\|_{\nu,\sigma}.
\]
\end{itemize}
\end{lemma}

\begin{lemma}\label{lemma:banach:Composition}
We define the formal composition of formal Fourier series
\[
 h(v+g(v,\xi),\xi)=\sum_{m=0}^\infty \frac{1}{m!}\pa_v^mh(v,\xi) g^m(v,\xi).
\]
Fix constants $\sigma'<\sigma$, $\kk'>\kk$ and $\de'>\de$. Let
$\kk'-\kk>\eta>0$. Then,
\begin{itemize}
\item If $h\in \QQQ_{\nu,\kk,\de,\sigma}$, $g\in \QQQ_{0,\kk',\de',\sigma'}$ and $\|g\|_{0,\sigma'}\leq \eta G_0^{-3}$
we have that
 $X(v,\xi)=h(v+g(v,\xi),\xi)$ satisfies  $X\in \YY_{\nu,\kk',\de',\sigma'}$ and
\[
 \|X\|_{\nu,\sigma'}\leq \left(\frac{\kk'}{\kk}\right)^\nu \left(1-\frac{\eta}{\kk'-\kk}\right)^{-1}
 \|h\|_{\nu,\sigma}.
\]
Moreover, if $\|g_1\|_{0,\sigma},\|g_2\|_{0,\sigma}\leq \eta
G_0^{-3}$, then $Y(v,\xi)=h(v+g_2(v,\xi),\xi)-h(v+g_1(v,\xi),\xi)$
satisfies
\[
 \|Y\|_{\nu,\sigma'}\leq \frac{G_0^{3}}{\kk'-\kk}
 \left(\frac{\kk'}{\kk}\right)^\nu \left(1-\frac{\eta}{\kk'-\kk}\right)^{-2}
 \|h\|_{\nu,\sigma}\|g_2-g_1\|_{0,\sigma}.
\]
\item If $\pa_v h\in \QQQ_{\nu,\kk,\de,\sigma}$, $g_1, g_2 \in \QQQ_{0,\kk',\de',\sigma'}$ and
$\|g_1\|_{0,\sigma'},\|g_2\|_{0,\sigma'}\leq \eta G_0^{-3}$ we
have that
  $Y\in \YY_{\nu,\kk',\de',\sigma'}$ and
\[
 \|Y\|_{\nu,\sigma'}\leq \left(\frac{\kk'}{\kk}\right)^{\nu}
 \frac{1}{1-\frac{\eta}{\kk'-\kk}}
 \|\pa_v h\|_{\nu,\sigma}\|g_2-g_1\|_{0,\sigma'}.
\]
\end{itemize}
\end{lemma}

\subsection{Straightening the operator $\wt\LL$}\label{sec:diff:straightening}
Once we have defined the Banach space $\QQQ_{\nu,\kk,\de,\sigma}$ we can
show that the operator~$\wt\LL$ can be straightened, as stated in the
following theorem.
\begin{theorem}\label{thm:StraighteningOperator}
Let $\sigma_2$, $\kk_2$ and $\de_2$ be the constants given by
Theorem~\ref{th:SummaryManifolds}. Let $\sigma_3<\sigma_2$, $\kk_3>\kk_2$ and $\de_3>\de_2$ be fixed. Then, for $G_0$ big enough, there exists
a (not necessarily convergent) Fourier series $\CCC\in\QQQ_{0,\kk_3,\de_3,\sigma_3}$ satisfying
\[
\|\CCC\|_{0,\sigma_3}\leq b_6\mu G_0^{-4},
\]
with $b_6>0$ a constant independent of $\mu$ and $G_0$, such
that
\begin{equation}\label{def:difference:reparameterized}
 \Delta(w,\xi)=\wt \Delta\left(w+\CCC(w,\xi),\xi\right),
\end{equation}
where $\wt \Delta$ is the function defined
in~\eqref{def:Difference}, is well defined and satisfies that
$\Delta\in\mathrm{Ker}\LL$ and $\LL$ is the operator defined
in~\eqref{def:Outer:DiffOperator}.

Moreover, one can choose $\CCC$ to satisfy
\[
 \CCC(-w,-\xi)=-\CCC(w,\xi),
\]
and therefore, by \eqref{def:Symmetry:GeneratingFunction}, one has that
\[
 \Delta(-w,-\xi)=\Delta(w,\xi).
\]
\end{theorem}

To prove this theorem we need a change of variables
$v=w+\CCC(w,\xi)$ such that $\wt \Delta\in\mathrm{Ker}\wt\LL$ if and
only if $\Delta\in \mathrm{Ker}\LL$. This fact
is equivalent to look for a Fourier series $\CCC$ which is solution of
the equation
\begin{equation}\label{eq:Diff:Canvi}
\LL(\CCC)= \KK(\CCC)
\end{equation}
where
\[
\KK(h)(w,\xi)= \left.\frac{A(v,\xi)-B(v,\xi)}{1+B(v,\xi)}\right|_{v=w+h(w,\xi)}
\]
and $A$ and $B$ are given in \eqref{def:DiffOperator:AB}.

We devote the rest of this section to obtain a solution of this
equation.  Let us point out that  the fraction in the right hand side, since it involves formal Fourier series,
is understood as
\[
 \frac{1}{1+B(v,\xi)}=\sum_{\ell\geq 0}(-B(v,\xi))^\ell.
\]
The solution of equation~\eqref{eq:Diff:Canvi} is found through a
fixed point argument.  As we will see in
Lemma~\ref{lemma:diff:AandB}, the functions $A$ and $B$ inherit the
symmetry properties of the manifolds $T^u$ and $T^s$ in
\eqref{def:Symmetry:GeneratingFunction}. To obtain a symmetric
solution~$\CCC$ of the equation~\eqref{eq:Diff:Canvi}, we need a
suitable  left inverse of the operator~$\LL$ in the domain
$D_{\kk,\de}$, which is a slight modification of the operator
$\wt\GG$ in~\eqref{def:Outer:operadorG} acting on the Fourier
coefficients and is defined as
\begin{equation}\label{def:canvi:operadorG}
\wt\GG(h)(v,\xi)=\sum_{\ell\in\ZZ}\wt\GG(h)^{[\ell]}(v)e^{i\ell\xi},
\end{equation}
where its Fourier coefficients are given by
\begin{align*}
\dps\wt\GG(h)^{[\ell]}(v)&= \int_{v_2}^v e^{i\ell G_0^3
(v-t)}h^{[\ell]}(t)\,dt& \text{ for }\ell< 0\\
\dps\wt\GG(h)^{[0]}(v)&=\int_{v^\ast}^v h^{[0]}(t)\,dt-\frac{1}{2}\int_{v^\ast}^{v_2} h^{[0]}(t)\,dt-\frac{1}{2}\int_{v^\ast}^{\ol v_2} h^{[0]}(t)\,dt&
\\
\dps\wt\GG (h)^{[\ell]}(v)&=\int_{\ol v_2}^v e^{i\ell G_0^3
(v-t)}h^{[\ell]}(t)\,dt& \text{ for }\ell>0.\\
\end{align*}
Here  $v_2=i(1/3-\kk G_0^{-3})$ is the  top vertex of the
domain $D_{\kk,\de}$, $\ol v_2$ is its conjugate, which corresponds to the bottom vertex of the domain  $D_{\kk,\de}$ and  $v^\ast$ is the left endpoint
of $D_{\kk,\de}\cap\RR$.

\begin{lemma}\label{lemma:diff:OperadorG}
The operator $\wt\GG$  in~\eqref{def:canvi:operadorG} satisfies that
if $h\in \QQQ_{\nu,\kk,\de,\sigma}$ for some $\nu\in (0,1)$, then
$\wt\GG(h)\in \QQQ_{0,\kk,\de,\sigma}$  and
\[
\left\|\wt \GG(h)\right\|_{0,\sigma}\leq K\|h\|_{\nu,\sigma}.
\]
Moreover,
\begin{itemize}
\item If $h$ is a real-analytic Fourier series, that is
\[
 h^{[\ell]}(\ol v)=\ol{h^{[-\ell]}(v)},
\]
then so is $\wt \GG(h)$.
\item If $h$ satisfies $h(-v,-\xi)=h(v,\xi)$, one has that
\[
 \wt\GG(h)(-v,-\xi)=-\wt\GG(h)(v,\xi).
\]
\end{itemize}
\end{lemma}
\begin{proof}
The first part is proven as Lemma~8.3 of~\cite{GuardiaOS10}. The
real-analyticity property is straightforward. To prove the symmetry
property  the first observation is that, if $h$ is even and taking
into account that $\ol v_2=-v_2$, one has that for $\ell\neq 0$,
\[
  \wt\GG(h)^{[-\ell]}(-v)=-\wt\GG(h)^{[\ell]}(v).
\]
Therefore, one just needs to see that $  \wt\GG(h)^{[0]}(v)$ is an odd function. To this end, it is enough to check that
\[
 f(v)= \wt\GG(h)^{[0]}(-v)+\wt\GG(h)^{[0]}(v)
\]
satisfies, wherever it makes sense, $f'(v)=0$ and $f(v_2)=0$.
\end{proof}

We look for a fixed point of the operator
\[ \wt \KK=\wt\GG\circ\KK
\]
in the space $\QQQ_{0,\kk,\de,\sigma}$.
Theorem~\ref{thm:StraighteningOperator} is a straightforward
consequence of next proposition.

\begin{proposition}\label{prop:diff:FixedPoint}
Fix $\sigma_3<\sigma_2$, $\kk_3>\kk_2$ and $\de_3>\de_2$. Then,
there exists a constant $b_6>0$ such that for $G_0$ big enough the
operator $\wt\KK$ is well defined from $B(b_6\mu
G_0^{-4})\subset\QQQ_{0,\kk_3,\de_3,\sigma_3}$ to itself and it is
contractive. Therefore, it has a unique fixed point~$\CCC$ in this
ball, which  is a real-analytic formal Fourier series and satisfies
the symmetry condition $\CCC(-v,-\xi)=-\CCC(v,\xi)$.
\end{proposition}

To prove the proposition, we first state the following technical
lemma, whose proof follows easily from Theorem
~\ref{th:SummaryManifolds} and the properties of the homoclinic
parameterization given in Corollary~\ref{coro:HomoSing} and \eqref{def:SimetriaHomo}.
\begin{lemma}\label{lemma:diff:AandB}
The functions $A$ and $B$
defined in~\eqref{def:DiffOperator:AB} satisfy that $A\in \QQQ_{1/2,\kk_2,\de_2,\sigma_2}$ and
$B\in \QQQ_{1/2,\kk_2,\de_2,\sigma_2}$ and the symmetry properties
\[
A(-v,-\xi)=A(v,\xi) \,\,\,\text{ and }\,\,\,B(-v,-\xi)=B(v,\xi).
\]
Moreover,
\[
 \begin{split}
\|A\|_{1/2,\sigma_2}&\leq K\mu G_0^{-4}\\
\|B\|_{1/2,\sigma_2}&\leq K\mu G_0^{-4}.
 \end{split}
\]
\end{lemma}

\begin{proof}[Proof of Proposition~\ref{prop:diff:FixedPoint}]
First we see that if we fix any $K>0$, then for $G_0$ big enough,
the operator $\wt\KK$ is well defined from $B(K\mu
G_0^{-4})\subset\QQQ_{0,\kk_1,\de,\sigma_1}$ to
$\QQQ_{0,\kk_3,\de_3,\sigma_3}$. Indeed, take $h\in B(K\mu
G_0^{-4})\subset\QQQ_{0,\kk_3,\de_3,\sigma_3}$. Then, applying
Lemmas~\ref{lemma:banach:Composition} and~\ref{lemma:diff:AandB}
with $\eta = K\mu G_0^{-1}$ and taking $G_0$ large enough, one can
easily see that $\KK(h)\in \QQQ_{1/2,\kk_3,\de_3,\sigma_3}$. Then,
applying Lemma~\ref{lemma:diff:OperadorG}, we have that
$\wt\KK(h)=\wt\GG \circ\KK(h)\in \QQQ_{0,\kk_1,\de,\sigma_1}$.

Once we know that the operator is well defined, we prove that it is
contractive in a certain ball. As a first step we look for bounds
for $\wt\KK(0)$. By Lemma~\ref{lemma:diff:AandB}, one can easily see
that
\[
 \|\KK(0)\|_{1/2,\sigma_3}\leq K\mu G_0^{-4}.
\]
Then, applying Lemma~\ref{lemma:diff:OperadorG}, we have that there
exists a constant $b_6>0$ such that
\[
 \|\wt\KK(0)\|_{0,\sigma_3}\leq \frac{b_6}{2}\mu G_0^{-4}.
\]
Finally, we take $h_1,h_2\in B(b_6\mu
G_0^{-4})\subset\QQQ_{0,\kk_3,\de_3,\sigma_3}$ and we prove that  $\wt\KK$ is
contractive. To this end, we first bound $\KK(h_2)-\KK(h_1)$. Using
Lemma~\ref{lemma:banach:Composition} and the estimates obtained in
Lemma~\ref{lemma:diff:AandB}, one can easily see that
\[
 \|\KK(h_2)-\KK(h_1)\|_{1/2,\sigma_3}\leq K\mu G_0\ii \|h_2-h_1\|_{0,\sigma_3}.
\]
Therefore, applying Lemma~\ref{lemma:diff:OperadorG}, we obtain that
\[
 \|\wt\KK(h_2)-\wt\KK(h_1)\|_{0,\sigma_3}\leq K\mu G_0\ii \|h_2-h_1\|_{0,\sigma_3}.
\]
Thus, this gives the existence of a unique fixed point $\CCC$ of the
operator $\wt\KK$. Furthermore, by Lemmas~\ref{lemma:diff:OperadorG}
and~\ref{lemma:diff:AandB}, one has that $\wt \KK(h)(-v,-\xi)=-\wt
\KK(h)(v,\xi)$ if $h$ satisfies $h(-v,-\xi)= -h(v,\xi)$ and therefore the
fixed point~$\CCC$ satisfies the required symmetry property. Reasoning
analogously one also obtains real-analyticity. This completes the
proof of the proposition.
\end{proof}

\subsection{Estimates for the difference between invariant manifolds}\label{sec:diff:ComputingDiff}
We devote this section to complete the proof of
Theorem~\ref{th:SplittingViaGeneratingFunctions}, that is, to prove
that the generalized Poincar\'e function~\eqref{def:FirstOrder}
 gives the first order of the difference between the
manifolds. As a first step, we show that not necessarily convergent Fourier series  $\Psi\in
\QQQ_{0,\kk,\de,\sigma}$ such that $\Psi\in\mathrm{Ker}\LL$ define
functions for real values of the variables which have exponentially
small bounds in $G_0$.

\begin{lemma}\label{lemma:Lazutkin}
Fix $\kk>0$, $\de>0$ and $\sigma>0$. Let us consider a formal Fourier series $\Psi\in \QQQ_{0,\kk,\de,\sigma}$ such that $\Psi\in\mathrm{Ker}\LL$. Then, the Fourier series $\Psi(v,\xi)$
\begin{itemize}
\item is of the form
\[
 \Psi(v,\xi)=\sum_{\ell\in\ZZ}\Psi^{[\ell]}(v)e^{i\ell\xi}=\sum_{\ell\in\ZZ}\Lambda^{[\ell]}e^{i\ell\left(G_0^3v+\xi\right)}
\]
for certain constants $\Lambda^{[\ell]}\in\CC$ and
\item  defines a function for $v\in D_{\kk,\de}\cap \RR$ and $\xi\in\TT$, which satisfies that
\[
\begin{split}
 \left|\Psi^{[\ell]}(v)\right|&\leq
 \sup_{u\in D_{\kk,\de}\cap D^u_{\rr,\kk,\de}}\left|\Psi^{[\ell]}(u)\right|K^{|\ell|} e^{-\dps\tfrac{|\ell|G_0^{3}}{3}}\\
 \left|\pa_v\Psi^{[\ell]}(v)\right|&\leq  \sup_{u\in D_{\kk,\de}\cap D^u_{\rr,\kk,\de}}\left|\Psi^{[\ell]}(u)\right|K^{|\ell|}G_0^3e^{-\dps\tfrac{|\ell|G_0^{3}}{3}}\\
 \left|\pa^2_v\Psi^{[\ell]}(v)\right|&\leq  \sup_{u\in D_{\kk,\de}\cap D^u_{\rr,\kk,\de}}\left|\Psi^{[\ell]}(u)\right|K^{|\ell|}G_0^6e^{-\dps\tfrac{|\ell|G_0^{3}}{3}}.
\end{split}
\]
\end{itemize}
\end{lemma}
\begin{proof}
The fact that $\Psi\in\mathrm{Ker}\LL$ implies that each Fourier
coefficient $\Psi^{[\ell]}$ satisfies
\[
 \pa_v \Psi^{[\ell]}(v)-i\ell G_0^3\Psi^{[\ell]}(v)=0
\]
and therefore, there exist constants $\Lambda^{[\ell]}$ such that
\[
 \Psi^{[\ell]}(v)=\Lambda^{[\ell]}e^{i\ell G_0^3v}.
\]

Evaluating this equality at the top vertex $v_2=i(1/3-\kk G_0^{-3})$ of $D_{\kk,\de}$ for $\ell<
0$ and at the bottom vertex $\ol v_2=-i(1/3-\kk G_0^{-3})$ for $\ell>0$, we obtain that
\[
 \left|\Lambda^{[\ell]}\right|\leq K^{|\ell|}
 \max\left\{\left|\Psi^{[\ell]}(v_2)\right|,\left|\Psi^{[\ell]}(\bar v_2)\right|\right\}e^{-|\ell|\dps \tfrac{G_0^3}{3}}.
\]
This  implies that for $v\in D_{\kk,\de}\cap \RR$
\[
\left|\Psi^{[\ell]}(v)\right|= \left|\Lambda^{[\ell]}\right|\leq   \sup_{u\in D_{\kk,\de}\cap D^u_{\rr,\kk,\de}}\left|\Psi^{[\ell]}(u)\right| K^{|\ell|}e^{-|\ell|\dps \tfrac{G_0^3}{3}}.
\]
Moreover, since $\Psi\in\QQQ_{0,\kk,\de,\sigma}$, we have that
\[
\left|\Psi^{[\ell]}(v)\right|\leq \|\Psi\|_{0,\sigma}\left(KG_0^{3}\right)^{\dps\tfrac{|\ell|}{2}}e^{-|\ell|\dps \tfrac{G_0^3}{3}}.
\]
From this bound one has that, for $v\in D_{\kk,\de}\cap \RR$ and $\xi\in\TT$,  the Fourier series of $\Psi$ is convergent and therefore it defines an analytic function. Finally, taking derivatives one can easily prove the bounds for $\pa_v\Psi^{[\ell]}$ and $\pa_v^2\Psi^{[\ell]}$.
\end{proof}

We use Lemma~\ref{lemma:Lazutkin} to prove
that~\eqref{def:FirstOrder} gives the first order of the difference
between the manifolds. First let us observe that, since by
Theorem~\ref{thm:StraighteningOperator}, the Fourier series $\Delta$
in~\eqref{def:difference:reparameterized} is symmetric,
real-analytic and satisfies $\Delta\in\mathrm{Ker}\LL$, it is of the
form
\begin{equation}\label{def:DeltaCosinus}
 \Delta(w,\xi)=\Gamma^{[0]}+2\sum_{\ell\in\ZZ}\Gamma^{[\ell]}\cos \ell\left(G_0^3w+\xi\right).
\end{equation}
for certain coefficients $\Gamma^{[\ell]}\in\RR$.

A direct application of Lemma~\ref{lemma:Lazutkin}, provides
exponentially small bounds for $\Delta$. Nevertheless, our goal is
not give bounds but to prove that the function $L$ in
\eqref{def:FirstOrder} is the main term in $\Delta$.  Thus, we
define the formal Fourier series
\[ \EE(w,\xi)=\Delta(w,\xi)-L(w,\xi),
\]
Note that, by~\eqref{def:whU1}, $L$ can be expressed as a formal
Fourier series with coefficients in the domain $D_{\kk,\de}\cap
D^u_{\rr,\kk,\de}$ as
\begin{equation}\label{def:FirstOrder:2}
L=\left(Q_0^s-Q_0^u\right)-\left(L_1^s-L_1^u\right),
\end{equation}
where $Q_0^u$ and $L_1^u$ are the Fourier series defined
in~\eqref{def:Q0} and~\eqref{def:HalfMelnikov:unst}, respectively,
and~$Q_0^s$ and~$L_1^s$ are the analogous Fourier series referred to
the stable manifold.

\begin{lemma}\label{lemma:ExpSmallBounds}
Consider the constants $\kk_3$ and $\delta_3$ defined in
Theorem~\ref{thm:StraighteningOperator}. Then, for
$(w,\xi)\in(D_{\kk_3,\de_3}\cap\RR)\times\TT$, the Fourier series $\EE$  defines an analytic function, which satisfies
\[
 \left|\EE(w,\xi)-E\right|\leq K \mu^2 \left(1-2\mu\right)G_0^{-2}e^{-\dps\tfrac{G_0^{3}}{3}}+K G_0^{-1/2}\mu^2 e^{-\dps\tfrac{2G_0^{3}}{3}},
\]
where $E\in\RR$ is a constant, and for $0<m+n\leq 2$,
\[
 \left|\pa_v^m\pa_\xi^n\EE(w,\xi)\right|\leq K \mu^2 \left(1-2\mu\right)G_0^{-2+3m}e^{-\dps\tfrac{G_0^{3}}{3}}+K G_0^{-1/2+3m}\mu^2 e^{-\dps\tfrac{2G_0^{3}}{3}}.
\]
\end{lemma}
\begin{proof}
To proof the lemma it is enough to point out that
$\EE\in\mathrm{Ker}\LL$ and then apply Lemma~\ref{lemma:Lazutkin}.
We use the expression of $L$ given in~\eqref{def:FirstOrder:2}. By
Theorem~\ref{thm:StraighteningOperator}, we have that
$\Delta\in\mathrm{Ker}\LL$. Moreover, $\LL( Q_0^{\ast})=\wh U_0$ for
both $\ast=u$ and $\ast=s$ and therefore $\LL(Q_0^s-Q_0^u)=0$.
Reasoning analogously, $\LL(L_1^s-L_1^u)=0$. Then, to apply
Lemma~\ref{lemma:Lazutkin} it only suffices to bound
$\|\EE\|_{0,\sigma_3}$ in the domain $D_{\kk_3,\de_3}\cap D^u_{\rr_2,\kk_3,\de_3}$. To this end, we split $\EE$ as
\[
 \EE=\EE_1^s-\EE_1^u+\EE_2
\]
where
\[
 \begin{split}
\EE_1^\ast&= T_1^\ast-Q^\ast_0-L_1^\ast\\
\EE_2&=\Delta-\wt \Delta.
 \end{split}
\]
We start by bounding $\|\EE_1^u\|_{1,\sigma_3}$. By the definition
of $Q^u$ in~\eqref{def:Q}, we have that $\EE_1^u=Q^u-L_1^u$. Then,
applying formula~\eqref{def:outer:uns:ApproxHalfMelnikov}, we obtain
\[
 \|\EE_1^u\|_{1,\sigma_3}\leq K\mu^2 G_0^{-8}.
\]
Then, applying Lemma~\ref{lemma:banach:AlgebraProps}, we obtain that
\[
 \|\EE_1^u\|_{0,\sigma_3}\leq K\mu^2 G_0^{-5}.
\]
The bound for $\EE_1^s$ is analogous. To bound $\EE_2$, recall that
by the definition of $\Delta$
in~\eqref{def:difference:reparameterized}, it can be written as
\[
 \EE_2(w,\xi)=\wt \Delta(w+\CCC(w,\xi),\xi)-\wt \Delta(w,\xi).
\]
From Lemma~\ref{lemma:outer:boundsU1andQ0}, to bound $Q_0^*$,
Theorem~\ref{thm:outer:ExistenceManifolds}, to bound $Q^*$, and
Lemma~\ref{lemma:banach:AlgebraProps}, we know that $\pa_v \wt
\Delta\in \QQQ_{3/2,\kk_3,\de_3,\sigma_3}$ and satisfies $\|\pa_v
\wt \Delta\|_{3/2,\sigma_3}\leq K\mu G_0^{-4}$. Therefore, applying
Lemma~\ref{lemma:banach:Composition} and using the estimates for
$\CCC$ given in Theorem \ref{thm:StraighteningOperator}, we obtain
that
\[
 \|\EE_2\|_{3/2,\sigma_3}\leq  \|\pa_v \wt \Delta\|_{3/2,\sigma_2}
 \|\CCC\|_{0,\sigma_3}\leq K\mu^2 G_0^{-8}.
\]
Applying Lemma~\ref{lemma:banach:AlgebraProps}, we obtain that
\[
 \|\EE_2\|_{0,\sigma_3}\leq  K\mu^2 G_0^{-7/2}
\]
and thus we can conclude that
\[
 \|\EE\|_{0,\sigma_3}\leq  K\mu^2 G_0^{-7/2},
\]
which implies
\[
 \left|\EE^{[\ell]}(v)\right|\leq  \mu^2 (KG_0)^{-\frac{7}{2}+\frac{3}{2}|\ell|}.
\]
Recall that, when $\mu=1/2$, the
Hamiltonian~\eqref{def:HamCircularRotating} is $\pi$-periodic. It
can be easily seen that this fact implies that $\EE$ is
$\pi$-periodic in $\xi$ and therefore, when $\mu=1/2$, it satisfies
that  $\EE^{[\ell]}=0$ for any odd $\ell$. On the other hand, the
function $\EE$ depends analytically on $\mu$ and therefore one can
apply the Schwarz Lemma to see that for any odd $\ell$,
\[
 \left|\EE^{[\ell]}(v)\right|\leq \mu^2 (1-2\mu)(KG_0)^{-\frac{7}{2}+\frac{3}{2}|\ell|}.
\]
Applying now Lemma~\ref{lemma:Lazutkin}, we obtain that for $w\in D_{\kk_3,\de_3}\cap\RR$,
\[
 \begin{split}
 \left|\EE^{[\ell]}(v)\right|&\leq \mu^2 (1-2\mu)(KG_0)^{-\frac{7}{2}+\frac{3}{2}|\ell|}e^{-\dps\tfrac{|\ell|G_0^{3}}{3}}\,\,\,\,\text{ for odd }|\ell|\\
  \left|\EE^{[\ell]}(v)\right|&\leq \mu^2 (KG_0)^{-\frac{7}{2}+\frac{3}{2}|\ell|}e^{-\dps\tfrac{|\ell|G_0^{3}}{3}}\,\,\,\,\qquad\quad\quad \text{ for even }|\ell|
\end{split}
\]
and analogous bounds for the derivatives of the Fourier
coefficients. Defining $E=\EE^{[0]}$ and summing up the odd
and even Fourier coefficients of $\EE$ one obtains the bounds stated
in Lemma~\ref{lemma:ExpSmallBounds}.
\end{proof}
Now it only remains to go  back to the original variables $(v,\xi)$.
This is summarized in the next lemma, from which the proof of
Theorem~\ref{th:SplittingViaGeneratingFunctions} follows. Recall that now
we are taking real values of the variables and therefore all the
objects we are dealing with are functions and not only formal Fourier series.

\begin{lemma}\label{lemma:MelnikovDomination}
Consider  the functions $\wt \Delta$ defined
in~\eqref{def:Difference} and  $L$ defined in~\eqref{def:FirstOrder}
and the constants $\kk_3$ and $\delta_3$ given by
Theorem~\ref{thm:StraighteningOperator}. Fix $\kk_4>\kk_3$ and
$\delta_4>\de_3$. Then, for $(v,\xi)\in
(D_{\kk_4,\de_4}\cap\RR)\times\TT$,
\[
 \left|\wt \Delta(v,\xi)-L(v,\xi)-E\right|\leq K \mu^2 \left(1-2\mu\right)G_0^{-2}e^{-\dps\tfrac{G_0^{3}}{3}}+K G_0^{-1/2}\mu^2 e^{-\dps\tfrac{2G_0^{3}}{3}}
\]
for certain $E\in\RR$, and for $0<m+n\leq 2$,
\[
 \left|\pa_v^m\pa_\xi^n\wt \Delta(v,\xi)-\pa_v^m\pa_\xi^nL(v,\xi)\right|\leq K \mu^2 \left(1-2\mu\right)G_0^{-2+3m}e^{-\dps\tfrac{G_0^{3}}{3}}+K G_0^{-1/2+3m}\mu^2 e^{-\dps\tfrac{2G_0^{3}}{3}}.
\]
\end{lemma}
\begin{proof}
It is enough to
consider the inverse of the change of coordinates
$(v,\xi)=(w+\CCC(w,\xi),\xi)$ obtained in Theorem
\ref{thm:StraighteningOperator}. Recall that now we are interested
in real values of the variables. Nevertheless, since we need to
apply Cauchy estimates we consider a small complex neighborhood of
$(D_{\kk_3,\de_3}\cap\RR)\times\TT$. Then,  shrinking slightly the domain, one can easily obtain the
inverse change by means of a fixed point argument. It is of the form
$(w,\xi)=(v+\Upsilon(v,\xi),\xi)$ and $\Upsilon$ satisfies
$\|\Upsilon\|_\infty\leq K\mu G_0^{-4}$. Applying Cauchy estimates
we also know that for real values of the variables
$\|\pa_w\Upsilon\|_\infty\leq K\mu G_0^{-4}$ and
$\|\pa^2_w\Upsilon\|_\infty\leq K\mu G_0^{-4}$. Applying the change
of coordinates to the bounds obtained in Lemma
\ref{lemma:ExpSmallBounds} and using the bounds obtained for
$\Upsilon$ and its derivatives, Lemma~\ref{lemma:MelnikovDomination}
easily follows.
\end{proof}

\begin{proof}[Proof of Theorem \ref{th:tangencies}]
Due to the form of  $\Delta$ in \eqref{def:DeltaCosinus}, one has that
\[
 \Delta(w,\xi)=\Gamma\left(G_0^3w+\xi\right)
\]
Moreover, it is clear that $\Gamma'(\pi k)=\Gamma'''(\pi k)=0$ for
any $k\in\ZZ$. On the other hand, by
Lemma~\ref{lemma:ExpSmallBounds} and
Proposition~\ref{prop:Melinkov}, on has that
\[
\begin{split}
 \Gamma''(z)=&2\sqrt{\pi}\mu(1-\mu)G_0^{-3/2}e^{-\frac{G_0^3}{3}}\left[\frac{1-2\mu}{4\sqrt{2}}\cos z \left(1+\OO\left(G_0^{-1/2}\right)\right)+8G_0^{2}e^{-\frac{G_0^3}{3}}\cos 2z\left(1+\OO\left(G_0^{-1/2}\right)\right)\right]\\
&+\OO\left(\mu G_0^{3/2}e^{-G_0^3}\right).
\end{split}
\]
Therefore, for any odd $k\in\ZZ$, the equation $\Gamma''(k\pi)=0$
defines a curve~$\eta$ in the parameter plane $(\mu,G_0)$, which is
of the form
\[
 \mu=\mu^\ast(G_0)=\frac{1}{2}-16\sqrt{2}G_0^2 e^{-\frac{G_0^3}{3}} \left(1+\OO\left(G_0^{1/2}\right)\right).
\]
Moreover, for $(\mu,G_0)\in\eta$, one has that
$\Gamma^{(iv)}(k\pi)\neq 0$. This implies that all partial
derivatives up to order three of the function~$\Delta$ are zero and
the fourth are not whenever $G_0^3w+\xi=k\pi$ for odd $k$.

Now, for $G_0$ big enough,  we can choose $k\in \ZZ$, odd,  such
that, for any $\xi\in\TT$, one has that   $w_k=G_0^{-3}(k\pi-\xi)\in
D_{\kk_3,\de_3}\cap\RR$ . Therefore, taking into account the
relation between $\Delta$ and $\wt\Delta$
in~\eqref{def:difference:reparameterized}, at the point
$(v_k,\xi)=(w_k+\CCC(w_k,\xi),\xi)$, all the partial derivatives up
to order three of the function $\wt \Delta$ vanish and the fourth
one does not.

To finish the proof of Theorem~\ref{th:tangencies}, it only remains
to go back to the parameterizations~\eqref{def:ParamInvManOriginal}
of the stable and unstable invariant curves $\gamma^{u,s}$ of the
Poincar\'e map~\eqref{def:PoincareMap}. They are defined by
formula~\eqref{def:InvariantManifoldsPoincare}. Then, at the point
$v=v_k$, one can easily see that
\[
\begin{split}
 \pa_v\left(Y^s_{\phi_0}(v_k;\mu^\ast(G_0),G_0)-Y^u_{\phi_0}(v_k;\mu^\ast(G_0),G_0)\right)&=0\\
\pa_v^2\left(Y^s_{\phi_0}(v_k;\mu^\ast(G_0),G_0)-Y^u_{\phi_0}(v_k;\mu^\ast(G_0),G_0)\right)&=0\\
\pa_v^3\left(Y^s_{\phi_0}(v_k;\mu^\ast(G_0),G_0)-Y^u_{\phi_0}(v_k;\mu^\ast(G_0),G_0)\right)&\neq0.
\end{split}
\]
Therefore, the invariant curves $\ga^{u,s}$ have a cubic  homoclinic tangency.

Let us emphasize that, for $(\mu,G_0)\in\eta$ and any even $k$, $\Gamma''(k\pi)\neq 0$. This implies that, at the corresponding homoclinic points, the invariant curves $\ga^{u,s}$ intersect transversally.
\end{proof}

\appendix

\section{Computation of the function L: proof of Proposition \ref{prop:Melinkov}}\label{app:proofofMelnikovintegral}
We devote this appendix to give the properties of the function $L$ in~\eqref{def:FirstOrder} stated in Proposition~\ref{prop:Melinkov}. We first state an auxiliary lemma, which gives  some properties of the potential $V$ in~\eqref{def:PerturbedPotentialscaled}.

\begin{lemma}
\label{lem:Fourier_coefficients_of_whU} The Fourier coefficients of
the function
\[
\wh U(v,\theta) = V(\wt r_\h(v),\theta;\mu,G_0)=\sum_{\ell\in\ZZ}\wh
U^{[\ell]}(v)e^{i\ell\theta},
\]
where $V$ is defined in~\eqref{def:PerturbedPotentialscaled} and $\wt r_\h$ in \eqref{def:RescalingHomo},
are
\[
\wh U^{[\ell]}(v) = \sum_{j\ge \max\{\delta_0(\ell),-\ell\}} c_j
c_{j+\ell}\frac{\mu(1-\mu)^{2j+\ell}+(-1)^\ell(1-\mu)\mu^{2j+\ell}}{G_0^{4j+2\ell}\wt
r_\h^{2j+\ell+1}(v)},
\]
where $c_j=\left(\begin{array}{c}-1/2\\j\end{array}\right)$, $\delta_0(0) = 1$ and $\delta_0(\ell) = 0$ for $\ell\neq
0$.
\end{lemma}

\begin{proof}
We use the identity
\[
\begin{aligned}
(1+A\cos \theta)^{-1/2} & = \alpha^{-1/2}(1+\beta
e^{i\theta})^{-1/2}(1+\beta e^{-i\theta})^{-1/2} \\
& = \alpha^{-1/2} \sum_{j\ge 0} \sum_{k \ge 0} c_j c_{k} \beta^{j+k} e^{i(k-j)\theta}\\
& = \sum_{\ell\in \ZZ} e^{i\ell\theta}\sum_{\substack{k-j = \ell\\ j,k \ge 0}}c_j c_{k} \beta^{j+k}\\
 & = \alpha^{-1/2} \sum_{\ell \in \ZZ} e^{i\ell \theta} \sum_{j\ge
\max\{0,-\ell\}} c_j c_{j+\ell} \beta^{2j+\ell},
\end{aligned}
\]
where $\alpha = A/(2\beta)$ and $\beta = (1- \sqrt{1-A^2})/A$, and
\[
\begin{aligned}
\frac{1-\mu}{\sqrt{G_0^4\wt r_\h^2-2\mu G_0^2 \wt r_\h \cos \theta +
\mu^2}} & = \frac{1-\mu}{G_0^2\wt r_\h}\left( 1-\frac{\mu}{G_0^2 \wt
r_\h}e^{i\theta}\right)^{-1/2}
\left( 1-\frac{\mu}{G_0^2 \wt r_\h}e^{-i\theta}\right)^{-1/2}\\
 \frac{\mu}{\sqrt{G_0^4\wt r_\h^2+2(1-\mu)G_0^2  \wt r_\h \cos
\theta + (1-\mu)^2}} & = \frac{\mu}{G_0^2\wt r_\h}\left(
1+\frac{1-\mu}{G_0^2 \wt r_\h}e^{i\theta}\right)^{-1/2} \left(
1+\frac{1-\mu}{G_0^2 \wt r_\h}e^{-i\theta}\right)^{-1/2}
\end{aligned}
\]
to obtain
\[
\begin{aligned}
\frac{1}{G_0^2}V(\wt r_\h(v),\theta) = & \frac{1-\mu}{\sqrt{G_0^4\wt
r_\h^2-2\mu G_0^2  \wt r_\h  \cos \theta + \mu^2}}+
\frac{\mu}{\sqrt{G_0^4\wt r_\h^2+2(1-\mu)G_0^2  \wt r_\h \cos
\theta + (1-\mu)^2}}\\
& - \frac{1}{G_0^2 \wt r_\h(v)}\\
= & \frac{1-\mu}{G_0^2 \wt r_\h(v)}\sum_{\ell \in \ZZ} e^{i\ell
\theta} \sum_{j\ge \max\{0,-\ell\}} c_j c_{j+\ell}
\frac{(-\mu)^{2j+\ell}}{G_0^{4j+2\ell}\wt r_\h^{2j+\ell}(v)} \\
& + \frac{\mu}{G_0^2 \wt r_\h(v)}\sum_{\ell \in \ZZ} e^{i\ell \theta}
\sum_{j\ge \max\{0,-\ell\}} c_j c_{j+\ell}
\frac{(1-\mu)^{2j+\ell}}{G_0^{4j+2\ell}\wt r_\h^{2j+\ell}(v)} -
\frac{1}{G_0^2
\wt r_\h(v)} \\
= & \sum_{\ell \in \ZZ} e^{i\ell \theta} \sum_{j\ge \max\{0,-\ell\}}
c_j c_{j+\ell}
\frac{(-1)^\ell(1-\mu)\mu^{2j+\ell}+\mu(1-\mu)^{2j+\ell}}{G_0^{4j+2\ell+2}\wt
r_\h^{2j+\ell+1}(v)} - \frac{1}{G_0^2 \wt r_\h(v)}.
\end{aligned}
\]
\end{proof}

Now we prove Proposition  \ref{prop:Melinkov}.

\begin{proof}[Proof Proposition  \ref{prop:Melinkov}]
First, we observe that the Poincar\'{e} function~$L$ in
\eqref{def:FirstOrder} can be written as
\[
L(v,\xi;\mu,G_0)=\int_{-\infty}^{+\infty} V(\wt r_\h(t),\xi+G_0^3
v-G_0^3 t+ \wt \al_\h(t);\mu,G_0)dt.
\]
Now, using that
\[
V(\wt r_\h(t), \xi+G_0^3 v-G_0^3 t+\wt \alpha_\h(t);\mu,G_0) = \sum
_{\ell \in \ZZ}\wh  U^{[\ell]}(t) e^{i\ell \wt\alpha_\h(t) } e^{-i\ell G_0^3
t} e^{i\ell (\xi + G_0^3 v)}
\]
and therefore
\[
L(v,\xi;\mu,G_0)= \sum L^{[\ell]} e^{i\ell (\xi + G_0^3 v)}
\]
with
\[
L^{[\ell]}= \int _{-\infty}^{+\infty}\wh U^{[\ell]} (t) e^{i\ell \wt\alpha_\h(t)}
e^{-i\ell G_0^3 t} dt
\]
to compute $L^{[\ell]}$ we use the expansions in Lemma~\ref{lem:Fourier_coefficients_of_whU} obtaining,
when $\ell \ne 0$
\[L^{[\ell]} = \sum _{j\ge 0} \left(\begin{array}{c}-\frac{1}{2}\\ j
\end{array}\right) \left(\begin{array}{c}-\frac{1}{2}\\ \ell +j
\end{array}\right) \frac{\mu (1-\mu)^{2j+\ell}+ (-1)^\ell (1-\mu)\mu
^{2j+\ell}}{G_0^{4j+2\ell}} \II (\ell,j)
\]
with
\[
\II(\ell,j)=\int _{-\infty}^{+\infty}\frac{e^{i\ell \wt \alpha_\h (t)}}{\wt
r_\h ^{2j+\ell+1}(t)} e^{-i\ell G_0^3 t}.
\]
To compute $\II(\ell,j)$ we use the change of variables $v =
\frac{1}{2}\left(\frac{1}{3} \tau^3 + \tau \right)$ given in
Lemma~\ref{lem:homoclinicexplicitexpressions} obtaining
\[
\II(\ell,j)=(-1) ^\ell 2^{2j+\ell}\int
_{-\infty}^{+\infty}\frac{e^{-i\ell \frac{G_0^3}{2} (\tau
+\frac{\tau ^3}{3}) }} {(\tau -i)^{2j} (\tau +i)^{2j+2\ell} }
d\tau:= (-1) ^\ell 2^{2j+\ell} I(-\ell, j,j+\ell),
\]
where we have introduced the notation
\[
I(\ell,m,n)=\int _{-\infty}^{+\infty}\frac{e^{i\ell \frac{G_0^3}{2}
(\tau +\frac{\tau ^3}{3}) }} { (\tau -i)^{2m} (\tau +i)^{2n}} d\tau
\]
to write the Poincar\'{e} function Fourier coefficients as
\[
\begin{split}
L^{[\ell]} &= \sum _{j\ge 0} \left(\begin{array}{c}-\frac{1}{2}\\ j
\end{array}\right) \left(\begin{array}{c}-\frac{1}{2}\\ \ell +j
\end{array}\right)
\frac{\mu (1-\mu)^{2j+\ell}+ (-1)^\ell (1-\mu)\mu ^{2j+\ell}}{G_0^{4j+2\ell}} (-1) ^\ell 2^{2j+\ell} I(-\ell, j,j+\ell) .\\
\end{split}
\]
The first observation is that
\[
I(-\ell, n,m) = I(\ell,m,n)=\ol{ I(\ell,m,n)}.
\]
Therefore all the Fourier coefficients are real and we just need to compute them for $\ell >0$:
\[
\begin{split}
L^{[\ell]} &= \sum _{j\ge 0} \left(\begin{array}{c}-\frac{1}{2}\\ j
\end{array}\right) \left(\begin{array}{c}-\frac{1}{2}\\ \ell +j
\end{array}\right)
\frac{\mu (1-\mu)^{2j+\ell}+ (-1)^\ell (1-\mu)\mu ^{2j+\ell}}{G_0^{4j+2\ell}} (-1) ^\ell 2^{2j+\ell} I(\ell, j+\ell,j) \\
\end{split}
\]
and then
\[
L(\xi, v) =2 \sum_{\ell\in\NN} L^{[\ell]} \cos \ell (\xi + G_0^3 v).
\]

To compute the integrals $I(\ell,m,n)$  for $\ell >0$, one uses the method  in \cite{Ederlyi56} (see also \cite{SimoL80, MartinezP94}) changing the path of integration to a suitable complex path
$\Re (\tau + \frac{\tau ^3}{3})= 0$ up to a neighborhood of the singularity $\tau = i$.

Using that $\tau + \frac{\tau ^3}{3}=\frac{2}{3}i +\OO((\tau -i)^2)$, to bound the integrals (see~\cite{DelshamsKRS12}) it is enough to reach a neighborhood of the singularity  $\tau =i$ of order $\OO(G_0^{-3/2})$, obtaining that there exists a constant $K>0$ such that, for any $\ell \in \ZZ$ and
$m,n\geq 1$:
\[
|I(\ell, m,n)| \le K G_0^{3m-
3/2} e^{-\frac{G_0^3}{3}\ell},
\]
and therefore
\[
|L^{[\ell]}| \le (K G_0)^{\ell-
3/2} e^{-\frac{G_0^3}{3}\ell}.
\]
To obtain the dominant terms of the function $L$, which correspond
to $\ell =1,2$, we can use the results in~\cite{DelshamsKRS12} (see
also~\cite{MartinezP94}) to obtain
\[
\begin{split}
I(1,2,1)&= \frac{1}{6} \sqrt{\frac{\pi}{2}}
G_0^{\frac{9}{2}}e^{-\frac{G_0^3}{3}} \left(1+\OO\left(G_0^{-3/2}\right)\right)\\
I(2,2,0)&= 2 \sqrt{\pi} G_0^{\frac{9}{2}}e^{-2\frac{G_0^3}{3}}
\left(1+\OO\left(G_0^{-3/2}\right)\right).
\end{split}
\]
Thus,
\[
\begin{split}
L^{[1]}&= -\frac{\mu(1-\mu)^3-(1-\mu)\mu^3}{4} \sqrt{\frac{\pi}{2}}
G_0^{-\frac{3}{2}}e^{-\frac{G_0^3}{3}} \left(1+\OO\left(G_0^{-3/2}\right)\right)\\
L^{[2]}&= 2 \mu (1-\mu)\sqrt{\pi}
G_0^{\frac{1}{2}}e^{-2\frac{G_0^3}{3}}
\left(1+\OO\left(G_0^{-3/2}\right)\right).
\end{split}
\]
\end{proof}

\section*{Acknowledgements}
The authors acknowledge useful discussions with V. Kaloshin and A.
Gorodetski. They have been partially supported by the Spanish
MCyT/FEDER grant MTM2009-06973 and the Catalan SGR grant 2009SGR859.
M. G. and P. M warmly thank the Institute for Advanced Study for
their hospitality, stimulating atmosphere and support. During his stay in the Institute for Advanced Study, M. G. was also partially supported by the NSF grant DMS-0635607.

\bibliography{references}

\def\cprime{$'$} \def\cprime{$'$}
\begin{thebibliography}{DKdlRS12}

\bibitem[AKN88]{ArnoldKN88}
V.I. Arnold, V.V. Kozlov, and A.I. Neishtadt.
\newblock {\em Dynamical Systems {I}{I}{I}}, volume~3 of {\em Encyclopaedia
  Math. Sci.}
\newblock Springer, Berlin, 1988.

\bibitem[BF04]{BaldomaF04}
I.~Baldom{\'a} and E.~Fontich.
\newblock Exponentially small splitting of invariant manifolds of parabolic
  points.
\newblock {\em Mem. Amer. Math. Soc.}, 167(792):x--83, 2004.

\bibitem[BFGS11]{BaldomaFGS11}
I.~Baldom{\'a}, E.~Fontich, M.~Gu\`ardia, and T.~M. Seara.
\newblock Exponentially small splitting of separatrices beyond melnikov
  analysis: rigorous results.
\newblock {\em Preprint}, 2011.

\bibitem[DKdlRS12]{DelshamsKRS12}
A.~Delshams, V.~Kaloshin, A.~de~la Rosa, and T.~Seara.
\newblock Parabolic orbits in the restricted three body problem.
\newblock Preprint, 2012.

\bibitem[DS92]{DelshamsS92}
A.~Delshams and T.~M. Seara.
\newblock An asymptotic expression for the splitting of separatrices of the
  rapidly forced pendulum.
\newblock {\em Comm. Math. Phys.}, 150(3):433--463, 1992.

\bibitem[DS97]{DelshamsS97}
A.~Delshams and T.M. Seara.
\newblock Splitting of separatrices in {H}amiltonian systems with one and a
  half degrees of freedom.
\newblock {\em Math. Phys. Electron. J.}, 3:Paper 4, 40 pp. (electronic), 1997.

\bibitem[Erd56]{Ederlyi56}
A.~Erd{\'e}lyi.
\newblock {\em Asymptotic expansions}.
\newblock Dover Publications Inc., New York, 1956.

\bibitem[Gel97]{Gelfreich97a}
V.~G. Gelfreich.
\newblock Melnikov method and exponentially small splitting of separatrices.
\newblock {\em Phys. D}, 101(3-4):227--248, 1997.

\bibitem[Gel00]{Gelfreich00}
V.~G. Gelfreich.
\newblock Separatrix splitting for a high-frequency perturbation of the
  pendulum.
\newblock {\em Russ. J. Math. Phys.}, 7(1):48--71, 2000.

\bibitem[GK10a]{GalanteK10c}
J.~Galante and V~Kaloshin.
\newblock Destruction of invariant curves using the ordering condition.
\newblock Preprint, available at
  \url{http://www.terpconnect.umd.edu/~vkaloshi}, 2010.

\bibitem[GK10b]{GalanteK10b}
J.~Galante and V~Kaloshin.
\newblock The method of spreading cumulative twist and its application to the
  restricted circular planar three body problem.
\newblock Preprint, available at
  \url{http://www.terpconnect.umd.edu/~vkaloshi}, 2010.

\bibitem[GK11]{GalanteK11}
J.~Galante and V.~Kaloshin.
\newblock Destruction of invariant curves in the restricted circular planar
  three-body problem by using comparison of action.
\newblock {\em Duke Math. J.}, 159(2):275--327, 2011.

\bibitem[GK12]{GorodetskiK12}
A.~Gorodetski and V~Kaloshin.
\newblock Hausdorff dimension of oscillatory motions for restricted three body
  problems.
\newblock Preprint, available at
  \url{http://www.terpconnect.umd.edu/~vkaloshi}, 2012.

\bibitem[GOS10]{GuardiaOS10}
M.~Guardia, C.~Oliv\'e, and T.~Seara.
\newblock Exponentially small splitting for the pendulum: a classical problem
  revisited.
\newblock {\em Journal of Nonlinear Science}, 20(5):595--685, 2010.

\bibitem[Gua12]{Guardia12}
M.~Guardia.
\newblock Splitting of separatrices in the resonances of nearly integrable
  hamiltonian systems of one and a half degrees of freedom.
\newblock Preprint available at \url{http://arxiv.org/abs/1107.6042}, to appear
  in Discrete and Continous Dynamical Systems A, 2012.

\bibitem[HMS88]{HolmesMS88}
P.~Holmes, J.~Marsden, and J.~Scheurle.
\newblock Exponentially small splittings of separatrices with applications to
  {KAM} theory and degenerate bifurcations.
\newblock In {\em Hamiltonian dynamical systems}, volume~81 of {\em Contemp.
  Math.} 1988.

\bibitem[LMS03]{LochakMS03}
P.~Lochak, J.-P. Marco, and D.~Sauzin.
\newblock On the splitting of invariant manifolds in multidimensional
  near-integrable {H}amiltonian systems.
\newblock {\em Mem. Amer. Math. Soc.}, 163(775):viii+145, 2003.

\bibitem[LS80a]{SimoL80}
J.~Llibre and C.~Sim{\'o}.
\newblock Oscillatory solutions in the planar restricted three-body problem.
\newblock {\em Math. Ann.}, 248(2):153--184, 1980.

\bibitem[LS80b]{LlibreS80}
J.~Llibre and C.~Sim{\'o}.
\newblock Some homoclinic phenomena in the three-body problem.
\newblock {\em J. Differential Equations}, 37(3):444--465, 1980.

\bibitem[McG73]{McGehee73}
R.~McGehee.
\newblock A stable manifold theorem for degenerate fixed points with
  applications to celestial mechanics.
\newblock {\em J. Differential Equations}, 14:70--88, 1973.

\bibitem[Mel63]{Melnikov63}
V.~K. Melnikov.
\newblock On the stability of the center for time periodic perturbations.
\newblock {\em Trans. Moscow Math. Soc.}, 12:1--57, 1963.

\bibitem[Mos73]{Moser01}
J.~Moser.
\newblock {\em Stable and random motions in dynamical systems}.
\newblock Princeton University Press, Princeton, N. J., 1973.
\newblock With special emphasis on celestial mechanics, Hermann Weyl Lectures,
  the Institute for Advanced Study, Princeton, N. J, Annals of Mathematics
  Studies, No. 77.

\bibitem[MP94]{MartinezP94}
R.~Mart{\'{\i}}nez and C.~Pinyol.
\newblock Parabolic orbits in the elliptic restricted three body problem.
\newblock {\em J. Differential Equations}, 111(2):299--339, 1994.

\bibitem[Ne{\u\i}84]{Neishtadt84}
A.~I. Ne{\u\i}shtadt.
\newblock The separation of motions in systems with rapidly rotating phase.
\newblock {\em Prikl. Mat. Mekh.}, 48(2):197--204, 1984.

\bibitem[Poi90]{Poincare90}
H.~Poincar{\'e}.
\newblock Sur le probl\`eme des trois corps et les \'equations de la dynamique.
\newblock {\em Acta Mathematica}, 13:1--270, 1890.

\bibitem[Sau01]{Sauzin01}
D.~Sauzin.
\newblock A new method for measuring the splitting of invariant manifolds.
\newblock {\em Ann. Sci. \'Ecole Norm. Sup. (4)}, 34, 2001.

\bibitem[Sit60]{Sitnikov60}
K.~Sitnikov.
\newblock The existence of oscillatory motions in the three-body problems.
\newblock {\em Soviet Physics. Dokl.}, 5:647--650, 1960.

\bibitem[Tre97]{Treshev97}
D.~Treschev.
\newblock Separatrix splitting for a pendulum with rapidly oscillating
  suspension point.
\newblock {\em Russ. J. Math. Phys.}, 5(1):63--98, 1997.

\bibitem[Xia92]{Xia92}
Z.~Xia.
\newblock Mel\cprime nikov method and transversal homoclinic points in the
  restricted three-body problem.
\newblock {\em J. Differential Equations}, 96(1):170--184, 1992.

\end{thebibliography}
\bibliographystyle{alpha}
\end{document}